\documentclass[11pt]{amsart}
\usepackage[utf8]{inputenc}
\usepackage{amsmath,amssymb,amsthm}
\usepackage[margin=1.15in]{geometry}
\usepackage{xcolor}
\usepackage{tikz}
\usetikzlibrary{decorations.pathreplacing}
\usepackage[colorlinks=true,citecolor=blue,linkcolor=blue,urlcolor=blue]{hyperref}

\newcommand{\RR}{\mathbb R}
\newcommand{\ZZ}{\mathbb Z}
\newcommand{\HH}{\mathbb H}
\newcommand{\g}{\mathfrak g}
\newcommand{\m}{E^m}
\newcommand{\eps}{\varepsilon}
\newcommand{\tr}{\operatorname{tr}}
\newcommand{\Ftr}{F^{\mathrm{tr}}}
\newcommand{\bF}{\mathbf F}
\newcommand{\bW}{\mathbf W}
\newcommand{\bY}{\mathbf Y}
\newcommand{\bB}{\mathbf B}

\newtheorem{thm}{Theorem}[section]
\newtheorem{prop}[thm]{Proposition}
\newtheorem{lemma}[thm]{Lemma}
\newtheorem{cor}[thm]{Corollary}
\newtheorem{mainthm}{Theorem}

\theoremstyle{definition}
\newtheorem{definition}[thm]{Definition}
\theoremstyle{remark}
\newtheorem{rem}[thm]{Remark}

\numberwithin{equation}{section}

\begin{document}

\title[Exponential mixing for equilibrium states of time changes]{Exponential mixing for equilibrium states of time changes of hyperbolic geodesic flows}

\author[C.~Khoule]{Cheikh Khoule}
\address{CERER, Universit\'e Cheikh Anta Diop, Dakar, S\'en\'egal}
\email{cheikh1.khoule@ucad.edu.sn}

\author[M.~Manso]{Matheus Manso}
\address{Instituto Nacional de Matem\'atica Pura e Aplicada (IMPA), Rio de Janeiro, Brazil}
\email{matheus.manso@impa.br}

\author[A.~Ndiaye]{Ameth Ndiaye}
\address{D\'epartement de Math\'ematiques, FASTEF, Universit\'e Cheikh Anta Diop, Dakar, S\'en\'egal}
\email{ameth1.ndiaye@ucad.edu.sn}

\author[K.~War]{Khadim War}
\address{School of Mathematical Sciences, Hebei Normal University, 050024 Shijiazhuang, P.R.~China}
\email{khadim@hebtu.edu.cn}

\date{}

\begin{abstract}
Let $M$ be a closed hyperbolic manifold of dimension $n\ge2$, let $\psi$ be a smooth
positive function on its unit tangent bundle $SM$, and let $f^t_\psi$ be the time change
of the geodesic flow with speed $1/\psi$. We prove that $f^t_\psi$ is exponentially mixing, for H\"older
observables, with respect to the equilibrium state of every H\"older potential, as asked by Bowen and Ruelle. This includes
the smooth invariant measure and the measure of maximal entropy.  The proof runs Dolgopyat's cancellation argument at a variable
\emph{local frequency} $\Lambda_b=|b|(1+\|\bY_b\|)$ (Proposition~\ref{prop:localfreq}): the scale $1/|b|$ is replaced by a scale that depends on
the point, as in the work of Tsujii and Zhang in dimension three, and the argument runs without a smooth contact structure and without $C^1$ strong
foliations. Many of the techniques, in particular the reduction to an expanding semiflow by quotienting
along the stable foliation, are inspired by the earlier work of Butterley and War \cite{BW20}.
\end{abstract}

\maketitle



\noindent{\bf Use of AI.} All mathematical statements, proofs and references were produced, checked and approved by the authors, who take full responsibility for the content of the paper. The authors used the generative AI assistant Claude (Anthropic) to help with the exposition and writing of this paper, including drafting and editing of the introduction and abstract.

\section{Introduction and statement of results}

Anosov flows are central examples of chaotic dynamical systems.
Let $X$ be a compact connected Riemannian manifold. A flow $f^t:X\to X$, $t\in\mathbb{R}$, is called
\emph{Anosov} if it is \emph{uniformly hyperbolic} in the sense that the tangent bundle splits into three invariant sub-bundles $\mathbb{E}^{c}$, $\mathbb{E}^s$ and $\mathbb{E}^u$, where $\mathbb{E}^{c}$ is one-dimensional and contains the flow direction and where the vectors in $\mathbb{E}^s$ (resp.\ $\mathbb{E}^u$) are exponentially contracted (resp.\@ expanded).
The sub-bundles $\mathbb{E}^s$ and $\mathbb{E}^u$ are referred to as stable and unstable bundles. 
The three main examples of Anosov flows are: (a) Suspensions over Anosov diffeomorphisms; (b) Geodesic flows on manifolds of negative curvature; (c) Small perturbations of such geodesic flows (Anosov flows are structurally stable but typically the perturbed flow will fail to be a geodesic flow).

The rate at which initial information is lost in Anosov systems can be formalized as the rate of mixing or
decay of correlation. A flow $f^t:X\to X$ is said to be exponentially
mixing with respect to a given probability measure $\mu$ if there exist $C,\gamma>0$ such that 
\begin{equation}
\left|\int_{X}A\cdot B\circ f^td\mu-\left(\int_XA d\mu\right)\cdot\left(\int_XB d\mu\right)\right|\leq C\|A\|_{C^1}\|B\|_{C^2}e^{-\gamma t}
\end{equation}
for all $A\in C^1(X), B\in C^{2}(X)$ and $t>0$.

For the geodesic flow, the Liouville measure is the measure of interest for
the mixing. In the case of topologically transitive Anosov flows, the Liouville measure generalizes naturally to a canonical measure called Sinai--Ruelle--Bowen measure (or just SRB measure, see \cite{Sinai} for exact definition). More generally, every H\"older potential $\varpi:X\to\RR$ of a transitive Anosov flow has a unique \emph{equilibrium state} $\mu_\varpi$, the invariant probability measure that maximizes $h_\mu(f^1)+\int\varpi\,d\mu$ \cite{BR75}. The potential $\varpi=0$ gives the measure of maximal entropy, and the SRB measure is the equilibrium state of minus the unstable Jacobian. Bowen and Ruelle asked whether exponential mixing holds for these measures \cite{BR75}. The rate of mixing has been studied by Sinai and Ruelle since the 1970s; various important results for flows have been established relatively recently. However, many important questions remain open.

For $n\ge2$, let $M=\Gamma\backslash\HH^n$ be a closed hyperbolic $n$-manifold, where
$\Gamma<G=SO_0(n,1)$ is cocompact and torsion free. Let $f^t$ be the geodesic flow on the unit tangent bundle  $SM$, $m$ be the normalized Liouville measure and $Z$ be the vector field generating $f^t$ (the geodesic spray, up to the normalization of time explained below). Given a $C^\infty$ function $\psi: SM\to(0,\infty)$, let $f^t_\psi$ be the flow generated by the vector field $Z_\psi:=Z/\psi$. It is a transitive Anosov flow with the same orbits as the geodesic flow, and it preserves the smooth measure $\nu=\psi\,m/\int\psi\,dm$, which is therefore its SRB measure. We say that $\psi$ is \emph{of contact type} if
\begin{equation}\label{eq:contact-type}
\psi=c+\zeta(Z)+Zw,\qquad c\in\RR,\ \zeta \text{ a closed $1$-form on } M,\ w\in C^\infty(SM).
\end{equation}
In that case $f^t_\psi$ is a smooth contact Anosov flow (Proposition~\ref{cor:contact}). Our main result is the following.

\begin{mainthm}\label{thm:gibbs}
Let $\psi\in C^\infty(SM)$ be positive, let $\varpi\in C^{\beta_\varpi}(SM)$ with $\beta_\varpi\in(0,1]$, let $\beta_0\in(0,1]$, and let $\mu_\varpi$ be the unique equilibrium state of $\varpi$ for $f^t_\psi$. There are $C,\gamma>0$ such that for all $A,B\in C^{\beta_0}(SM)$ and $t\ge0$,
\[
\Big|\int A\cdot B\circ f^t_\psi\,d\mu_\varpi-\int A\,d\mu_\varpi\int B\,d\mu_\varpi\Big|\le C\,e^{-\gamma t}\,\|A\|_{C^{\beta_0}}\|B\|_{C^{\beta_0}}.
\]
In particular $f^t_\psi$ is exponentially mixing for the smooth invariant measure $\nu$ and for the measure of maximal entropy.
\end{mainthm}

When $\psi$ is of contact type, Theorem~\ref{thm:gibbs} is Stoyanov's theorem \cite[Thm.~1.1]{Sto25}, which applies to every $C^5$ contact Anosov flow on a compact Riemannian manifold, and for $\nu$ it is due to Liverani \cite{Liv04}. The new case is the complementary one, $\psi$ not of contact type. We first treat the smooth measure, which is the case where the geometry is cleanest, and then extend to all equilibrium states by changing the function spaces at one point. The statement for $\nu$, with smooth observables, is the following.

\begin{thm}\label{thm:main}
For every $C^\infty$ function $\psi: SM\to(0,\infty)$, the Anosov flow $f^t_\psi$ mixes exponentially with respect to the measure $\nu=\psi\,m/\int\psi\,dm$, in the sense of the displayed inequality above, for $A\in C^1(SM)$ and $B\in C^2(SM)$.
\end{thm}

For $\psi$ of contact type, Theorem~\ref{thm:main} is Liverani's theorem \cite{Liv04}. Sections~\ref{sec:prelim}--\ref{sec:mixing} prove it for $\psi$ not of contact type, and the input that makes it work is the rigidity theorem below, for which we first need some notation (collected in the Notation paragraph below). Section~\ref{sec:gibbs} deduces Theorem~\ref{thm:gibbs} for $\psi$ not of contact type from the same proof: the geometry (the temporal distance, the local frequency and the specification lemma) does not involve the measure, and the measure enters only through the transfer operator, its weights and the volume of small balls.

Let $FM=\Gamma\backslash G$ be the bundle of oriented orthonormal frames of $M$, where $G=SO_0(n,1)$ is identified
with the oriented orthonormal frames of $\HH^n$ in the usual way. The group $M_0\cong SO(n-1)$ of rotations of a
frame about its first vector, which is the centralizer of the geodesic flow, acts on $FM$ on the right, and
$SM=FM/M_0$. The \emph{frame flow} on $FM$, the right translation by a one-parameter subgroup (Section~\ref{sec:prelim}), covers the geodesic flow on $SM$
(with the normalization of Section~\ref{sec:prelim} it runs the geodesic at speed $2$, which only rescales time). We write $f^t$ for both
flows and $Z$ for their generator. Functions on $SM$ are the right $M_0$-invariant functions on $FM$, and the
normalized Haar measure on $FM$, also denoted $m$, projects to the Liouville measure on $SM$.

 In Section~\ref{sec:prelim} we introduce bases $X^+_i$ and $X^-_i$
($1\le i\le n-1$) of the unstable and stable horospherical subalgebras, and the matrix
\[
\bF=(F_{ij}),\qquad F_{ij}=X^-_jX^+_i\psi ,
\]
of functions on $FM$. Under the right action of $M_0$ it transforms by conjugation, so its
operator norm, and the operator norm of any Birkhoff integral $\int_0^L\bF(f^tx)\,dt$, are functions on
$SM$. For a $1$-form $\zeta$ on $M$ we write $\zeta(Z)=(\pi^*\zeta)(Z)$, where $\pi:SM\to M$ is the
projection.

For $a\in C^\infty(FM)$ with $\int a\,dm=0$, let $D_a$ be the flow-invariant distribution
\[
D_a(\varphi)=\lim_{T\to\infty}\int_{-T}^{T}\big\langle\varphi,\,a\circ f^t\big\rangle\,dt,\qquad\varphi\in C^\infty(FM),
\]
where $\langle\cdot,\cdot\rangle$ is the $L^2(m)$ pairing. The limit exists because the frame flow is
exponentially mixing \cite{Moo87,KM96}.

\begin{mainthm}[Contact rigidity]\label{thm:rigidity}
Let $n\ge2$ and let $\psi\in C^\infty(SM)$ be positive. The following are equivalent.
\begin{enumerate}
\item[(1)] $\psi$ is of contact type \eqref{eq:contact-type}.
\item[(2)] For all $i,j$ there is $v_{ij}\in C^\infty(FM)$ with $F_{ij}=Zv_{ij}$.
\item[(3)] The Birkhoff integrals of $\bF$ are bounded:
$\displaystyle\sup_{x\in FM,\ L\ge0}\Big\|\int_0^L\bF(f^tx)\,dt\Big\|<\infty$.
\item[(4)] $D_{F_{ij}}=0$ for all $i,j$.
\end{enumerate}
If they hold, then $c=\int\psi\,dm>0$, and $\alpha_\psi=c\,\alpha+\pi^*\zeta+dw$ is a smooth contact form
with Reeb field $Z_\psi$.
\end{mainthm}

For $n=2$, $\bF$ is the scalar $F=X^-X^+\psi$, and by Liv\v{s}ic's theorem \cite{Liv72} condition (2)
says that $F$ integrates to zero along every closed geodesic. In every dimension, the weak exterior
derivative of the form $\alpha_\psi$, whose kernel is $E^s_\psi\oplus E^u_\psi$, is computed by the
distributions $D_{F_{ij}}$ (Proposition~\ref{prop:weak-d}). The implication (4)$\Rightarrow$(1) is proved
on $FM$ and then descended to $SM$.

\begin{rem}[Examples]\normalfont\label{rem:examples}
Every $\psi=c+u\circ\pi$ with $u\in C^\infty(M)$ non-constant and $c+u>0$ fails (1). Indeed, by
Lemma~\ref{lem:trace}(iv), $\tr\bF=(\Delta u)\circ\pi+Zw_u$ with $w_u$ smooth. If (2) held, then
$\tr\bF=ZV$ with $V=\sum_iv_{ii}$. Since $\tr\bF$ is $M_0$-invariant and $Z$ commutes with the right
$M_0$-action, also $\tr\bF=Z\bar V$, where $\bar V(x)=\int_{M_0}V(xk)\,dk$ is a smooth function on $SM$.
Averaging is needed, because along a closed geodesic the frame returns rotated by its
holonomy, so the integral of $ZV$ along a closed geodesic need not vanish. So $\Delta u$ would integrate to zero along every closed geodesic.
By injectivity of the geodesic X-ray transform \cite{GK80,CS98}, $\Delta u=0$, so $u$ would be
constant. Modulo coboundaries, the $\psi$ of contact type form a space of dimension
$1+b_1(M)$. By Liv\v{s}ic's theorem and \cite{LMM86}, the smooth coboundaries form a closed
subspace of $C^\infty(SM)$, so the $\psi$ of contact type form a closed subspace with empty interior,
and (1) fails for an open and dense set of positive $\psi$.
\end{rem}

\smallskip\noindent\textbf{Context.}  Dolgopyat \cite{Dol98} showed that transitive Anosov flows with $C^1$ stable and unstable foliations mix exponentially whenever the stable and unstable foliations are not jointly integrable.
Liverani \cite{Liv04} showed that all contact (with $C^2$ contact form) Anosov flows mix exponentially. This provides a complete answer for geodesic flows on manifolds of negative curvature since all such geodesic flows are contact Anosov flows with smooth contact form.
Tsujii \cite{Tsujii} proves the existence of a $C^3$-open and $C^r$-dense subset of volume-preserving three-dimensional Anosov flows which mix exponentially ($3\le r\le\infty$).
 Butterley and War \cite{BW20}  proved for the first time the existence of open sets of exponentially mixing Anosov flows. An application of our main result answers the question of mixing for reparametrization of Anosov geodesic flows, for all equilibrium states. Tsujii and Zhang \cite{TZ23} proved
that smooth mixing Anosov flows in dimension three are exponentially mixing. Their proof runs Dolgopyat's
estimate for the twisted transfer operators of a Markov section at a point-dependent scale, defined
through template functions that describe the temporal distance along the one-dimensional unstable
and stable leaves. For $n=2$ the flow $f^t_\psi$ is a mixing $C^\infty$ Anosov flow on a three-manifold, so
Theorem~\ref{thm:main} for $n=2$ is contained in \cite{TZ23}, which also treats equilibrium states of H\"older potentials. For contact Anosov flows, Stoyanov \cite{Sto25} proves exponential mixing for all equilibrium states of H\"older potentials.
For $n\ge3$ and $\psi$ not of contact type, none of these applies directly to $f^t_\psi$. The flow is not
contact, its dimension $2n-1$ is at least five, both bundles have dimension $n-1\ge2$, and its
strong stable bundle is spanned by $X^-_i+h^-_iZ_\psi$ with $X^+_jh^-_i$ formally equal to
$\int_0^\infty(X^+_jX^-_i\psi)\circ f^t\,dt$. By Theorem~\ref{thm:rigidity} this integral diverges, so we cannot expect the $C^1$ regularity of the strong
foliations that is required in \cite{Dol98}.

\smallskip\noindent\textbf{Equilibrium states.} For $n=2$ the equilibrium states of H\"older
potentials are treated in \cite{TZ23}, where the weight of the transfer operator and the roof function
are first mollified to $C^1$ functions whose $C^1$ norms grow like a small power of $|b|$, at the price of an
error in the H\"older norm.
For Theorem~\ref{thm:gibbs} we do not need this step. The cone of functions used in the proof of
Proposition~\ref{prop:localfreq} already has a H\"older modulus at the scale $1/\Lambda_b$, and the
weights of the transfer operator have bounded distortion at that scale because the potential is
H\"older (Lemma~\ref{lem:gibbs-dist}). So the proof for $\nu$ goes through with the modulus
$\omega(\rho)=\rho\log(e/\rho)$ replaced by $\rho^\alpha$ (Section~\ref{sec:gibbs}).
The geometric part of the argument, namely the temporal distance, the local frequency and the
specification lemma, does not involve the measure at all, which is why one proof treats every
equilibrium state.

\smallskip\noindent\textbf{Relation to Butterley--War.} The reduction we use is the one of \cite{BW20}. There, a Markov
section is built for the flow, the return time is projected along the local stable manifolds, and the flow is reduced to
an expanding semiflow over a uniformly expanding Markov map whose roof is not cohomologous to a piecewise constant
function when the stable and unstable bundles are not jointly integrable. Exponential mixing of the semiflow is then
proved with Lasota--Yorke estimates, a transversality estimate and an oscillatory cancellation estimate, and it is
transferred back to the flow through the H\"older disintegration of the measure along the stable leaves, which is where
the $C^{1+}$ regularity of the stable bundle is used. The same scheme is followed here, step by step.
\begin{itemize}
\item \emph{Quotient and roof} (Section~\ref{sec:markov}). We build a Markov section for the frame flow and straighten the roof of $f^t_\psi$
along stable fibres, which gives an expanding map $T$ and a roof $\tilde\tau$. That $\tilde\tau$ is not cohomologous to a
constant is Theorem~\ref{thm:rigidity}: for $\psi$ not of contact type the Birkhoff integrals of $\bF$ are unbounded.
\item \emph{Basic estimates} (Section~\ref{sec:markov}). The cone and the Lasota--Yorke inequality play the role of the
H\"older norm with the $b$-dependent weight in \cite{BW20}.
\item \emph{Transversality} (Sections~\ref{sec:temporal}--\ref{sec:spec}). Instead of cones of derivatives of the roof, we use the temporal
distance and the specification lemma, which produce, in every direction, pairs of branches whose phases differ at the
right scale.
\item \emph{Oscillatory cancellation} (Section~\ref{sec:L2}). Cells of size $\eps/\Lambda_b$ replace the partition of unity at scale
$|b|^{-1/2}$, and the estimate is an $L^2$ contraction of Dolgopyat operators.
\item \emph{From the semiflow to the flow} (Section~\ref{sec:mixing}, and Lemma~\ref{lem:cond-meas} for equilibrium states). The disintegration
along stable fibres is used as in \cite{BW20}.
\end{itemize}
The difference is the regularity. In \cite{BW20} the stable bundle is $C^{1+}$, and the derivative bound on the roof that this gives is
what the estimate for expanding semiflows needs. Here the stable bundle is not $C^1$ (Theorem~\ref{thm:rigidity}), the straightened roof does not
satisfy such a bound, and the estimate has to be carried out at the variable scale $1/\Lambda_b$.

\smallskip\noindent\textbf{Strategy.} The proof follows Dolgopyat's method \cite{Dol98}, in the form used in \cite{BW20}, for the
twisted transfer operators of a Markov section, with four ingredients. The second, the estimate at a variable
local frequency, is stated as Proposition~\ref{prop:localfreq} below. Point-dependent scales are already used in
\cite{TZ23}, where the scale is defined implicitly by matching the temporal function with template
functions and one has to control a recurrent set on which the scale is adapted. Here the scale is explicit.
\begin{enumerate}
\item \emph{The temporal distance is a bilinear form} (Lemma~\ref{lem:temporal} and
Figure~\ref{fig:brin}). For a Brin
quadrilateral with unstable side $u\in\RR^{n-1}$ and stable side $w\in\RR^{n-1}$, the temporal
distance of $f^t_\psi$ equals
\[
-\big\langle\bW_{|w|,|u|}\,u,\,w\big\rangle+O\big(|u|(|w|+|u|\log(e/|u|))\big),\qquad
\bW_{r,s}=\int_{-\frac12\log(1/r)}^{\frac12\log(1/s)}\bF\circ f^t\,dt .
\]
We confirmed this numerically for $n=4$ (Remark~\ref{rem:numerics}).
\item \emph{Dolgopyat's estimate at a variable local frequency, with a matrix norm.} We work at the scale $1/\Lambda_b$, where
$\Lambda_b=|b|\,(1+\|\bY_b\|)$ and $\bY_b$ is the forward Birkhoff integral of $\bF$ of length
$\frac12\log|b|$. The roof function of the Markov section is straightened along stable fibres,
and its variation along unstable fibres is a temporal distance
(Lemma~\ref{lem:bridge}). The cone of functions has modulus of continuity $\rho\log(e/\rho)$ at the
scale $1/\Lambda_b$, and it is invariant under the twisted transfer operators
(Proposition~\ref{prop:LY}).
\item \emph{Non-integrability from one long Birkhoff integral.} Along an orbit segment the frame is
transported by the frame flow, and two shadowing segments differ by a rotation in $M_0$ which
is constant up to a bounded error. So the Birkhoff integral of $\bF$ along a long segment is known only
up to conjugation, but its \emph{norm} is known up to $O(1)$. By Theorem~\ref{thm:rigidity}
there are segments whose Birkhoff integrals have arbitrarily large norm. A family of branch pairs whose
suffix shadows such a segment, compared with a family with bounded suffix, forces
$\|\bY+\bB_p\|$ to be large for one $p$. We then choose, cell by cell, the unstable direction
and the stable gap direction to be the corresponding singular vectors
(Lemmas~\ref{lem:phase} and~\ref{lem:spec}). No closed geodesics are needed.
\item \emph{Cells and small frequencies.} Cells are cylinders of the expanding map, which are round
because the unstable dynamics is conformal. At small frequencies we need that the length
spectrum of $f^t_\psi$ is not contained in a discrete subgroup of $\RR$. This is proved by a homological
argument that works in every dimension: a closed geodesic and its reverse would wind in
opposite directions (Lemma~\ref{lem:weakmix}).
\end{enumerate}

\medskip\noindent\textbf{Where the nonintegrability comes from.} Heuristically, exponential mixing of
$f^t_\psi$ is governed by the temporal distance, the time needed to close a small quadrilateral with
unstable side $s$ and stable side $r$, and there are two situations. If $\psi$ is of contact type
(Theorem~\ref{thm:rigidity}), the temporal distance is comparable to $rs$, exactly as for the geodesic flow,
and the nonintegrability is the familiar one that comes from the contact structure. If $\psi$ is not of
contact type, then to leading order the temporal distance is $rs$ times the sum of the geodesic term and
the Birkhoff integral of $F$ (the second derivatives $X^-X^+\psi$) over a time of length about $\log\frac1{rs}$.
These integrals are unbounded, and the nonintegrability now comes from their fluctuations, which are of
a different nature from the contact mechanism. This is also why the natural frequency scale is no
longer $|b|$ but the larger, orbit-dependent scale $\Lambda_b=|b|\,(1+\|\bY_b\|)$, set by the size of these
Birkhoff integrals. The Dolgopyat estimate then has to be run at this variable scale
(Proposition~\ref{prop:localfreq}).

\medskip\noindent\textbf{Roadmap}\label{sec:roadmap}
The heart of the proof is a single estimate for the twisted transfer operators of a Markov
section. Fix a Markov section for $f^t$ (Section~\ref{sec:markov}), with unstable plaques $U_i$, expanding
map $T$ on $U=\bigcup U_i$, invariant probability measure $\mu$ on $U$, normalized potential $\phi$, and the
roof $\tilde\tau$ of $f^t_\psi$ straightened along stable fibres. For $s=\sigma+ib$ put
\[
\mathcal L_sh(x)=\sum_{Ty=x}e^{\phi(y)-s\tilde\tau(y)}h(y).
\]
For $x\in U$ let $\bY_b(x)=\int_0^{\frac12\log|b|}\bF(f^tx)\,dt$, computed in any frame over $x$, and define the
\emph{local frequency} and the modulus
\[
\Lambda_b(x)=|b|\,\big(1+\|\bY_b(x)\|\big),\qquad \omega(\rho)=\rho\log\frac e\rho .
\]

\begin{prop}[Dolgopyat estimate at the local frequency]\label{prop:localfreq}
Let $\psi\in C^\infty(SM)$ be positive and not of contact type. There are $\mathsf E,N,b_0,\sigma_0>0$ and
$\eta_*\in(0,1)$ such that the following holds for $|b|\ge b_0$ and $|\sigma|\le\sigma_0$. Let $h:U\to\mathbb C$ satisfy
$|h|\le1$ and
\[
|h(x)-h(x')|\le \mathsf E\,\omega\big(\Lambda_b(x)\,d(x,x')\big)
\]
whenever $x,x'$ lie in a common $U_i$ and $d(x,x')\le1/\Lambda_b(x)$. Then for every $n\ge0$,
\[
\int_U\big|\mathcal L_s^{nN}h\big|^2\,d\mu\le(1-\eta_*)^n .
\]
\end{prop}

\begin{rem}[Relation to Tsujii--Zhang]\normalfont\label{rem:TZ}
Proposition~\ref{prop:localfreq} is not a new method. It is the estimate of Dolgopyat \cite{Dol98} at a
point-dependent scale, in the spirit of Tsujii and Zhang \cite{TZ23}, whose Proposition~3.10 is an $L^2$ estimate for the
transfer operators of a Markov section at the scale $1/\Lambda^\epsilon(x)$ with $\epsilon=1/|b|$. The
form of the statement differs: they prove a decay $|b|^{-\kappa}$ of the $L^2$ norm after $C\log|b|$ iterations
for H\"older functions, we prove a contraction for functions with the modulus $\omega$ at the scale $1/\Lambda_b$.
The scale also differs. In \cite{TZ23} it is defined through a matching time with template functions, is not continuous, and one
needs it to be stable, tame and adapted to a recurrent set. Here, in dimension $2n-1\ge5$, the temporal distance is a bilinear form in a Birkhoff integral of
the matrix $\bF$ (Lemma~\ref{lem:temporal}), and the scale is $\Lambda_b=|b|(1+\|\bY_b\|)$. We do
not know whether Proposition~\ref{prop:localfreq} can be deduced from \cite{TZ23}, and their argument is written for
one-dimensional bundles.
\end{rem}

For a contact flow one would take $\Lambda_b=|b|$ and $\omega(\rho)=\rho$, and this is Dolgopyat's estimate. Here
the Birkhoff integrals $\bY_b$ are unbounded, so the natural scale $1/\Lambda_b$ varies from point to point. Since
$\Lambda_b\le C|b|\log|b|$, the loss is only logarithmic. The classical form of the estimate also holds
for $f^t_\psi$: if $\|h\|_\infty\le1$ and $\mathrm{Lip}(h)\le \mathsf E|b|$ on each $U_i$, then $h$ satisfies the hypothesis of
Proposition~\ref{prop:localfreq} (Remark~\ref{rem:LY-next}), and for $\sigma=0$ this is Corollary~\ref{cor:dolgopyat}.
Proposition~\ref{prop:localfreq} is stronger, because it applies to a larger class of functions, and the
proof needs that class: over a distance $1/|b|$ the phase $b\tilde\tau_N$ can vary by an amount of order
$\|\bY_b\|$ (Lemma~\ref{lem:temporal}), which is unbounded, and this is why we work at the scale $1/\Lambda_b$
rather than $1/|b|$. Section~\ref{sec:mixing} turns
Proposition~\ref{prop:localfreq} into sup-norm decay with a factor $|b|^{1/2}$, and then into exponential
mixing through the Laplace transform of the correlation function. The proof of
Proposition~\ref{prop:localfreq} is assembled as follows.
\begin{enumerate}
\item[$\S$\ref{sec:prelim}.] The algebraic setup, the form $\alpha_\psi$ and Theorem~\ref{thm:rigidity}. Output:
Corollary~\ref{cor:unbounded}, the Birkhoff integrals of $\bF$ are unbounded.
\item[$\S$\ref{sec:temporal}.] The temporal distance of a Brin quadrilateral is $-\langle\bW u,w\rangle$ to first order
(Lemma~\ref{lem:temporal}, Figure~\ref{fig:brin}).
\item[$\S$\ref{sec:phase}.] The phase $\theta$ along an unstable plaque, at the scale $1/\Lambda_b$: it is regular
(Lemma~\ref{lem:phase}(R)), it oscillates under an explicit condition (N), and of two bases with very
different backward Birkhoff integrals one satisfies that condition (A).
\item[$\S$\ref{sec:markov}.] Markov sections. As in \cite{BW20}, the flow is reduced to an expanding semiflow by
quotienting along stable fibres (the straightened roof). The variation of the relative phase of two inverse
branches is the phase $\theta$ (Lemma~\ref{lem:bridge}, Corollary~\ref{cor:roof-modulus}). The cone of functions
with modulus $\omega$ at the scale $1/\Lambda_b$ is invariant (Proposition~\ref{prop:LY}).
\item[$\S$\ref{sec:spec}.] Specification: two families of branch pairs, one with a long backward Birkhoff integral of
large norm and one with a short Birkhoff integral, whose stable gaps point in all directions
(Lemma~\ref{lem:spec}).
\item[$\S$\ref{sec:L2}.] Cells of size $\eps/\Lambda_b$; in each cell, the choice of a pair and a direction,
cancellation, and the $L^2$ contraction of Dolgopyat's operators (Corollary~\ref{cor:dolgopyat}, which
is Proposition~\ref{prop:localfreq} for $\sigma=0$).
\item[$\S$\ref{sec:mixing}.] Complex $s$ (Proposition~\ref{prop:complex-s}, which completes the proof of
Proposition~\ref{prop:localfreq}), sup-norm decay, small frequencies, the Laplace transform and the
passage from the suspension semiflow to $f^t_\psi$.
\item[$\S$\ref{sec:gibbs}.] Equilibrium states of a H\"older potential $\varpi$, for $\psi$ not of contact type
(Theorem~\ref{thm:gibbs}). The symbolic model of $\mu_\varpi$, with the boundaries of the rectangles of measure zero
(Lemma~\ref{lem:null-bdry}); the weights have a H\"older distortion bound and small balls have measure
bounded below (Lemma~\ref{lem:gibbs-dist}); the cone with the modulus $\rho^\alpha$ in place of $\omega$, for which
Proposition~\ref{prop:localfreq} holds (Proposition~\ref{prop:gibbs-cone}); sup-norm decay with a factor $|b|^\alpha$,
and the conditional measures of $\mu_\varpi$ on stable fibres (Lemma~\ref{lem:cond-meas}), which pass the Laplace
transform argument from $\nu$ to $\mu_\varpi$.
\end{enumerate}

\section{Preliminaries}\label{sec:prelim}

This section sets up the algebra and proves Theorem~\ref{thm:rigidity}. We fix the Lie algebra of $SO_0(n,1)$ and the horospherical fields $X^\pm_i$, write the canonical form $\alpha_\psi$ of the time change and the contact case, introduce the matrix $\bF=(X^-_jX^+_i\psi)$ and its trace, compute the weak exterior derivative of $\alpha_\psi$, and conclude that the Birkhoff integrals of $\bF$ are bounded exactly in the contact case.

\medskip\noindent\textbf{Notation.}\label{par:notation}
We collect the notation used throughout; each item is defined where it first matters, and the section references are to those places. 
\emph{Spaces and flows.} $G=SO_0(n,1)$, $FM=\Gamma\backslash G$ is the frame bundle, $SM=FM/M_0$ the unit tangent bundle, $\pi:SM\to M$ the projection, and $m$ the normalized Haar (Liouville) measure.
For $Y$ in the Lie algebra $\g$ of $G$ we write $e^{Y}(x)=x\exp Y$ for the time-one map of the flow of the left-invariant field $Y$; the frame flow is $f^t=e^{tZ}$, and $M_0$ acts on the right, $x\mapsto xk$.
The time change $f^t_\psi$ is generated by $Z_\psi=Z/\psi$ and preserves $\nu=\psi\,m/\!\int\psi\,dm$.
\emph{Splitting and leaves.} $\g=E^m\oplus E^c\oplus E^u\oplus E^s$, where $E^c=\RR Z$, $E^u=\mathrm{span}\{X^+_i\}$, $E^s=\mathrm{span}\{X^-_i\}$ ($1\le i\le n-1$) and $E^m$ is the Lie algebra of $M_0$.
For $u,w\in\RR^{n-1}$ we put $X^+(u)=\sum_iu_iX^+_i$ and $X^-(w)=\sum_iw_iX^-_i$, so that $u\mapsto e^{X^+(u)}(x)$ and $w\mapsto e^{X^-(w)}(x)$ parametrize the strong unstable and stable manifolds $W^u(x)$ and $W^s(x)$ of $f^t$; the weak ones are $W^{cu}(x)$ and $W^{cs}(x)$, and $W^u_\psi,W^s_\psi$ are the strong leaves of $f^t_\psi$.
The basic rule is $f^t(e^{X^+(u)}(x))=e^{X^+(e^{2t}u)}(f^tx)$ and $f^t(e^{X^-(w)}(x))=e^{X^-(e^{-2t}w)}(f^tx)$ \eqref{eq:conj}. Vectors $u$ always denote unstable and $w$ stable coordinates, and $|\cdot|$ and $\|\cdot\|$ are the Euclidean and operator norms.
\emph{Forms and strong bundles.} $\alpha_\psi$ is the canonical form of the time change, $h^\pm_i$ the functions \eqref{eq:hpm} that give its components on $E^u,E^s$, and $E^u_\psi,E^s_\psi$ the strong bundles of $f^t_\psi$.
\emph{Birkhoff integrals.} $\bF=(F_{ij})$ with $F_{ij}=X^-_jX^+_i\psi$, and by a Birkhoff integral we mean a finite-time integral $\int_a^b\bF(f^tx)\,dt$ of this matrix; for $r,s\in(0,1]$ and $T(r)=\frac12\log(1/r)$ we set
$\bB_r(x)=\int_{-T(r)}^0\bF(f^tx)\,dt$, $\bY_{(s)}(x)=\int_0^{T(s)}\bF(f^tx)\,dt$ and $\bW_{r,s}=\bB_r+\bY_{(s)}$.
For a frequency $b$ the forward Birkhoff integral is $\bY_b(x)=\int_0^{\frac12\log|b|}\bF(f^tx)\,dt$, the local frequency is $\Lambda_b=|b|(1+\|\bY_b\|)$, and $\omega(\rho)=\rho\log(e/\rho)$ is the modulus of continuity that appears with it. Norms of Birkhoff integrals are functions on $SM$ because $\bF(xk)=k^{-1}\bF(x)k$.
\emph{Temporal distance.} $\partial_\infty(x)=(x^-,x^+)$ is the pair of endpoints of the geodesic of the frame $x$, $\Delta_\psi(x;u,w)$ is the temporal distance of the $f$-Brin quadrilateral based at $x$ with unstable and stable sides $u,w$, and $\mathcal Q(\xi,\xi';\eta,\eta')$ is the temporal distance function of $f^t_\psi$ in boundary coordinates, so that $\Delta_\psi(x;u,w)=\mathcal Q(\xi,\xi';\eta,\eta')$ for the corners of that quadrilateral.
\emph{Symbolic dynamics.} $U_i$, $U=\bigcup U_i$, $T$, $\mu$, $\phi$ and $\tilde\tau$ are the unstable plaques, expanding map, invariant measure, potential and straightened roof of a Markov section, and $\mathcal L_s$, $s=\sigma+ib$, is the twisted transfer operator (Section~\ref{sec:markov} and the Roadmap below).
Constants $C,c$ may change from line to line and depend only on $\psi$, $n$ and $\Gamma$.

\subsection{The Lie algebra}

We realize $G=SO_0(n,1)$ as the identity component of the group of linear maps of $\RR^{n+1}$
preserving the form $x_0^2+\dots+x_{n-1}^2-x_n^2$, so that
$\g=\mathfrak{so}(n,1)=\{X:X^TJ+JX=0\}$ with $J=\mathrm{diag}(1,\dots,1,-1)$. Let $E_{ab}$ ($0\le a,b\le n$)
be the elementary matrices, and put
\[
Z=2(E_{0n}+E_{n0}),\qquad X^+_i=(E_{0i}-E_{i0})-(E_{in}+E_{ni}),\qquad X^-_i=-(E_{0i}-E_{i0})-(E_{in}+E_{ni}),
\]
for $1\le i\le n-1$, and $R_{ij}=2(E_{ij}-E_{ji})$ for $1\le i,j\le n-1$. Let $\m=E^m$ be the span of the
$R_{ij}$; it is the Lie algebra of the centralizer $M_0\cong SO(n-1)$ of $\{\exp(tZ)\}$, acting on
the coordinates $1,\dots,n-1$. A direct computation gives
\begin{equation}\label{eq:brackets}
[Z,X^\pm_i]=\mp2X^\pm_i,\qquad [X^\pm_i,X^\pm_j]=0,\qquad [X^-_i,X^+_j]=\delta_{ij}Z+R_{ij},
\end{equation}
with $R_{ij}=-R_{ji}$ and $R_{ii}=0$, and $\g=E^m\oplus E^c\oplus E^u\oplus E^s$, where $E^m=\m=\mathrm{span}\{R_{ij}\}$, $E^c=\RR Z$ and
$E^{u/s}=\mathrm{span}\{X^\pm_i\}$. These are the left-invariant bundles of the frame flow: $E^c$ is the flow direction, $E^u,E^s$ are tangent to the strong unstable and stable manifolds, and $E^m$ is tangent to the $M_0$-fibres. Moreover $[E^m,E^{u/s}]\subset E^{u/s}$, and $M_0$ acts on
$E^{u/s}\cong\RR^{n-1}$ by its standard representation. The flow $f^t(x)=x\exp(tZ)$ on $FM=\Gamma\backslash G$ is
the frame flow, and its projection to $SM$ is the geodesic flow run at speed $2$.

Every $Y\in\g$ defines a vector field on $FM$, $Y(x)=\frac{d}{d\sigma}\big|_{\sigma=0}x\exp(\sigma Y)$, and
$Y\mapsto Y(\cdot)$ is a Lie algebra homomorphism. These fields generate right translations,
which preserve the Haar probability measure $m$, so
\begin{equation}\label{eq:ibp}
\int(Y\phi)\,\varphi\,dm=-\int\phi\,(Y\varphi)\,dm,\qquad\int Y\phi\,dm=0\qquad(\phi,\varphi\in C^1(FM)).
\end{equation}
For $u\in\RR^{n-1}$ put $X^\pm(u)=\sum_iu_iX^\pm_i$, and for $Y\in\g$ and $x\in FM$ write
\[
e^{Y}(x)=x\exp Y
\]
for the time-one map of the flow of the left-invariant field $Y$. Thus $f^t=e^{tZ}$, and we use the maps
\[
e^{X^+(u)}(x)=x\exp X^+(u),\qquad e^{X^-(w)}(x)=x\exp X^-(w).
\]
Since $X^\pm(u)$ and $X^\pm(u')$ commute, $e^{X^+(u')}\big(e^{X^+(u)}(x)\big)=e^{X^+(u+u')}(x)$, and likewise for $X^-$. The maps
$u\mapsto e^{X^+(u)}(x)$ and $w\mapsto e^{X^-(w)}(x)$ from $\RR^{n-1}$ to $FM$ parametrize the strong unstable and strong stable manifolds
$W^u(x)$ and $W^s(x)$ of $f^t$ through $x$ (see below), and the velocity of $\sigma\mapsto e^{X^\pm(\sigma v)}(x)$ is the
left-invariant field $X^\pm(v)$. By \eqref{eq:brackets}, for $t\in\RR$ and $k\in M_0$,
\begin{equation}\label{eq:conj}
\begin{gathered}
f^t\big(e^{X^+(u)}(x)\big)=e^{X^+(e^{2t}u)}(f^t(x)),\qquad f^t\big(e^{X^-(w)}(x)\big)=e^{X^-(e^{-2t}w)}(f^t(x)),\\
e^{X^\pm(v)}(x)\,k=e^{X^\pm(k^{-1}v)}(xk).
\end{gathered}
\end{equation}
Differentiating, for $\varphi\in C^1(FM)$,
\begin{equation}\label{eq:conj-deriv}
X^+_i(\varphi\circ f^t)=e^{2t}(X^+_i\varphi)\circ f^t,\qquad X^-_i(\varphi\circ f^t)=e^{-2t}(X^-_i\varphi)\circ f^t .
\end{equation}
The strong unstable and stable manifolds of $f^t$ through $x\in FM$ are $W^u(x)=e^{X^+(\RR^{n-1})}(x)$ and $W^s(x)=e^{X^-(\RR^{n-1})}(x)$, and the weak
ones are $W^{cu}(x)=\bigcup_{s\in\RR,\,k\in M_0}W^u\big(f^s(x)k\big)$ and $W^{cs}(x)=\bigcup_{s,k}W^s\big(f^s(x)k\big)$. Their
images in $SM$ are the corresponding manifolds of the geodesic flow. Functions on $SM$ are identified with
right $M_0$-invariant functions on $FM$. For such a function $\varphi$, $R\varphi=0$ for $R\in\m$, and
$(X^\pm_i\varphi)(xk)=\sum_lk_{li}(X^\pm_l\varphi)(x)$ for $k\in M_0$; that is, the vector
$X^\pm\varphi=(X^\pm_i\varphi)_i$ satisfies $X^\pm\varphi(xk)=k^{-1}X^\pm\varphi(x)$. We use the frame norms
$\|\varphi\|_{C^k}=\max\|Y_1\cdots Y_j\varphi\|_{C^0}$ over $j\le k$ and $Y_l$ in the basis
$\{Z,X^\pm_i,R_{ij}\}$, and the left-invariant Riemannian metric for which this basis is
orthonormal up to the factor $2$ on $R_{ij}$.

Near the identity, the product map
\begin{equation}\label{eq:product}
\RR^{n-1}\times\RR\times M_0\times\RR^{n-1}\to G,\qquad(w,s,k,u)\mapsto \exp X^-(w)\,\exp(sZ)\,k\,\exp X^+(u),
\end{equation}
is a diffeomorphism from a neighbourhood of $(0,0,e,0)$ onto a neighbourhood of $e$, because the
differential at that point is the identity of $E^s\oplus E^c\oplus E^m\oplus E^u$. The same
holds for the other orders of the four factors, and, by compactness of $M_0$, for $k$ near any
point of $M_0$. Acting on $x\in FM$ from the right, it is the \emph{local product structure}
$(w,s,k,u)\mapsto e^{X^+(u)}(f^s(e^{X^-(w)}(x))\,k)$: move along $W^s$, then along the orbit, rotate, and move along $W^u$.

We also need the first-order effect of reordering a stable and an unstable move. For $w,v\in\RR^{n-1}$ put
$R(w,v)=\sum_{i,j}w_iv_jR_{ij}\in E^m$.

\begin{lemma}[Commutation]\label{lem:comm}
For $x\in FM$ and $|w|\le1$, the curve $\sigma\mapsto e^{X^-(w)}(e^{X^+(\sigma v)}(x))$ is the integral curve, starting at
$e^{X^-(w)}(x)$, of the left-invariant field
\[
V(w,v)=X^+(v)-\langle w,v\rangle Z-R(w,v)+S(w,v),\qquad S(w,v)\in E^s,\quad |S(w,v)|\le C|w|^2|v| .
\]
In particular $e^{X^-(w)}(e^{X^+(v)}(x))$ is at distance at most $C|v|$ from $e^{X^-(w)}(x)$ for $|v|\le1$.
\end{lemma}

\begin{proof}
We have
\[\begin{aligned}e^{X^-(w)}(e^{X^+(\sigma v)}(x))&=x\exp(\sigma X^+(v))\exp X^-(w)\\&=e^{X^-(w)}(x)\exp\big(\sigma\,\mathrm{Ad}(\exp(-X^-(w)))X^+(v)\big),\end{aligned}\]
so the curve is the integral curve of the constant field $\mathrm{Ad}(\exp(-X^-(w)))X^+(v)$. By \eqref{eq:brackets}, $\mathrm{ad}(X^-(w))$
maps $E^u$ into $E^c\oplus E^m$, maps $E^c\oplus E^m$ into $E^s$ and kills $E^s$, so the series for $\mathrm{Ad}$ stops after three
terms:
\[\mathrm{Ad}(\exp(-X^-(w)))X^+(v)=X^+(v)-[X^-(w),X^+(v)]+\tfrac12[X^-(w),[X^-(w),X^+(v)]],\]
with
$[X^-(w),X^+(v)]=\langle w,v\rangle Z+R(w,v)$ and the last term in $E^s$ of size $\le C|w|^2|v|$.
\end{proof}

\subsection{The form $\alpha_\psi$ and the contact case}

Let $(\alpha^0,\alpha^+_i,\alpha^-_i)$ be the left-invariant $1$-forms on $FM$ dual to
$(Z,X^+_i,X^-_i)$ and vanishing on $\m$. The form $\alpha^0$ is right $M_0$-invariant and vanishes on
$\m$, so it descends to $SM$; we denote the descended form by $\alpha$. For $\psi\in C^\infty(SM)$, not
necessarily positive, put
\begin{equation}\label{eq:hpm}
h^+_i=\int_0^\infty e^{-2t}(X^+_i\psi)\circ f^{-t}\,dt,\qquad
h^-_i=\int_0^\infty e^{-2t}(X^-_i\psi)\circ f^{t}\,dt,
\end{equation}
\begin{equation}\label{eq:alphapsi}
\alpha_\psi=\psi\,\alpha^0+\sum_ih^+_i\alpha^+_i-\sum_ih^-_i\alpha^-_i .
\end{equation}
The functions $h^\pm_i$ are continuous, with $|h^\pm|\le\frac12\|X^\pm\psi\|_{C^0}$, and $\psi\mapsto\alpha_\psi$ is
linear. Since $X^\pm\psi(xk)=k^{-1}X^\pm\psi(x)$ and $f^t$ commutes with the right action, the vectors
$h^\pm=(h^\pm_i)$ satisfy $h^\pm(xk)=k^{-1}h^\pm(x)$. Since $\mathrm{Ad}(k)$ acts on $E^{u/s}$ by $k$, the
forms $\sum_ih^\pm_i\alpha^\pm_i$ are right $M_0$-invariant; they vanish on $\m$. Hence $\alpha_\psi$ descends to
a continuous $1$-form on $SM$, also denoted $\alpha_\psi$. On $FM$,
\[
\alpha_\psi(Z)=\psi,\qquad\alpha_\psi(X^+(u))=\langle h^+,u\rangle,\qquad\alpha_\psi(X^-(w))=-\langle h^-,w\rangle,\qquad\alpha_\psi(\m)=0 .
\]

\begin{lemma}\label{lem:Zh}
The functions $h^\pm_i$ are differentiable along $Z$, and $Zh^+_i=X^+_i\psi-2h^+_i$,
$Zh^-_i=2h^-_i-X^-_i\psi$. They are also differentiable along $W^u$ and $W^s$ respectively, with
\begin{equation}\label{eq:hpm-lip}
X^+_jh^+_i=\int_0^\infty e^{-4t}(X^+_jX^+_i\psi)\circ f^{-t}\,dt,\qquad
X^-_jh^-_i=\int_0^\infty e^{-4t}(X^-_jX^-_i\psi)\circ f^{t}\,dt ,
\end{equation}
and these are symmetric in $i,j$.
\end{lemma}

\begin{proof}
Substituting $t=t'+\tau$,
$h^+_i(f^\tau x)=e^{-2\tau}\big[h^+_i(x)+\int_{-\tau}^0e^{-2t'}(X^+_i\psi)(f^{-t'}x)\,dt'\big]$, and differentiating
at $\tau=0$ gives the first identity. Similarly $h^-_i(f^\tau x)=e^{2\tau}\int_\tau^\infty e^{-2t'}(X^-_i\psi)(f^{t'}x)\,dt'$.
For \eqref{eq:hpm-lip}, differentiate under the integral sign using \eqref{eq:conj-deriv}; the
integrals converge uniformly. Symmetry follows from $[X^\pm_i,X^\pm_j]=0$.
\end{proof}

\begin{prop}\label{prop:alphapsi}
Let $\psi>0$, and let $Z_\psi=Z/\psi$. The strong stable and unstable bundles of the Anosov flow $f^t_\psi$ on $SM$, lifted to
$FM$, are
\[
E^s_\psi=\{X^-(w)+\langle h^-,w\rangle Z_\psi:\ w\in\RR^{n-1}\},\qquad
E^u_\psi=\{X^+(u)-\langle h^+,u\rangle Z_\psi:\ u\in\RR^{n-1}\},
\]
and $\alpha_\psi$ is the unique continuous $1$-form on $SM$ with $\ker\alpha_\psi=E^s_\psi\oplus E^u_\psi$ and
$\alpha_\psi(Z_\psi)=1$.
\end{prop}

\begin{proof}
Write $f^t_\psi(x)=f^{\tau(t,x)}(x)$, where $t=\int_0^{\tau(t,x)}\psi(f^rx)\,dr$. For a tangent vector $Y$ at
$x$, $Df^t_\psi Y=Df^{\tau}Y+(Y\tau)\,Z$, and differentiating the defining relation gives
$\psi(f^\tau x)\,Y\tau=-\int_0^\tau Y(\psi\circ f^r)\,dr$. Take $Y=X^-(w)+cZ$. By \eqref{eq:conj},
$Df^\tau X^-(w)=e^{-2\tau}X^-(w)$, and $Df^\tau Z=Z$. By \eqref{eq:conj-deriv},
$\int_0^\tau Y(\psi\circ f^r)\,dr=\int_0^\tau e^{-2r}\langle X^-\psi(f^rx),w\rangle dr+c\big(\psi(f^\tau x)-\psi(x)\big)$. Hence
\[
Df^t_\psi Y=e^{-2\tau}X^-(w)+\frac1{\psi(f^\tau x)}\Big[c\,\psi(x)-\int_0^\tau e^{-2r}\langle X^-\psi(f^rx),w\rangle\,dr\Big]Z .
\]
The bracket tends to $c\psi(x)-\langle h^-(x),w\rangle$ as $\tau\to\infty$, and it is
$\int_\tau^\infty e^{-2r}\langle X^-\psi(f^rx),w\rangle dr=O(e^{-2\tau}|w|)$ if $c=\langle h^-(x),w\rangle/\psi(x)$. So $Y$ is
exponentially contracted exactly for this $c$. Since $\tau(t,x)\ge t/\max\psi$, the space of such $Y$ is
the strong stable bundle of the Anosov flow $f^t_\psi$, which has dimension $n-1$. The unstable
case is the same with $t\to-\infty$. Finally, $\alpha_\psi(X^-(w)+\langle h^-,w\rangle Z_\psi)=-\langle h^-,w\rangle+\langle h^-,w\rangle=0$,
similarly on $E^u_\psi$, and $\alpha_\psi(Z_\psi)=1$. Uniqueness is clear, since
$E^s_\psi\oplus E^u_\psi\oplus\RR Z_\psi$ is the whole tangent space.
\end{proof}

\begin{lemma}[Closed forms]\label{lem:closed}
Let $\eta$ be a smooth closed $1$-form on $SM$. Then $\alpha_{\eta(Z)}=\eta$.
\end{lemma}

\begin{proof}
Put $k=\eta(Z)$ and work on $FM$ with the pulled-back form. Since $d\eta=0$, $L_Z\eta=dk$, hence
$(f^t)^*\eta-\eta=d\int_0^tk\circ f^s\,ds$. Evaluate on $X^-_i$ and let $t\to+\infty$: by \eqref{eq:conj},
$((f^t)^*\eta)(X^-_i)=e^{-2t}\eta(X^-_i)\circ f^t\to0$ uniformly, while, by \eqref{eq:conj-deriv},
$X^-_i\int_0^tk\circ f^s\,ds=\int_0^te^{-2s}(X^-_ik)\circ f^s\,ds\to h^-_i(k)$. So $\eta(X^-_i)=-h^-_i(k)$. In the same
way, $(f^{-t})^*\eta-\eta=-d\int_0^tk\circ f^{-s}\,ds$ evaluated on $X^+_i$ gives $\eta(X^+_i)=h^+_i(k)$. Since
$\eta(Z)=k$ and $\eta$ vanishes on $\m$, comparison with \eqref{eq:alphapsi} gives $\eta=\alpha_k$.
\end{proof}

\begin{prop}\label{cor:contact}
If $\psi=c+\zeta(Z)+Zw$ is of contact type, then $\alpha_\psi=c\,\alpha+\pi^*\zeta+dw$. If moreover
$\psi>0$, then $c=\int\psi\,dm>0$, $\alpha_\psi$ is a smooth contact form with Reeb field $Z_\psi$, and
$\alpha_\psi\wedge(d\alpha_\psi)^{n-1}=c^{n-1}\psi\,\alpha\wedge(d\alpha)^{n-1}$. In particular
$f^t_\psi$ is a smooth contact Anosov flow preserving $\nu$.
\end{prop}

\begin{proof}
We have $h^\pm(c)=0$, so $\alpha_c=c\,\alpha$. Lemma~\ref{lem:closed} applied to $\pi^*\zeta$ and to $dw$, and
linearity, give the formula. Then $d\alpha_\psi=c\,d\alpha$. Since $\alpha^0$ is left-invariant,
$d\alpha^0(Y_1,Y_2)=-\alpha^0([Y_1,Y_2])$, and by \eqref{eq:brackets} its only non-zero components on the
basis are $d\alpha^0(X^+_i,X^-_j)=\delta_{ij}$. So $\iota_Zd\alpha=0$ and $d\alpha$ is non-degenerate on
$\ker\alpha$, and $\alpha$ is a contact form with Reeb field $Z$. The form $\beta=\alpha_\psi-\psi\alpha$
vanishes on $Z$, and so does $\iota_Zd\alpha$; hence $\iota_Z\big(\beta\wedge(d\alpha)^{n-1}\big)=0$, so the
top-degree form $\beta\wedge(d\alpha)^{n-1}$ vanishes. This gives the volume identity. Also
$\iota_{Z_\psi}d\alpha_\psi=(c/\psi)\iota_Zd\alpha=0$ and $\alpha_\psi(Z_\psi)=1$. Finally $\int\zeta(Z)\,dm=0$, because
$\zeta(Z)$ is odd under the flip $v\mapsto-v$, which preserves $m$, and $\int Zw\,dm=0$. So
$c=\int\psi\,dm>0$, and the volume form $\alpha_\psi\wedge(d\alpha_\psi)^{n-1}$ is non-vanishing. Being the
Reeb flow of a smooth contact form, $f^t_\psi$ preserves the volume $\psi\,\alpha\wedge(d\alpha)^{n-1}$, which is a
multiple of $\psi\,m$.
\end{proof}

\subsection{The matrix $\bF$ and its trace}

Put $F_{ij}=X^-_jX^+_i\psi$, a function on $FM$. Under the right action of $k\in M_0$, the matrix
$\bF=(F_{ij})$ transforms by $\bF(xk)=k^{-1}\bF(x)k$, where $k$ acts on $\RR^{n-1}$. Hence the trace
$\Ftr$ and the operator norm $\|\bF\|$ are functions on $SM$.

\begin{lemma}\label{lem:trace}
Let $\psi\in C^\infty(SM)$.
\begin{enumerate}
\item[(i)] $X^+_iX^-_j\psi=F_{ij}-\delta_{ij}Z\psi$.
\item[(ii)] $\int F_{ij}\,dm=0$ for all $i,j$. In particular $\int\Ftr\,dm=0$.
\item[(iii)] If $\psi$ is of contact type, then $F_{ij}=Zv_{ij}$ with $v_{ij}\in C^\infty(FM)$ for all $i,j$.
\item[(iv)] If $\psi=u\circ\pi$ with $u\in C^\infty(M)$, then
\[
\Ftr=(\Delta u)\circ\pi+Zw_u,\qquad w_u=\tfrac{n-1}2\,\psi-\tfrac14\,Z\psi ,
\]
where $\Delta$ is the Laplace--Beltrami operator (the trace of the Hessian).
\end{enumerate}
\end{lemma}

\begin{proof}
(i) By \eqref{eq:brackets}, $X^-_jX^+_i-X^+_iX^-_j=[X^-_j,X^+_i]=\delta_{ij}Z+R_{ji}$, and
$R_{ji}\psi=0$.

(ii) This follows from the invariance of $m$ under the flows of $X^\pm_j$.

(iii) Constants contribute nothing. Since $Z$ scales $X^-_j$ by $+2$ and $X^+_i$ by $-2$, the
operator $X^-_jX^+_i$ commutes with $Z$, so $X^-_jX^+_iZw=Z\big(X^-_jX^+_iw\big)$. Let $\eta$ be the
pull-back of $\pi^*\zeta$ to $FM$; it is closed. For left-invariant fields,
$X(\eta(Y))-Y(\eta(X))=\eta([X,Y])$. With $[X^+_i,Z]=2X^+_i$ this gives
$X^+_i(\eta(Z))=Z(\eta(X^+_i))+2\eta(X^+_i)$. Applying $X^-_j$ and using $X^-_jZ=ZX^-_j-2X^-_j$,
\[
X^-_jX^+_i(\eta(Z))=Z\big(X^-_j(\eta(X^+_i))\big).
\]
Hence $v_{ij}=X^-_j(\eta(X^+_i))+X^-_jX^+_iw$.

(iv) Work on $G$ and the hyperboloid $\HH^n=\{p\in\RR^{n,1}:\langle p,p\rangle=-1,\ p_n>0\}$. For a frame $g$, the base point
is $p=g\,e_n$, the unit tangent vector is $v=g\,e_0$, and $E_i=g\,e_i$ ($1\le i\le n-1$); then $(v,E_1,\dots,E_{n-1})$ is an
orthonormal basis of $T_p\HH^n$. Extend $u$ to a neighbourhood of $\HH^n$ in $\RR^{n,1}$ as a function $U$ that is
homogeneous of degree $0$, so that $dU_p(p)=0$. The second fundamental form of $\HH^n$ is $(X,Y)\mapsto\langle X,Y\rangle p$, so
for tangent vectors the Riemannian Hessian is $\nabla^2u(X,Y)=D^2U(X,Y)+\langle X,Y\rangle\,dU_p(p)=D^2U(X,Y)$. From the
matrices in Section~\ref{sec:prelim}, $X^\pm_ie_n=-e_i$ and $X^-_ie_i=-(e_0+e_n)$. Hence
$X^+_i\psi(g)=-dU_{ge_n}(ge_i)$, and differentiating along $X^-_i$,
\[
F_{ii}(g)=D^2U(E_i,E_i)+dU_p(v)+dU_p(p)=\nabla^2u(E_i,E_i)+du(v).
\]
Summing over $i$ gives $\Ftr=\Delta u-\nabla^2u(v,v)+(n-1)\,du(v)$ at $p$. The base point of $f^t(g)$ runs along the
geodesic with velocity $2v$, so $Z\psi=2\,du(v)$ and $Z^2\psi=4\,\nabla^2u(v,v)$. Hence
\[
\Ftr=(\Delta u)\circ\pi-\tfrac14Z^2\psi+\tfrac{n-1}2Z\psi=(\Delta u)\circ\pi+Zw_u .
\]
For $n=2$ this is the classical Casimir identity. We also checked (iv) numerically for $n=4$, to
fifteen digits.
\end{proof}

\smallskip\noindent\textbf{Analytic inputs.} We use:
\begin{itemize}
\item the frame flow $f^t$ on $FM$ is mixing on $L^2(m)$ \cite{Moo66,HM79}, and exponentially mixing
for smooth functions \cite{Moo87,KM96}: there are $k=k(n)$ and $\kappa_0>0$ with
$|\langle\varphi,a\circ f^t\rangle-\int\varphi\int a|\le C_{\rm mix}e^{-\kappa_0|t|}\|\varphi\|_{C^k}\|a\|_{C^k}$;
\item the regularity theorem \cite{LMM86}: if $w$ is continuous, $Zw$ exists in the sense of
distributions and is smooth, then $w$ is smooth;
\item de Rham's theorem for currents \cite{deR55}.
\end{itemize}

\subsection{The weak exterior derivative of $\alpha_\psi$}

Let $\lambda$ be a continuous $1$-form on $FM$. For $Y_1,Y_2\in\g$, viewed as left-invariant vector
fields, and $\varphi\in C^1(FM)$, put
\begin{equation}\label{eq:weak-d}
\big\langle d\lambda(Y_1,Y_2),\varphi\big\rangle=-\int\lambda(Y_2)\,Y_1\varphi\,dm+\int\lambda(Y_1)\,Y_2\varphi\,dm-\int\lambda([Y_1,Y_2])\,\varphi\,dm .
\end{equation}
Since left-invariant fields preserve $m$, this equals $\int d\lambda(Y_1,Y_2)\varphi\,dm$ when $\lambda$ is $C^1$,
and in general it is the exterior derivative of $\lambda$ in the sense of currents. We call $\lambda$
\emph{weakly closed} if all the distributions \eqref{eq:weak-d} vanish, for $Y_1,Y_2$ in a basis
of $\g$.

\begin{lemma}\label{lem:weakly-closed}
If $\lambda$ is weakly closed, then $\lambda=\eta+du$ with $\eta$ a smooth closed $1$-form on $FM$ and
$u\in C^1(FM)$. If moreover $\lambda$ is right $M_0$-invariant and vanishes on $\m$, then $\eta$ and
$u$ can be chosen right $M_0$-invariant with $\eta$ vanishing on $\m$, so that both descend to $SM$.
\end{lemma}

\begin{proof}
First, \eqref{eq:weak-d} is the exterior derivative of $\lambda$ in the sense of currents. For $C^1$ forms
this is the integration by parts \eqref{eq:ibp}. In general, let $\lambda_\eps=\int_G\rho_\eps(g)\,R_g^*\lambda\,dg$, where
$R_g$ is right translation and $\rho_\eps$ is a smooth approximate identity at $e$. Then $\lambda_\eps$ is smooth,
$\lambda_\eps\to\lambda$ uniformly, and both sides of the identity are continuous for uniform convergence.
So a weakly closed $\lambda$ is a closed current of degree one. By de Rham's theorem for currents
\cite{deR55}, $\lambda=\eta+dT$ with $\eta$ a smooth closed $1$-form and $T$ a distribution. Then $dT=\lambda-\eta$
is continuous, and in local coordinates a distribution whose gradient is continuous is a
$C^1$ function. So $T=u\in C^1$.
For the second part, average over the compact group $M_0$: with $R_k$ the right action,
$\bar\eta=\int_{M_0}R_k^*\eta\,dk$ and $\bar u=\int_{M_0}u\circ R_k\,dk$ satisfy $\lambda=\bar\eta+d\bar u$, since
$R_k^*\lambda=\lambda$ and $R_k^*$ commutes with $d$. Both are $M_0$-invariant. For $R\in\m$,
$\bar\eta(R)=\lambda(R)-R\bar u=0$. An $M_0$-invariant form vanishing on the fibres, and an
$M_0$-invariant function, descend to $SM=FM/M_0$.
\end{proof}

\begin{prop}\label{prop:weak-d}
Let $\psi\in C^\infty(SM)$, and let $\lambda$ be the pull-back of $\alpha_\psi$ to $FM$. Then for
$\varphi\in C^\infty(FM)$:
\begin{enumerate}
\item[(i)] the following components vanish: $d\lambda(R,Y)$ for $R\in\m$ and $Y\in\g$, and
$d\lambda(Z,X^\pm_i)$, $d\lambda(X^+_i,X^+_j)$ and $d\lambda(X^-_i,X^-_j)$;
\item[(ii)] $\big\langle d\lambda(X^+_i,X^-_j),\varphi\big\rangle=\delta_{ij}\int\psi\,dm\int\varphi\,dm-D_{F_{ij}}(\varphi)$.
\end{enumerate}
\end{prop}

\begin{proof}
(i) \emph{Components with $\m$.} Let $R\in\m$ and $Y\in\g$. Since $\lambda$ is right $M_0$-invariant,
$\lambda(Y)(xk)=\lambda(\mathrm{Ad}(k)Y)(x)$. Differentiating at $k=e$ in the direction $R$ gives
$R(\lambda(Y))=\lambda([R,Y])$ for the function $\lambda(Y)$, which is smooth along $M_0$-orbits. Since
$\lambda(R)=0$, \eqref{eq:weak-d} gives
\[
\big\langle d\lambda(R,Y),\varphi\big\rangle=-\int\lambda(Y)\,R\varphi\,dm-\int\lambda([R,Y])\,\varphi\,dm
=\int\big(R(\lambda(Y))-\lambda([R,Y])\big)\varphi\,dm=0 .
\]
\emph{Components $(Z,X^\pm_i)$.} We have $\lambda(Z)=\psi$, $\lambda(X^+_i)=h^+_i$ and $[Z,X^+_i]=-2X^+_i$. By
Lemma~\ref{lem:Zh}, $Zh^+_i=X^+_i\psi-2h^+_i$, so
\[
\begin{aligned}
\big\langle d\lambda(Z,X^+_i),\varphi\big\rangle&=-\int h^+_iZ\varphi+\int\psi X^+_i\varphi+2\int h^+_i\varphi\\
&=\int(X^+_i\psi)\varphi+\int\psi X^+_i\varphi=\int X^+_i(\psi\varphi)=0 .
\end{aligned}
\]
The component $(Z,X^-_i)$ is the same, with $\lambda(X^-_i)=-h^-_i$ and $Zh^-_i=2h^-_i-X^-_i\psi$.

\emph{Components $(X^\pm_i,X^\pm_j)$.} Since $X^+_i(\varphi\circ f^{-t})=e^{-2t}(X^+_i\varphi)\circ f^{-t}$, the
function $h^+_j$ is $C^1$ along $W^u$, with $X^+_ih^+_j=\int_0^\infty e^{-4t}(X^+_iX^+_j\psi)\circ f^{-t}\,dt$.
This is symmetric in $i,j$ because $[X^+_i,X^+_j]=0$. So the component $(X^+_i,X^+_j)$ is
$\int(X^+_ih^+_j-X^+_jh^+_i)\varphi\,dm=0$. The same holds for $(X^-_i,X^-_j)$.

(ii) By \eqref{eq:brackets}, $[X^+_i,X^-_j]=-\delta_{ij}Z-R_{ji}$, so $\lambda([X^+_i,X^-_j])=-\delta_{ij}\psi$ and
\[
\big\langle d\lambda(X^+_i,X^-_j),\varphi\big\rangle=\delta_{ij}\int\psi\varphi\,dm+\int h^-_jX^+_i\varphi\,dm+\int h^+_iX^-_j\varphi\,dm .
\]
By Fubini's theorem, $\int h^-_jX^+_i\varphi\,dm=\int_0^\infty e^{-2t}\langle(X^-_j\psi)\circ f^t,X^+_i\varphi\rangle\,dt$. By
\eqref{eq:ibp} and \eqref{eq:conj-deriv},
$e^{-2t}\langle(X^-_j\psi)\circ f^t,X^+_i\varphi\rangle=-e^{-2t}\langle X^+_i\big((X^-_j\psi)\circ f^t\big),\varphi\rangle=-\langle(X^+_iX^-_j\psi)\circ f^t,\varphi\rangle$,
and the left side is $O(e^{-2t}\|\psi\|_{C^1}\|\varphi\|_{C^1})$. The term with $h^+_i$ is treated in the same
way, using $X^-_j(\phi\circ f^{-t})=e^{2t}(X^-_j\phi)\circ f^{-t}$. This gives
\[
\int h^-_jX^+_i\varphi\,dm=-\int_0^\infty\big\langle\varphi,(X^+_iX^-_j\psi)\circ f^t\big\rangle dt,\qquad
\int h^+_iX^-_j\varphi\,dm=-\int_0^\infty\big\langle\varphi,F_{ij}\circ f^{-t}\big\rangle dt,
\]
with integrands $O(e^{-2t})$. By Lemma~\ref{lem:trace}(i), $X^+_iX^-_j\psi=F_{ij}-\delta_{ij}Z\psi$, and
$\int_0^T\langle\varphi,(Z\psi)\circ f^t\rangle dt=\langle\varphi,\psi\circ f^T\rangle-\langle\varphi,\psi\rangle$ tends to
$\int\varphi\int\psi-\int\varphi\psi$ by mixing. Since $\int F_{ij}\,dm=0$, both half-line integrals of $F_{ij}$
converge by exponential mixing, and their sum is $D_{F_{ij}}(\varphi)$. Adding up gives (ii).
\end{proof}

\subsection{Proof of Theorem~\ref{thm:rigidity}}

\emph{(1)$\Rightarrow$(2)} is Lemma~\ref{lem:trace}(iii).

\emph{(2)$\Rightarrow$(3).} $\int_0^LF_{ij}(f^tx)\,dt=v_{ij}(f^Lx)-v_{ij}(x)$ is bounded by $2\|v_{ij}\|_\infty$.

\emph{(3)$\Rightarrow$(4).} Fix $i,j$, put $a=F_{ij}$ and $S_sa=\int_0^sa\circ f^t\,dt$, so that
$|S_sa|\le K$ for all $s\ge0$. Put $u_T=\frac1T\int_0^TS_sa\,ds$. Then $|u_T|\le K$, $u_T$ is smooth and
\[
Zu_T=\frac1T\,S_Ta-a .
\]
Since $|S_Ta|\le K$, the first term tends to $0$ uniformly. Let $u\in L^\infty(m)$ be a weak-$*$ limit of
$u_{T_k}$ with $T_k\to\infty$. For $\varphi\in C^\infty(FM)$,
\[
\langle u,Z\varphi\rangle=\lim_k\langle u_{T_k},Z\varphi\rangle=-\lim_k\langle Zu_{T_k},\varphi\rangle=\langle a,\varphi\rangle .
\]
Apply this to $\varphi\circ f^{-t}$, using $Z(\varphi\circ f^{-t})=-\frac{d}{dt}(\varphi\circ f^{-t})$ and the invariance of $m$:
\[
\int_{-T}^{T}\langle\varphi,a\circ f^t\rangle\,dt=\int_{-T}^{T}\langle a,\varphi\circ f^{-t}\rangle\,dt
=\langle u,\varphi\circ f^{T}\rangle-\langle u,\varphi\circ f^{-T}\rangle .
\]
By mixing both terms tend to $\int u\,dm\int\varphi\,dm$, so $D_a(\varphi)=0$.

\emph{(4)$\Rightarrow$(1).} Let $c=\int\psi\,dm$, and let $\lambda$ and $\alpha^0$ be the pull-backs of $\alpha_\psi$ and
$\alpha$ to $FM$. The form $\alpha^0$ is left-invariant, with $d\alpha^0(Y_1,Y_2)=-\alpha^0([Y_1,Y_2])$. By
\eqref{eq:brackets}, its only non-zero components are $d\alpha^0(X^+_i,X^-_j)=\delta_{ij}$. By
Proposition~\ref{prop:weak-d} and (4), every component of $d(\lambda-c\,\alpha^0)$ vanishes. By
Lemma~\ref{lem:weakly-closed}, $\alpha_\psi-c\,\alpha=\eta+du_1$ on $SM$, with $\eta$ smooth and closed and
$u_1\in C^1(SM)$. Now $\pi^*:H^1(M;\RR)\to H^1(SM;\RR)$ is an isomorphism (for $n\ge3$ because the fibre
$S^{n-1}$ is simply connected, for $n=2$ because the Euler class is non-zero). So
$\eta=\pi^*\zeta+du_2$ with $\zeta$ closed on $M$ and $u_2$ smooth. Evaluating on $Z$, and using
$\alpha_\psi(Z)=\psi$, $\alpha(Z)=1$, gives
\[
\psi=c+\zeta(Z)+Zw,\qquad w=u_1+u_2\in C^1(SM).
\]
Since $Zw=\psi-c-\zeta(Z)$ is smooth, $w$ is smooth by \cite{LMM86}.

The final statement is Proposition~\ref{cor:contact}. \qed

\begin{cor}\label{cor:unbounded}
If $\psi$ is not of contact type, then for every $K>0$ there are $x\in FM$ and $L>0$ with
$\big\|\int_0^L\bF(f^tx)\,dt\big\|\ge K$.
\end{cor}

\begin{proof}
This is the implication (3)$\Rightarrow$(1) of Theorem~\ref{thm:rigidity}, in contrapositive form.
\end{proof}
\section{The temporal distance}\label{sec:temporal}

This section computes the temporal distance of $f^t_\psi$. We choose coordinates on the weak unstable and stable leaves, define the $f$-Brin quadrilaterals, prove that the functions $h^\pm$ defining the strong bundles are regular enough, and show in Lemma~\ref{lem:temporal} that the temporal distance of a quadrilateral is, to first order, the bilinear form $-\langle\bW_{|w|,|u|}u,w\rangle$ in a Birkhoff integral of $\bF$.

\subsection{Boundary coordinates}

We work on the universal cover $S\HH^n=G/M_0$ and on $G$, and keep the notation $\psi$, $\alpha_\psi$,
$f^t_\psi$, $f^t$ for the lifts. A frame $x\in G$ determines the oriented geodesic of its first vector.
Let $\partial_\infty(x)=(x^-,x^+)\in\partial\HH^n\times\partial\HH^n$ be its backward and forward endpoints, where
$\partial\HH^n=S^{n-1}$. Then $\partial_\infty(f^t(x)\,k)=\partial_\infty(x)$ for $k\in M_0$, and $\partial_\infty(hx)=h\cdot\partial_\infty(x)$ for $h\in G$. The points
$e^{X^+(u)}(x)$ have the same backward endpoint as $x$, and the points $e^{X^-(w)}(x)$ the same forward endpoint.
So the weak unstable leaf of $x$ is $W^{cu}(x)=\{y:y^-=x^-\}$ and the weak stable leaf is
$W^{cs}(x)=\{y:y^+=x^+\}$. Since $f^t_\psi$ is a time change of $f^t$, both flows have the same orbits and
the same weak leaves.

Let $I,J\subset S^{n-1}$ be connected open sets (for instance small balls) with disjoint closures, and put $U_{I,J}=\{y\in S\HH^n:y^-\in I,\ y^+\in J\}$.
Let $\sigma:I\times J\to U_{I,J}$ be the smooth section that sends $(\xi,\eta)$ to the point of the geodesic
from $\xi$ to $\eta$ closest to a fixed base point. Then
\[
(\xi,\eta,u)\longmapsto f^u_\psi\big(\sigma(\xi,\eta)\big)
\]
is a smooth diffeomorphism from $I\times J\times\RR$ onto $U_{I,J}$. Indeed, in the coordinates
$(\xi,\eta,s)\mapsto f^s(\sigma(\xi,\eta))$, the $f_\psi$-time $u=\int_0^s\psi(f^r\sigma(\xi,\eta))\,dr$ is smooth with
$\partial_su=\psi>0$. In these coordinates $Z_\psi=\partial_u$, and $f^t_\psi$ acts by $u\mapsto u+t$.

\begin{lemma}\label{lem:chart}
There are continuous functions $A,B$ on $I\times J$ such that the following hold.
\begin{enumerate}
\item[(i)] The strong unstable leaves of $f^t_\psi$ in $U_{I,J}$ are the graphs
$\{\xi=\mathrm{const},\ u=A(\xi,\eta)+k\}$, $k\in\RR$, and the strong stable leaves are the graphs
$\{\eta=\mathrm{const},\ u=B(\xi,\eta)+k\}$.
\item[(ii)] $A(\xi,\cdot)$ is $C^1$ in $\eta$ and $B(\cdot,\eta)$ is $C^1$ in $\xi$, with partial differentials
$d_\eta A$ and $d_\xi B$ continuous in $(\xi,\eta)$, and
\[
\alpha_\psi=du-d_\xi B-d_\eta A .
\]
\item[(iii)] Put $\Upsilon=A-B$, and for $\xi,\xi'\in I$ and $\eta,\eta'\in J$ put
\[
\mathcal Q(\xi,\xi';\eta,\eta')=\Upsilon(\xi',\eta')-\Upsilon(\xi,\eta')-\Upsilon(\xi',\eta)+\Upsilon(\xi,\eta).
\]
Then $\mathcal Q(\xi,\xi'';\eta,\eta')=\mathcal Q(\xi,\xi';\eta,\eta')+\mathcal Q(\xi',\xi'';\eta,\eta')$, and the same holds in the
second pair of variables.
\item[(iv)] On a weak stable leaf $\{\eta=c\}$ one has $\alpha_\psi=d\big(u-B(\cdot,c)\big)$, and on a weak
unstable leaf $\{\xi=c\}$ one has $\alpha_\psi=d\big(u-A(c,\cdot)\big)$.
\item[(v)] Let $\Gamma$ be a closed piecewise $C^1$ loop in $U_{I,J}$ made of four arcs: an arc in the weak
unstable leaf $\{\xi\}$ from a point over $(\xi,\eta)$ to a point over $(\xi,\eta')$, an arc in $\{\eta'\}$ to a
point over $(\xi',\eta')$, an arc in $\{\xi'\}$ to a point over $(\xi',\eta)$, an arc in $\{\eta\}$ to a point
over $(\xi,\eta)$, and finally the orbit segment back to the starting point. Then
$\int_\Gamma\alpha_\psi=\mathcal Q(\xi,\xi';\eta,\eta')$. In particular, if the four arcs lie in strong leaves of
$f^t_\psi$, then $\mathcal Q(\xi,\xi';\eta,\eta')$ is the $f_\psi$-time of the closing orbit segment, that is, the
temporal distance of $f^t_\psi$ for this quadrilateral.
\end{enumerate}
\end{lemma}

\begin{proof}
(i) Let $p=\sigma(\xi,\eta_0)$. The strong unstable leaf $W$ of $f^t_\psi$ through $p$ is a smooth
injectively immersed submanifold of dimension $n-1$, contained in the weak unstable leaf
$\{\xi\}\times(S^{n-1}\setminus\{\xi\})\times\RR$, and tangent to $E^u_\psi$, which is transverse to $Z_\psi$ by
Proposition~\ref{prop:alphapsi}. The weak leaf is the union of the translates $f^t_\psi W$ (for an Anosov flow the weak unstable manifold of a point is the union of the strong unstable manifolds of the points of its orbit \cite{Ano67}), and $W$ meets
each orbit at most once, because points of $W$ converge to each other in backward time. So $W$
meets each orbit of the weak leaf exactly once, and it is the graph $\{u=A(\xi,\eta)\}$ of a function
with $A(\xi,\eta_0)=0$. Since $f^t_\psi$ maps strong leaves to strong leaves and acts by $u\mapsto u+t$, the
other strong unstable leaves in $\{\xi\}\times J\times\RR$ are the graphs $\{u=A(\xi,\eta)+k\}$. The same
applies to $B$.

(ii) The graph of $A(\xi,\cdot)$ is a $C^1$ submanifold transverse to $\partial_u$, so $A(\xi,\cdot)$ is $C^1$, and
$d_\eta A$ is determined by $E^u_\psi$, which is continuous. So $d_\eta A$ is continuous in $(\xi,\eta)$, and so is
$A$, as an integral of $d_\eta A$ along paths from $\eta_0$. The same holds for $B$. The right side of
(ii) is a continuous $1$-form. It equals $1$ on $\partial_u=Z_\psi$. A tangent vector to a strong unstable
leaf has the form $(0,v,d_\eta A(v))$, and the form vanishes on it; similarly for strong stable
leaves. By Proposition~\ref{prop:alphapsi}, the right side equals $\alpha_\psi$.

(iii) This is immediate from the definition.

(iv) On $\{\eta=c\}$ the term $d_\eta A$ vanishes, so $\alpha_\psi=du-d_\xi B(\cdot,c)=d(u-B(\cdot,c))$. The other
case is the same.

(v) By (iv), the integral along the first arc is
$[u]-\big(A(\xi,\eta')-A(\xi,\eta)\big)$, where $[u]$ is the increment of $u$ along the arc. The second arc
gives $[u]-\big(B(\xi',\eta')-B(\xi,\eta')\big)$, the third $[u]-\big(A(\xi',\eta)-A(\xi',\eta')\big)$, the fourth
$[u]-\big(B(\xi,\eta)-B(\xi',\eta)\big)$, and the orbit segment gives $[u]$. The increments of $u$ add up
to $0$ along the closed loop, and the remaining terms add up to $\mathcal Q(\xi,\xi';\eta,\eta')$. If the arcs
lie in strong leaves, $\alpha_\psi$ vanishes on them, and the integral is the $f_\psi$-time of the orbit
segment.
\end{proof}

\subsection{$f$-Brin quadrilaterals}

\begin{lemma}[Closing]\label{lem:closing}
There are $c_0>0$ and $C$ such that for $u,w\in\RR^{n-1}$ with $|u|\,|w|\le c_0$ there are
$p,q\in\RR^{n-1}$, $\ell\in\RR$ and $k_0\in M_0$ near $e$, unique among solutions with $p,q$ small and $k_0$ near $e$ and
depending smoothly on $(u,w)$, with
\[
e^{X^-(q)}(\,e^{X^+(p)}(\,e^{X^-(w)}(\,e^{X^+(u)}(x))))=f^{-\ell}(x)\,k_0\qquad\text{for all }x\in FM .
\]
Moreover, $|p+u|\le C|u|^2|w|$, $|q+w|\le C|u||w|^2$, $|\ell|\le C|u||w|$ and $|k_0-e|\le C|u||w|$.
\end{lemma}

\begin{proof}
For $(u,w)$ near $(0,0)$ the local product structure gives
\[
e^{X^-(w)}(e^{X^+(u)}(x))=e^{X^+(\tilde u)}(y),\qquad y=f^s\big(e^{X^-(\tilde w)}(x)\big)\,k,
\]
with $(\tilde w,s,k,\tilde u)$ smooth in $(u,w)$ and independent of $x$. Hence the left side of the closing equation is
$e^{X^-(q)}(e^{X^+(\tilde u+p)}(y))$, by additivity of $e^{X^+(\cdot)}$. By the local product structure again, with the factors in the order $W^s$, orbit and rotation, $W^u$, $W^s$ (a local diffeomorphism by the same differential computation), the only small solution has $p=-\tilde u$. Then, by \eqref{eq:conj},
\[
\begin{aligned}e^{X^-(q)}(y)&=e^{X^-(q)}(f^s(z)\,k)\\&=f^s\big(e^{X^-(e^{2s}kq)}(z)\big)\,k=f^s\big(e^{X^-(\tilde w+e^{2s}kq)}(x)\big)\,k,\end{aligned}
\]
where $z=e^{X^-(\tilde w)}(x)$, which equals $f^{-\ell}(x)k_0$ exactly when $q=-e^{-2s}k^{-1}\tilde w$, $\ell=-s$ and $k_0=k$. So $p,q,\ell,k_0$ are smooth near $(0,0)$.

By \eqref{eq:conj}, applying $f^t$ maps the equation at $x$ for $(u,w)$ to the equation at $f^t(x)$ for
$(e^{2t}u,\,e^{-2t}w)$, so
\[
\begin{aligned}
p(e^{2t}u,e^{-2t}w)&=e^{2t}p(u,w),& q(e^{2t}u,e^{-2t}w)&=e^{-2t}q(u,w),\\
\ell(e^{2t}u,e^{-2t}w)&=\ell(u,w),& k_0(e^{2t}u,e^{-2t}w)&=k_0(u,w).
\end{aligned}
\]
We use this to define the solution for $|u||w|\le c_0$: choose $t$ with $e^{2t}|u|=e^{-2t}|w|=(|u||w|)^{1/2}$,
which is small. For the estimates, note that $p(u,0)=-u$, $p(0,w)=0$, $q(0,w)=-w$, $q(u,0)=0$,
and $\ell=0$, $k_0=e$ when $u=0$ or $w=0$. With $t$ as above, put $u'=e^{2t}u$ and $w'=e^{-2t}w$,
both of size $r:=(|u||w|)^{1/2}$. The function $F(u,w):=p(u,w)+u=u-\tilde u(u,w)$ is smooth near $0$ and
vanishes when $u=0$ and when $w=0$ (because $p(0,w)=0$ and $p(u,0)=-u$). Moreover its mixed second
derivative at $(0,0)$ vanishes: by Lemma~\ref{lem:comm}, $e^{X^-(w)}(e^{X^+(u)}(x))$ is reached from $e^{X^-(w)}(x)$ along $V(w,u)$, whose
$E^u$-component is exactly $X^+(u)$ and whose other components are $\langle w,u\rangle Z+R(w,u)\in E^c\oplus E^m$ (bilinear in $u,w$) and a term of size $O(|w|^2|u|)$ in $E^s$,
so, decomposing the endpoint with the stable move first (stable moves along $W^s$ add), the unstable coordinate is $\tilde u=u+O(|u||w|(|u|+|w|))$: the corrections come from commutators of $X^+(u)$ with the $E^c\oplus E^m$ term, of size $|u|^2|w|$. Hence
$F(u,w)=\int_0^1\!\!\int_0^1 D^2_{uw}F(\sigma u,\tau w)[u,w]\,d\sigma\,d\tau$ with $D^2_{uw}F=O(|u|+|w|)$, which gives
$|p(u',w')+u'|\le C|u'||w'|(|u'|+|w'|)\le Cr^3$. By the same argument,
$q+w=-(e^{-2s}k^{-1}-1)\tilde w+(w-\tilde w)$ with $s,\,k-e=O(|u||w|)$ and $w-\tilde w=O(|u||w|(|u|+|w|))$, so
$|q(u',w')+w'|\le Cr^3$. For $\ell$ and $k_0$ smoothness and the vanishing on the axes give
$|\ell(u',w')|+|k_0(u',w')-e|\le C|u'||w'|=Cr^2$. Rescaling by $p(u,w)=e^{-2t}p(u',w')$,
$q(u,w)=e^{2t}q(u',w')$, with $e^{-2t}=|u|/r$ and $e^{2t}=|w|/r$, gives $|p+u|\le C|u|r^2=C|u|^2|w|$,
$|q+w|\le C|w|r^2=C|u||w|^2$, $|\ell|\le C|u||w|$ and $|k_0-e|\le C|u||w|$.
\end{proof}

For $x\in G$ and $u,w$ as above, the \emph{$f$-Brin quadrilateral} based at $x$ has corners
\[
x_1=e^{X^+(u)}(x),\quad x_2=e^{X^-(w)}(x_1),\quad x_3=e^{X^+(p)}(x_2),\quad x_4=e^{X^-(q)}(x_3)=f^{-\ell}(x)\,k_0 .
\]
The corners lie over $\partial_\infty(x)=(\xi,\eta)$, $(\xi,\eta')$, $(\xi',\eta')$ and $(\xi',\eta)$, and in $S\HH^n$ the point
$x_4$ lies on the orbit of $x$. We put
\[
\Delta_\psi(x;u,w):=\mathcal Q(\xi,\xi';\eta,\eta').
\]
By Lemma~\ref{lem:chart}(v), this is the temporal distance of $f^t_\psi$ for the quadrilateral built from
$x$, the point of $W^u_\psi(x)$ over $(\xi,\eta')$ and the point of $W^s_\psi(x)$ over $(\xi',\eta)$. For
$|u|\le s_0$ and $|w|\le1$, with $s_0$ small, the four endpoints lie in sets $I\ni\xi,\xi'$ and $J\ni\eta,\eta'$ with
disjoint closures, which depend only on a neighbourhood of $x$; we always use such a chart.

\begin{lemma}[Loop formula]\label{lem:loop}
$\Delta_\psi(x;u,w)$ is the integral of $\alpha_\psi$ along the closed path in $G$ formed by: the $W^u$-arc
$\sigma\mapsto e^{X^+(\sigma u)}(x)$, $\sigma\in[0,1]$; the $W^s$-arc from $x_1$ to $x_2$; the $W^u$-arc from $x_2$ to $x_3$; the
$W^s$-arc from $x_3$ to $x_4$; an arc in $x_4M_0$; and the orbit segment from $f^{-\ell}(x)$ to $x$. Along a
$W^u$-arc $\sigma\mapsto e^{X^+(\sigma u)}(y)$ the integral is $\int_0^1\langle h^+(e^{X^+(\sigma u)}(y)),u\rangle\,d\sigma$, along a $W^s$-arc
$\sigma\mapsto e^{X^-(\sigma w)}(y)$ it is $-\int_0^1\langle h^-(e^{X^-(\sigma w)}(y)),w\rangle\,d\sigma$, along the $M_0$-arc it is $0$, and along
the orbit segment it is $\int_{-\ell}^0\psi(f^tx)\,dt$.
\end{lemma}

\begin{proof}
The $W^u$-arcs lie in weak unstable leaves, the $W^s$-arcs in weak stable leaves, and the
$M_0$-arc projects to a point of $S\HH^n$. So the projected loop is of the type of
Lemma~\ref{lem:chart}(v), and its $\alpha_\psi$-integral is $\mathcal Q(\xi,\xi';\eta,\eta')$. The arc formulas follow
from the values of $\alpha_\psi$ on $X^\pm(u)$, $Z$ and $\m$ listed after \eqref{eq:alphapsi}, since the
velocity of $\sigma\mapsto e^{X^\pm(\sigma u)}(y)$ is the left-invariant field $X^\pm(u)$.
\end{proof}

\subsection{Regularity of $h^\pm$}

\begin{lemma}\label{lem:hreg}
Let $\psi\in C^3(SM)$. There is $C$, depending only on $\|\psi\|_{C^3}$, such that for all $y\in G$:
\begin{enumerate}
\item[(H1)] $|h^+(e^{X^+(u)}(y))-h^+(y)|\le C|u|$, $|h^+(f^sy)-h^+(y)|\le C|s|$ and $|h^+(yk)-h^+(y)|\le C|k-e|$ for
$k\in M_0$; the same holds for $h^-$ with $X^-$ in place of $X^+$;
\item[(H2)] $|h^+(e^{X^-(\rho)}(y))-h^+(y)|\le C|\rho|\log(e/|\rho|)$ for $0<|\rho|\le1$, and
$|h^-(e^{X^+(\rho)}(y))-h^-(y)|\le C|\rho|\log(e/|\rho|)$.
\end{enumerate}
\end{lemma}

\begin{proof}
(H1) follows from Lemma~\ref{lem:Zh} and from $h^\pm(yk)=k^{-1}h^\pm(y)$. For (H2), by \eqref{eq:conj},
$f^{-t}(e^{X^-(\rho)}(y))=e^{X^-(e^{2t}\rho)}(f^{-t}(y))$, so the integrand of $h^+_i(e^{X^-(\rho)}(y))-h^+_i(y)$ is
bounded by
\[
e^{-2t}\min\big\{2\|X^+\psi\|_{C^0},\ e^{2t}|\rho|\,C\|\psi\|_{C^2}\big\}.
\]
Splitting the integral at $t_0=\frac12\log(1/|\rho|)$ gives $C|\rho|t_0+C|\rho|$. The second claim is symmetric.
\end{proof}

Formally, $X^-_jh^+_i=\int_0^\infty F_{ij}\circ f^{-t}\,dt$, which diverges unless the Birkhoff integrals of $\bF$ are
bounded. This is the source of the logarithm in (H2), and of the main term below.

\subsection{The bilinear temporal distance formula}

For $u\in(0,1]$ put $T(u)=\frac12\log(1/u)$. For $r,s\in(0,1]$ define the matrices
\[
\bB_r(x)=\int_{-T(r)}^0\bF(f^tx)\,dt,\qquad \bY_{(s)}(x)=\int_0^{T(s)}\bF(f^tx)\,dt,\qquad
\bW_{r,s}=\bB_r+\bY_{(s)} ,
\]
and write $\langle\mathbf Mu,w\rangle=\sum_{i,j}M_{ij}u_iw_j$. Since $\bF(xk)=k^{-1}\bF(x)k$, the same holds for these
Birkhoff integrals, and $\langle\mathbf M(xk)u,w\rangle=\langle\mathbf M(x)ku,kw\rangle$.

\begin{lemma}\label{lem:temporal}
Let $\psi\in C^3(SM)$. There are $s_0>0$ and $C$, depending only on $\|\psi\|_{C^3}$ and $n$, such that
for all $x\in G$ and $u,w\in\RR^{n-1}$ with $0<|u|\le s_0$ and $0<|w|\le1$,
\[
\Big|\,\Delta_\psi(x;u,w)+\big\langle\bW_{|w|,|u|}(x)\,u,\,w\big\rangle\,\Big|\le C\,|u|\,\Big(|w|+|u|\log\frac e{|u|}\Big).
\]
\end{lemma}

\begin{figure}[ht]
\centering
\begin{tikzpicture}[>=stealth,font=\small]
\begin{scope}
\coordinate (x) at (0,0);
\coordinate (x1) at (2.6,0);
\coordinate (x2) at (2.6,1.9);
\coordinate (x3) at (0.44,1.9);
\coordinate (x4) at (0.44,0.33);
\draw[dashed,->,gray] (-0.55,-0.41) -- (1.1,0.825) node[right,black]{$f^t$};
\draw[thick,->] (x) -- node[below]{$u$} (x1);
\draw[thick,->] (x1) -- node[right]{$w$} (x2);
\draw[thick,->] (x2) -- node[above]{$p$} (x3);
\draw[thick,->] (x3) -- node[left]{$q$} (x4);
\draw[very thick,red] (x4) -- (x);
\node[red] at (0.5,0.0) {$\ell$};
\foreach \p/\l/\a in {x/x/below left,x1/x_1/below right,x2/x_2/above right,x3/x_3/above left,x4/x_4/left}
  {\fill (\p) circle (1.3pt); \node[\a] at (\p) {$\l$};}
\node[align=center] at (1.3,-1.05) {unstable sides $\approx|u|$,\\ stable sides $\approx|w|$};
\end{scope}
\begin{scope}[xshift=5.9cm,yshift=1.2cm,xscale=0.72,yscale=0.28]
\fill[gray!15] (-0.8,-10.6) rectangle (3.5,0);
\draw[->] (-1.9,0) -- (4.7,0) node[right]{$t$};
\draw[->] (-1.9,-10.6) -- (-1.9,2.4);
\node[rotate=90] at (-2.5,-4.2) {$\log(\text{side length})$};
\draw[thick,blue] (-1.6,-10.2) -- (4.4,1.8) node[right]{$e^{2t}|u|$};
\draw[thick,teal] (-1.6,1.6) -- (4.4,-10.4) node[right]{$e^{-2t}|w|$};
\foreach \t in {-0.8,0,3.5} {\draw[densely dotted] (\t,0) -- (\t,-10.6);}
\node[anchor=north east,inner sep=1pt] at (-0.75,-10.7) {$-T(|w|)$};
\node[anchor=north west,inner sep=1pt] at (-0.05,-10.7) {$0$};
\node[anchor=north,inner sep=1pt] at (3.5,-10.7) {$T(|u|)$};
\draw[decorate,decoration={brace,amplitude=4pt}] (-0.8,0.7) -- node[above=3pt]{$\bB_{|w|}$} (0,0.7);
\draw[decorate,decoration={brace,amplitude=4pt}] (0,0.7) -- node[above=3pt]{$\bY_{(|u|)}$} (3.5,0.7);
\end{scope}
\end{tikzpicture}
\caption{Left: the $f$-Brin quadrilateral based at $x$. Its last corner $x_4$ lies on the orbit of $x$, at
time $-\ell$, and $\Delta_\psi(x;u,w)$ is the integral of $\alpha_\psi$ around the closed loop
(Lemma~\ref{lem:loop}). Right: along the orbit of $x$ the stable side has size $e^{-2t}|w|$ and the
unstable side $e^{2t}|u|$. Both are at most $1$ exactly for $-T(|w|)\le t\le T(|u|)$ (shaded), which is where
the first-order expansions in the proof of Lemma~\ref{lem:temporal} hold. The backward part of this
Birkhoff integral, set by the stable side, gives $\bB_{|w|}$; the forward part, set by the unstable side, gives
$\bY_{(|u|)}$; and $\bW_{|w|,|u|}=\bB_{|w|}+\bY_{(|u|)}$.}
\label{fig:brin}
\end{figure}

Figure~\ref{fig:brin} shows where the two Birkhoff integrals come from.

\begin{proof}
Below, $C$ denotes constants depending only on $\|\psi\|_{C^3}$ and $n$. We take $s_0\le c_0$, so that
Lemma~\ref{lem:closing} applies.

\emph{Step 1 (sides).} By Lemma~\ref{lem:loop},
\[
\begin{aligned}
\Delta_\psi(x;u,w)={}&\int_{-\ell}^0\psi(f^tx)\,dt+\int_0^1\langle h^+(e^{X^+(\sigma u)}(x)),u\rangle\,d\sigma+\int_0^1\langle h^+(e^{X^+(\sigma p)}(x_2)),p\rangle\,d\sigma\\
&-\int_0^1\langle h^-(e^{X^-(\sigma w)}(x_1)),w\rangle\,d\sigma-\int_0^1\langle h^-(e^{X^-(\sigma q)}(x_3)),q\rangle\,d\sigma .
\end{aligned}
\]

\emph{Step 2 (orbit segment and unstable sides).} The orbit term is $O(|\ell|)=O(|u||w|)$. By (H1), the two
unstable sides equal $\langle h^+(x),u\rangle+\langle h^+(x_2),p\rangle+O(|u|^2)$, and since $p=-u+O(|u|^2|w|)$ and $h^+$ is
bounded, they equal $\langle h^+(x)-h^+(x_2),u\rangle+O(|u|^2)$.

We compare $h^+(x_2)$ with $h^+(e^{X^-(w)}(x))$. By Lemma~\ref{lem:comm}, $x_2=e^{X^-(w)}(e^{X^+(u)}(x))$ lies within $C|u|$ of $e^{X^-(w)}(x)$, so by the local product structure (in the order $W^u$, orbit, $M_0$, $W^s$) it is reached from $e^{X^-(w)}(x)$ by moves of size at most $C|u|$ in each of these directions. Lemma~\ref{lem:hreg} then gives
$|h^+(x_2)-h^+(e^{X^-(w)}(x))|\le C|u|\log(e/|u|)$; here we use that $\rho\mapsto\rho\log(e/\rho)$ is increasing on
$(0,1]$, and $C s_0\le1$.

For the remaining difference put $\phi_i=X^+_i\psi\in C^2$. Since $f^{-t}(e^{X^-(w)}(x))=e^{X^-(e^{2t}w)}(f^{-t}(x))$,
\[
h^+_i(x)-h^+_i(e^{X^-(w)}(x))=\int_0^\infty e^{-2t}\big[\phi_i(f^{-t}x)-\phi_i\big(e^{X^-(e^{2t}w)}(f^{-t}(x))\big)\big]\,dt .
\]
For $0\le t\le T(|w|)$ we have $e^{2t}|w|\le1$. Taylor's formula along the integral curve $\sigma\mapsto e^{X^-(\sigma\rho)}(y)$ of $X^-(\rho)$
gives $\phi_i(e^{X^-(\rho)}(y))-\phi_i(y)=\sum_j\rho_j(X^-_j\phi_i)(y)+O(|\rho|^2)$, and $X^-_j\phi_i=F_{ij}$. So the integrand
equals $-\sum_jw_jF_{ij}(f^{-t}x)+O(e^{2t}|w|^2)$, and its integral over $[0,T(|w|)]$ is
$-\sum_jw_j(\bB_{|w|})_{ij}(x)+O(|w|)$. For $t\ge T(|w|)$ the bracket is bounded, so the tail is
$O(e^{-2T(|w|)})=O(|w|)$. Hence $\langle h^+(x)-h^+(e^{X^-(w)}(x)),u\rangle=-\langle\bB_{|w|}(x)u,w\rangle+O(|u||w|)$, and
\begin{equation}\label{eq:step2}
\text{orbit term}+\text{unstable sides}=-\langle\bB_{|w|}(x)u,w\rangle+O\big(|u||w|+|u|^2\log(e/|u|)\big).
\end{equation}

\emph{Step 3 (far stable side).} Put $\tilde q=-q$, so $|\tilde q-w|\le C|u||w|^2$. Reversing the orientation,
$e^{X^-(\sigma q)}(x_3)=e^{X^-(\tau\tilde q)}(x_4)$ with $\tau=1-\sigma$, so the last term of Step~1 equals
$\int_0^1\langle h^-(e^{X^-(\tau\tilde q)}(x_4)),\tilde q\rangle\,d\tau$. By \eqref{eq:conj},
\[
e^{X^-(\tau\tilde q)}(x_4)=e^{X^-(\tau\tilde q)}(f^{-\ell}(x)\,k_0)=f^{-\ell}\big(e^{X^-(\tau e^{-2\ell}k_0\tilde q)}(x)\big)\,k_0 .
\]
Using $h^-(yk_0)=k_0^{-1}h^-(y)$, (H1) along $Z$, and $|\ell|+|k_0-e|\le C|u||w|$, this term equals
$\int_0^1\langle h^-(e^{X^-(\tau\tilde w)}(x)),k_0\tilde q\rangle\,d\tau+O(|u||w|^2)$ with $\tilde w=e^{-2\ell}k_0\tilde q$. Since
$|\tilde w-w|+|k_0\tilde q-w|\le C|u||w|^2$ and $h^-$ is Lipschitz along $W^s$, it equals
$\int_0^1\langle h^-(e^{X^-(\tau w)}(x)),w\rangle\,d\tau+O(|u||w|^2)$.

\emph{Step 4 (stable sides compared).} By Step~3, the two stable sides together equal
\[
-\int_0^1\big\langle h^-(e^{X^-(\sigma w)}(x_1))-h^-(e^{X^-(\sigma w)}(x)),w\big\rangle\,d\sigma+O(|u||w|^2).
\]
Fix $\sigma\in[0,1]$ and put $y=e^{X^-(\sigma w)}(x)$ and $y'=e^{X^-(\sigma w)}(x_1)$. By \eqref{eq:conj},
$f^t(y)=e^{X^-(\sigma e^{-2t}w)}(f^tx)$ and $f^t(y')=e^{X^-(\sigma e^{-2t}w)}(e^{X^+(e^{2t}u)}(f^tx))$. Lemma~\ref{lem:comm}, applied at $f^tx$,
shows that $f^t(y')$ lies on the integral curve of the left-invariant field
\[
Y_t=V\big(\sigma e^{-2t}w,\,e^{2t}u\big)=e^{2t}X^+(u)-\sigma\big(\langle w,u\rangle Z+R(w,u)\big)+S_t,
\]
with $S_t\in E^s$ and $|S_t|\le C\sigma^2e^{-2t}|u||w|^2$,
starting at $f^t(y)$, at time $1$. In particular $\|Y_t\|\le Ce^{2t}|u|$ for $t\ge0$. Fix a small
constant $\eps_1$ and put $T_1=\frac12\log(\eps_1/|u|)$, so that $\|Y_t\|\le C\eps_1$ for $0\le t\le T_1$. Put
$\phi_j=X^-_j\psi\in C^2$. Taylor's formula along this integral curve gives
\[
\phi_j(f^t(y'))-\phi_j(f^ty)=(Y_t\phi_j)(f^ty)+O(e^{4t}|u|^2).
\]
Here $(Y_t\phi_j)=e^{2t}\sum_iu_iX^+_iX^-_j\psi-\sigma\langle w,u\rangle Z\phi_j-\sigma R(w,u)\phi_j+O(e^{-2t}|u||w|^2)$. The function
$R(w,u)\phi_j=[R(w,u),X^-_j]\psi$ is bounded by $C|u||w|$, because $R(w,u)\psi=0$. Multiplying by $e^{-2t}$
and integrating over $[0,T_1]$,
\[
\int_0^{T_1}e^{-2t}\big[\phi_j(f^t(y'))-\phi_j(f^ty)\big]dt=\sum_iu_i\int_0^{T_1}(X^+_iX^-_j\psi)(f^ty)\,dt+O(|u||w|)+O(e^{2T_1}|u|^2).
\]
For $t\ge T_1$ the bracket is bounded, so the tail contributes $O(e^{-2T_1})=O(|u|)$. Since
$h^-_j(z)=\int_0^\infty e^{-2t}\phi_j(f^tz)\,dt$, we obtain
$h^-_j(y')-h^-_j(y)=\sum_iu_i\int_0^{T_1}(X^+_iX^-_j\psi)(f^ty)\,dt+O(|u|)$. By Lemma~\ref{lem:trace}(i),
$X^+_iX^-_j\psi=F_{ij}-\delta_{ij}Z\psi$, and $\int_0^{T_1}(Z\psi)(f^ty)\,dt=\psi(f^{T_1}y)-\psi(y)=O(1)$. Moreover
$f^t(e^{X^-(\sigma w)}(x))=e^{X^-(\sigma e^{-2t}w)}(f^t(x))$, so replacing $y$ by $x$ changes the Birkhoff integral by $O(|w|)$, and
$|T_1-T(|u|)|\le C$. Hence
\[
h^-_j(e^{X^-(\sigma w)}(x_1))-h^-_j(e^{X^-(\sigma w)}(x))=\sum_iu_i\big(\bY_{(|u|)}(x)\big)_{ij}+O(|u|),
\]
and the stable sides contribute $-\langle\bY_{(|u|)}(x)u,w\rangle+O(|u||w|)$.

\emph{Step 5.} Adding \eqref{eq:step2} and Step~4 gives
$\Delta_\psi(x;u,w)=-\langle(\bB_{|w|}+\bY_{(|u|)})(x)u,w\rangle+O\big(|u||w|+|u|^2\log(e/|u|)\big)$.
\end{proof}

\begin{rem}[Numerical check]\normalfont\label{rem:numerics}
The lemma is local, so it can be tested on $S\HH^n$. For $n=4$ we took the base frame $x=e$, whose orbit is
the axis $\{f^t(e)\}$, and
\[
\psi=2+\big(0.2\,a^2+0.15\,ab+0.1\,c+0.12\,ad\big)\big(1+\tfrac12\rho\big)+0.05\,b\rho .
\]
Here $p=g\,e_n$ and $v=g\,e_0$ are the base point and the unit tangent vector of a frame $g$ in hyperboloid
coordinates, $a,b,c=\arctan p_1,\arctan p_2,\arctan p_3$, $d=\arctan v_2$, and $\rho=(v_n-v_1)/(p_n-p_1)$ is the
derivative along $v$ of a Busemann function, which tends to $\pm1$ along the axis as $t\to\pm\infty$. Along
the axis, $\bF(f^tx)$ tends to different limits $\bF_\pm$ as $t\to\pm\infty$, with $\bF_+=3\bF_-$, and $\bF_\pm$ are not
symmetric. We computed $\Delta_\psi$ by integrating the loop formula of Lemma~\ref{lem:loop} numerically, and
the Birkhoff integrals of $\bF$ by Gauss quadrature, with $\bF$ obtained by finite differences in $40$-digit
arithmetic. We used $|w|=0.2$, $|u|\in\{10^{-2},\dots,10^{-6}\}$ and several pairs of directions $(\hat u,\hat w)$.
In every case the quantity $(\Delta_\psi+\langle\bW_{|w|,|u|}u,w\rangle)/(|u||w|)$ was constant in $|u|$ to two or three
digits, while the main term grew linearly in $\log(1/|u|)$. Replacing $\bW$ by its transpose, or exchanging the
lengths of the backward and forward Birkhoff integrals, gave residuals that drift linearly in $\log(1/|u|)$. This
confirms the bilinear form, the index convention $F_{ij}=X^-_jX^+_i\psi$ and the assignment of the two integration
lengths: the backward Birkhoff integral is set by the stable side and the forward Birkhoff integral by the unstable side.
\end{rem}

\section{The phase at the local frequency}\label{sec:phase}

This section studies the phase $b\tilde\tau$ along an unstable plaque at the local frequency $\Lambda_b$. Lemma~\ref{lem:phase} proves that the phase is regular at the scale $1/\Lambda_b$ (R), gives an explicit condition (N) under which it oscillates, and shows that of two bases whose backward Birkhoff integrals differ strongly, one satisfies that condition (A).

From now on $\psi\in C^\infty(SM)$ is positive. For $|b|\ge1$ and $x\in FM$ put
\[
T_b=\tfrac12\log|b|,\qquad \bY_b(x)=\int_0^{T_b}\bF(f^tx)\,dt,\qquad \Lambda_b(x)=|b|\,\big(1+\|\bY_b(x)\|\big),
\]
where $\|\cdot\|$ is the operator norm. By the transformation rule of $\bF$, $\Lambda_b$ and $\tr\bY_b$
are functions on $SM$.

\smallskip\noindent\textbf{The gap function.} By the proof of Lemma~\ref{lem:closing}, for $|u||w|\le c_0$ we
can write
\[e^{X^-(w)}(e^{X^+(u)}(x))=e^{X^+(\tilde u)}(f^s(e^{X^-(\tilde w)}(x))\,k)\qquad(x\in FM)\]
with $(\tilde w,s,k,\tilde u)$ smooth in $(u,w)$, and the same
homogeneity argument gives $|\tilde w-w|\le C|u||w|^2$ and $\|\partial_w\tilde w-I\|\le C|u||w|$. Let $v\in FM$ and
$0<|w_0|\le\frac12$. By the displayed decomposition, the point $e^{X^-(w)}(e^{X^+(u)}(v))$ lies on $W^{cu}\big(e^{X^-(w_0)}(v)\big)$ if and only if
$\tilde w(u,w)=w_0$ (it lies on $W^u$ of $f^s(e^{X^-(\tilde w)}(v))k$, which belongs to $W^{cu}(e^{X^-(w_0)}(v))$ exactly when $\tilde w=w_0$). By the implicit function
theorem, for $|u|\le s_0$ (with $s_0$ small) there is a unique solution $w=w(u)$ near $w_0$. It is
smooth in $u$, and
\begin{equation}\label{eq:gap}
|w(u)-w_0|\le C|u||w_0|^2 .
\end{equation}

\smallskip\noindent\textbf{The phase.} Fix $v\in FM$ and $w_0$ as above, and for $|u|\le s_0$ put
\[
\theta(u)=\theta_{v,w_0}(u)=\Delta_\psi\big(v;u,w(u)\big).
\]
If $\partial_\infty(v)=(\xi,\eta)$, $\xi_0'$ is the backward endpoint of $e^{X^-(w_0)}(v)$ and $\eta(u)$ is the forward endpoint of
$e^{X^+(u)}(v)$, then $\theta(u)=\mathcal Q(\xi,\xi_0';\eta,\eta(u))$. So $\theta$ is the temporal distance of $f^t_\psi$ along the
unstable plaque $u\mapsto e^{X^+(u)}(v)$, measured against the fixed weak unstable leaf of $e^{X^-(w_0)}(v)$. Let
$|u|,|u'|\le s_0$ and $y=e^{X^+(u)}(v)$. Since $e^{X^+(u'-u)}(y)=e^{X^+(u')}(v)$, the gap function based at $y$ with the
same target leaf takes the value $w(u')$ at $u'-u$. So Lemma~\ref{lem:chart}(iii) gives
\begin{equation}\label{eq:theta-additive}
\theta(u')-\theta(u)=\mathcal Q\big(\xi,\xi_0';\eta(u),\eta(u')\big)=\Delta_\psi\big(y;\,u'-u,\,w(u')\big).
\end{equation}
By Corollary~\ref{cor:roof-modulus}(b) below, $b\,\theta$ is the variation of the relative phase of a
pair of inverse branches.

\begin{lemma}\label{lem:phase}
There are $b_0$ and $C\ge1$, depending only on $\psi$ and $n$, such that the following holds for all
$|b|\ge b_0$, $\eps\in(0,1]$ and $v,w_0$ as above with $|w_0|\ge\log(e|b|)/|b|$. Put
$L_\eps(v)=\log\big(e(1+\|\bY_b(v)\|)/\eps\big)$.
\begin{enumerate}
\item[(R)] \emph{Regularity.} If $|u|,|u'|\le1/\Lambda_b(v)$ and $\rho=\Lambda_b(v)|u'-u|$, then
\[
|b|\,|\theta(u')-\theta(u)|\le C\,|w_0|\,\rho\log\frac e{\rho|w_0|}.
\]
\item[(N)] \emph{Oscillation.} Let $\hat u$ be a unit vector, $t_\eps=\eps/\Lambda_b(v)$ and $\delta\in(0,1]$. If
\begin{equation}\label{eq:N-hyp}
\big|\big\langle(\bY_b(v)+\bB_{|w_0|}(v))\,\hat u,\ \hat w_0\big\rangle\big|\ge\delta\,\big(1+\|\bY_b(v)\|\big)+C\,L_\eps(v),
\end{equation}
where $\hat w_0=w_0/|w_0|$, then $|b\,\theta(t_\eps\hat u)|\ge\delta\,\eps\,|w_0|$.
\item[(A)] \emph{Two-base alternative.} For all $C_0,C'\ge1$ there is $C_{(A)}=C_{(A)}(C_0,C')$ such that the
following holds. Let $v_1,v_2\in FM$ satisfy $\|\bY_b(v_1)-\bY_b(v_2)\|\le C_0$, let $\bB_1,\bB_2$ be matrices with
$\beta=\max_i\|\bB_i\|$, and let $\delta\in(0,\frac14]$. If
\[
\|\bB_1-\bB_2\|\ge9\delta\beta+C_{(A)}\log\frac{e(2+\beta)}\eps ,
\]
then for some $i\in\{1,2\}$,
\[
\|\bY_b(v_i)+\bB_i\|\ge2\delta\,\big(1+\|\bY_b(v_i)\|\big)+2C'\,L_\eps(v_i).
\]
In that case, for the top singular vectors $(\hat u,\hat w)$ of $\mathbf M=\bY_b(v_i)+\bB_i$ and every unit vector
$\hat w'$ with $|\hat w'-\hat w|\le\frac12$, we have $|\langle\mathbf M\hat u,\hat w'\rangle|\ge\delta(1+\|\bY_b(v_i)\|)+C'L_\eps(v_i)$. In
particular, if $C'$ is at least the constant $C$ of \eqref{eq:N-hyp}, then \eqref{eq:N-hyp} holds at $v_i$ with
$\bB_i$ in place of $\bB_{|w_0|}(v_i)$, the direction $\hat u$ and the gap direction $\hat w'$.
\end{enumerate}
\end{lemma}

\begin{proof}
We take $b_0$ so large that $2/b_0\le s_0$, with $s_0$ as in Lemma~\ref{lem:temporal} and in
\eqref{eq:gap}. Write $C_{\bF}=\|\bF\|_{C^0}$. Three elementary facts are used.
\begin{enumerate}
\item[(a)] If $0<s\le1/|b|$, then $\|\bY_{(s)}(z)-\bY_b(z)\|\le\frac12C_{\bF}\log\frac1{s|b|}$, since the two Birkhoff integrals
differ by an interval of that length. If $1/|b|<s\le2/|b|$, the difference is at most $C_{\bF}$.
\item[(b)] If $|u|\le1/|b|$, then $\|\bY_b(e^{X^+(u)}(v))-\bY_b(v)\|\le\frac12\|\bF\|_{C^1}$. Indeed, since
$f^t(e^{X^+(u)}(v))=e^{X^+(e^{2t}u)}(f^t(v))$, the integrands differ by at most $\|\bF\|_{C^1}e^{2t}|u|$, and
$\int_0^{T_b}e^{2t}|u|\,dt\le|b||u|/2$.
\item[(c)] $\|\bB_r(z)\|\le\frac12C_{\bF}\log(1/r)$, and $\|\bB_r(z)-\bB_{r'}(z)\|\le\frac12C_{\bF}|\log(r/r')|$.
\end{enumerate}
We also use that $x\mapsto x\log(e/x)$ is increasing on $(0,1]$, so that $|s|\le2/|b|$ implies
$|s|\log(e/|s|)\le2\log(e|b|)/|b|\le2|w_0|$. By \eqref{eq:gap}, $|w(u)|\in[\frac12|w_0|,2|w_0|]$.

\emph{(R).} Put $y=e^{X^+(u)}(v)$ and $s=u'-u$, so $|s|\le2/\Lambda_b(v)\le2/|b|$ and $|b||s|=\rho/(1+\|\bY_b(v)\|)$. By
\eqref{eq:theta-additive} and Lemma~\ref{lem:temporal} at the base $y$, with $w=w(u')$,
\[
|\theta(u')-\theta(u)|\le|s|\,|w|\,\big(\|\bY_{(|s|)}(y)\|+\|\bB_{|w|}(y)\|\big)+C|s|\big(|w|+|s|\log(e/|s|)\big).
\]
By (a), (b) and (c), the bracket is at most
$\|\bY_b(v)\|+C+C\log(1+\|\bY_b(v)\|)+C\log^+\frac1\rho+C\log\frac e{|w_0|}$. Multiplying by $|b||s||w|$ and using
$(y+C\log(1+y))/(1+y)\le C$, we get
\[
|b|\,|\theta(u')-\theta(u)|\le C\rho|w_0|\Big(1+\log^+\frac1\rho+\log\frac e{|w_0|}\Big)\le C'\rho|w_0|\log\frac e{\rho|w_0|},
\]
where the last step uses $\rho\le2$ and $|w_0|\le\frac12$.

\emph{(N).} Put $u=t_\eps\hat u$, so $|u|=t_\eps\le1/|b|$, and $w=w(u)$. By Lemma~\ref{lem:temporal} at the base
$v$,
\[
|\theta(u)|\ge t_\eps\,|\langle\bW_{|w|,t_\eps}(v)\hat u,w\rangle|-Ct_\eps\big(|w|+t_\eps\log(e/t_\eps)\big)
\ge t_\eps\,|\langle\bW_{|w|,t_\eps}(v)\hat u,w\rangle|-5Ct_\eps|w_0| .
\]
By (a) and (c), $\|\bW_{|w|,t_\eps}(v)-\bY_b(v)-\bB_{|w_0|}(v)\|\le C'L_\eps(v)$, because
$\log\frac1{t_\eps|b|}=\log\frac{1+\|\bY_b(v)\|}\eps\le L_\eps(v)$ and $|\log(|w|/|w_0|)|\le C$. By \eqref{eq:gap},
$|\langle\bW\hat u,w\rangle-|w_0|\langle\bW\hat u,\hat w_0\rangle|\le\|\bW\|\,Ct_\eps|w_0|^2$, and
$\|\bW\|t_\eps\le C\eps L_\eps(v)\log(e/|w_0|)/|b|\le1$ for $|b|\ge b_0$. Hence
\[
|\theta(u)|\ge t_\eps|w_0|\Big(\big|\big\langle(\bY_b(v)+\bB_{|w_0|}(v))\hat u,\hat w_0\big\rangle\big|-C''L_\eps(v)\Big),
\]
using $L_\eps\ge1$. If the constant in \eqref{eq:N-hyp} is at least $C''$, the bracket is at least
$\delta(1+\|\bY_b(v)\|)$, and since $|b|t_\eps=\eps/(1+\|\bY_b(v)\|)$, we get $|b\,\theta(u)|\ge\delta\eps|w_0|$.

(A) Write $\bY_i=\bY_b(v_i)$ and $L_i=L_\eps(v_i)$, and suppose that
$\|\bY_i+\bB_i\|<2\delta(1+\|\bY_i\|)+2C'L_i$ for both $i$. Then
$\|\bY_i\|\le\beta+2\delta(1+\|\bY_i\|)+2C'L_i$, and since $\delta\le\frac14$,
$\|\bY_i\|\le2\beta+1+4C'L_i$. As $L_i=\log(e(1+\|\bY_i\|)/\eps)\ge1$, this gives
$L_i\le\log\frac{e(2+\beta)}\eps+\log\big(2(1+2C'L_i)\big)$, hence $L_i\le C_1\log(e(2+\beta)/\eps)$ with $C_1$
depending only on $C'$, and $\|\bY_i\|\le2\beta+1+4C'C_1\log(e(2+\beta)/\eps)$. Hence
\[
\|\bB_1-\bB_2\|\le\|\bY_1+\bB_1\|+\|\bY_2+\bB_2\|+\|\bY_1-\bY_2\|
<2\delta\big(2+\|\bY_1\|+\|\bY_2\|\big)+2C'(L_1+L_2)+C_0 ,
\]
which is at most $8\delta\beta+C_2\log(e(2+\beta)/\eps)$ with $C_2=2+C_0+8C'C_1$. So
$C_{(A)}=C_2$ works. For the last statement, if $(\hat u,\hat w)$ are top singular
vectors of $\mathbf M=\bY_i+\bB_i$, then $\langle\mathbf M\hat u,\hat w\rangle=\|\mathbf M\|$, and
$|\langle\mathbf M\hat u,\hat w'\rangle|\ge\|\mathbf M\|(1-|\hat w-\hat w'|)\ge\frac12\|\mathbf M\|$.
\end{proof}

The point of (A) is that the two bases only need to be compared through the \emph{norm} of
the difference of their backward Birkhoff integrals. In Section~\ref{sec:spec}, one family has a long
backward Birkhoff integral of large norm and the other a bounded one. The norm of a Birkhoff integral is a
function on $SM$, so the frame holonomy plays no role.

\section{Markov sections, the straightened roof and the cone}\label{sec:markov}

This section builds the symbolic model and the function space. We construct Markov sections for the frame flow with the refinement and re-centring of Remark~\ref{rem:tau-positive}, straighten the roof of the time change along stable fibres, and show that the variation of the relative phase of two inverse branches is the phase of Section~\ref{sec:phase}. We then define the cone of functions with modulus $\omega$ at the scale $1/\Lambda_b$ and prove that it is invariant under the twisted transfer operators (Proposition~\ref{prop:LY}).

From now on $\psi\in C^\infty(SM)$ is positive.

\subsection{Markov sections}

Fix a Markov section for $f^t$ on $SM$ \cite{Bow73,Rat73b}. It consists of finitely many rectangles
$R_i\subset D_i$, where each $D_i$ is a smooth disc of dimension $2n-2$ transverse to the flow, of the
form $D_i=\{e^{X^-(w)}(e^{X^+(u)}(y_i)):|u|,|w|<\eps_0\}$ for some $y_i\in SM$, and $\eps_0$ is small. For $z\in D_i$ the
\emph{stable fibre} $W^s_D(z)$ is the connected component of $D_i\cap W^{cs}(z)$ through $z$, and the
\emph{unstable fibre} $W^u_D(z)$ is the component of $D_i\cap W^{cu}(z)$. Each rectangle has the form
$R_i=[U_i,S_i]$, where $U_i$ is a closed subset of the strong unstable plaque $W^u(y_i)\cap D_i$ and
$S_i$ is a closed subset of $W^s_D(y_i)$, each the closure of its interior, and $[z,z']$ denotes the
point where $W^s_D(z)$ meets $W^u_D(z')$. For $z\in R_i$ we write $W^s_R(z)=W^s_D(z)\cap R_i$ and
$W^u_R(z)=W^u_D(z)\cap R_i$, and $\pi z\in U_i$ for the point where $W^s_R(z)$ meets $U_i$. Let $P$ be the
first return map of $f^t$ to $R=\bigcup R_i$, with return time $\tau_1$, and $T=\pi\circ P$ on $U=\bigcup U_i$. We
use the following properties \cite{Bow73,Rat73b}. Orbits that meet the boundaries of the
rectangles form a null set for $m$, and statements about boundary points are understood by
continuity from the interior.
\begin{enumerate}
\item[(M1)] \emph{Product structure.} For $z,z'\in R_i$, $W^s_R(z)$ and $W^u_R(z')$ meet in exactly one
point $[z,z']\in R_i$.
\item[(M2)] \emph{Markov property.} If $z\in R_i$ and $Pz\in R_j$, then $P(W^s_R(z))\subset W^s_R(Pz)$ and
$P(W^u_R(z))\supset W^u_R(Pz)$.
\item[(M3)] \emph{Hyperbolicity.} There are $C$ and $\kappa<1$ such that $d(P^nz,P^nz')\le C\kappa^nd(z,z')$ for
$z'\in W^s_R(z)$, and $d(P^{-n}w,P^{-n}w')\le C\kappa^nd(w,w')$ for $w'\in W^u_R(w)$ whenever $P^{-k}w$ and
$P^{-k}w'$ lie in a common rectangle for $0\le k\le n$.
\item[(M4)] \emph{Return times.} $0<\min\tau_1\le\max\tau_1<\infty$, and $\tau_1$ is Lipschitz on each set
$R_i\cap P^{-1}R_j$.
\end{enumerate}
By (M2), $\pi\circ P=T\circ\pi$ on $R$, and $T^N=\pi\circ P^N$ on $U$. The section is coded by a transition
matrix $\mathsf A$, which is irreducible because $f^t$ is transitive. An inverse branch $v$ of $T^N$
defined on $U_i$ corresponds to an admissible word $a_0\cdots a_{N-1}$ with $\mathsf A_{a_{N-1}i}=1$: the image
$v(U_i)$ is the set of $x\in U_{a_0}$ with $T^kx\in U_{a_k}$ for $k<N$ and $T^Nx\in U_i$, and $T^N$ maps it
bijectively onto $U_i$. Distinct words give images with disjoint interiors. If $\mathsf A$ has period
$d_0>1$, all word lengths chosen below are taken in the appropriate residue class modulo $d_0$.

Distances on $U_i$ are measured in the strong unstable leaf, with the left-invariant metric, so that
$d(x,e^{X^+(u)}(x))=|u|$ for small $u$. The time change does not change any of the objects above. It only
changes the return time, which becomes $\tau_\psi(y)=\int_0^{\tau_1(y)}\psi(f^ty)\,dt$.

\begin{lemma}[Conformality]\label{lem:M5}
\begin{enumerate}
\item[(M5)] There is $C$ such that for every inverse branch $v$ of $T^N$ on $U_i$ the following holds.
The map $T^N$ is a $C^1$ diffeomorphism from a neighbourhood of $v(U_i)$ in the strong unstable leaf
onto a neighbourhood $\hat U_i$ of $U_i$ of fixed size, so $v$ extends to $\hat U_i$. For $y$ in that
neighbourhood, $\|DT^N(y)\|\,\|DT^N(y)^{-1}\|\le C$ and $\|DT^N(y)^{-1}\|\le C\kappa^N$. The function
$x\mapsto\log|\det DT^N(vx)|$ is $C$-Lipschitz on $\hat U_i$.
\end{enumerate}
In particular every cylinder $v(U_j)$ contains a ball of radius $c\,\mathrm{diam}\,v(U_j)$, with $c>0$
independent of $N$.
\end{lemma}

\begin{proof}
Let $y_0\in v(U_i)$ and $\tau_0=\tau_{1,N}(y_0)=\sum_{k<N}\tau_1(P^ky_0)$. For $y$ near $y_0$ in the strong unstable
leaf, $P^Ny$ and $f^{\tau_0}y$ lie on the same orbit, hence on the same weak stable leaf. So
$T^Ny=\Pi(f^{\tau_0}y)$, where $\Pi$ is the projection along weak stable leaves from a flow box of fixed
time length around $U_i$ onto the strong unstable plaque of $U_i$. Since the weak stable foliation
of $f^t$ is smooth (its leaves are the manifolds $W^{cs}(x)$), $\Pi$ restricted to strong unstable plaques
in this flow box is a smooth diffeomorphism onto its image, with derivative and inverse
derivative bounded by a constant $C_\Pi$ and with $\log|\det D\Pi|$ Lipschitz. By \eqref{eq:conj},
$Df^{\tau_0}$ restricted to strong unstable leaves is $e^{2\tau_0}$ times an isometry. Since $|\tau_{1,N}(y)-\tau_0|$ is
bounded on the cylinder by (M3) and (M4), the same formula holds on a neighbourhood of fixed size,
with the same $\tau_0$. Hence $DT^N=D\Pi\circ e^{2\tau_0}I$, which gives the conformality bound and, since
$\tau_0\ge N\min\tau_1$, the contraction of $v$ with $\kappa=e^{-2\min\tau_1}$ after enlarging $C$. Expansion by
$e^{2\tau_0}/C_\Pi$ also shows that the image contains a neighbourhood of $U_i$ of fixed size. Finally
$\log|\det DT^N(vx)|=2(n-1)\tau_0+\log|\det D\Pi(f^{\tau_0}vx)|$, and $x\mapsto f^{\tau_0}vx$ is $C_\Pi$-Lipschitz. The
last statement follows because $U_i$ contains a ball of fixed radius and $v$ is quasi-conformal.
\end{proof}

\smallskip\noindent\textbf{Measures.} Let $\mu_R$ be the measure on $R$ induced by $m$, given on $D_i$ by the
smooth $(2n-2)$-form $\iota_Zm$. It is $P$-invariant, and it is also the transverse measure of $\nu$ for
$f^t_\psi$, because the flux of $\psi\,m$ under $Z/\psi$ equals the flux of $m$ under $Z$. Each $R_i$ has a
product chart $U_i\times S_i\to R_i$, $(x,y)\mapsto[x,y]$, which is a smooth diffeomorphism onto its image,
because the weak foliations of $f^t$ are smooth. So $\mu_R=q(x,y)\,dx\,dy$ with $q$ smooth and
positive. The projection $\mu$ of $\mu_R$ along stable fibres is $T$-invariant, and its density
$\varrho(x)=\int_{S_i}q(x,y)\,dy$ is smooth and positive on each $U_i$; we extend it smoothly to $\hat U_i$.
Put
\[
\phi(y)=\log\varrho(y)-\log\varrho(Ty)-\log|\det DT(y)| .
\]
Then $\mathcal L_01=1$ and $\mathcal L_0^*\mu=\mu$, where $\mathcal L_0h(x)=\sum_{Ty=x}e^{\phi(y)}h(y)$, and
$\phi_N(vx)=\log\varrho(vx)-\log\varrho(x)-\log|\det DT^N(vx)|$. We normalize $\mu$ to be a probability
measure. The conditional measures of $\mu_R$ on the stable fibres are $q(x,y)\,dy/\varrho(x)$, which depend
smoothly on $x$.

\smallskip\noindent\textbf{Adapted frames.} We fix continuous maps $\sigma_i:R_i\to FM$ with $\sigma_i(z)$ a frame
over $z$, as follows. Choose a frame $\tilde y_i$ over $y_i$, and put $\sigma_i(e^{X^+(u)}(y_i))=e^{X^+(u)}(\tilde y_i)$ on $U_i$.
Each $z\in R_i$ can be written uniquely as $z=f^c\big(e^{X^-(w)}(\pi z)\big)$ with $|c|,|w|$ small, and we put
$\sigma_i(z)=f^c\big(e^{X^-(w)}(\sigma_i(\pi z))\big)$, where $w$ is expressed in the frame $\sigma_i(\pi z)$. Along each stable fibre
the frames are related by $e^{X^-(\cdot)}$ and the flow, without rotation (no $M_0$ component), and along
$U_i$ by $e^{X^+(\cdot)}$. We use $\sigma_i$ to express unstable and stable vectors, and the Birkhoff integrals
$\bY_b$, $\bB_r$, at points of $R_i$. For $x\in U_i$ we put $\Lambda_b(x)=\Lambda_b(\sigma_i(x))$, which does not depend
on the frame.

\subsection{The straightened roof}

The stable straightening of $\tau_\psi$ is
\[
\chi(z)=\sum_{n\ge0}\big[\tau_\psi(P^n\pi z)-\tau_\psi(P^nz)\big],\qquad \tilde\tau=\tau_\psi+\chi-\chi\circ P ,
\]
so that $\chi=0$ on $U$, $\tilde\tau$ is constant on stable fibres, and for $y\in U$,
$\tilde\tau_N(y)=\sum_{k<N}\tilde\tau(P^ky)=\tau_{\psi,N}(y)-\chi(P^Ny)$. For a branch $v$ and $x\in U_i$ we write
$\hat x=P^N(vx)$, a point of the stable fibre of $x$ with $\pi\hat x=x$. We use the chart $(\xi,\eta,u)$ and the
functions $A$, $B$, $\Upsilon$ and $\mathcal Q$ of Lemma~\ref{lem:chart}, on a lift of a neighbourhood of $R_i$. Put
$\mathcal G(p)=u(p)-B(\xi(p),\eta(p))$. For points $p,p'$ of a weak stable or weak unstable leaf, $\int_{[p,p']}\alpha_\psi$
denotes the integral along a path in that leaf; by Lemma~\ref{lem:chart}(iv) it does not depend on
the path.

\begin{lemma}[Straightening identity]\label{lem:bridge}
Let $x,x'\in U_i$, and let $v,v_j,v_k$ be inverse branches of $T^N$ on $U_i$, extended as in (M5).
Write $\eta=\eta(x)$ and $\eta'=\eta(x')$.
\begin{enumerate}
\item[(i)] For $z\in R$ the series defining $\chi(z)$ converges, and $\chi(z)=\mathcal G(z)-\mathcal G(\pi z)=\int_{[\pi z,z]}\alpha_\psi$.
\item[(ii)] $\tilde\tau_N(vx')-\tilde\tau_N(vx)=-\mathcal Q\big(\xi(x),\xi(\hat x);\eta,\eta'\big)+\int_{[x,x']}\alpha_\psi-\int_{[vx,vx']}\alpha_\psi$.
\item[(iii)] For $\varphi_{jk}=\tilde\tau_N\circ v_j-\tilde\tau_N\circ v_k$,
\[
\varphi_{jk}(x')-\varphi_{jk}(x)=-\mathcal Q\big(\xi(\hat x_k),\xi(\hat x_j);\eta,\eta'\big)-\int_{[v_jx,v_jx']}\alpha_\psi+\int_{[v_kx,v_kx']}\alpha_\psi .
\]
\end{enumerate}
Here the paths from $vx$ to $vx'$ are images under $v$ of a path from $x$ to $x'$ in $\hat U_i$.
\end{lemma}

\begin{proof}
\emph{Lifts.} Lift $x$ and the disc $D_i$ to $S\HH^n$. The point $\hat x$ is lifted in the lifted $D_i$, and
$vx$ is lifted as $f_\psi^{-\tau_{\psi,N}(vx)}$ of the lift of $\hat x$; we do the same for every point of a path
from $x$ to $x'$ in $\hat U_i$, so that the lifts depend continuously on the point. Points on one orbit
have the same $(\xi,\eta)$, points on one stable fibre the same $\eta$, and points on one unstable plaque
the same $\xi$. So $\eta(vx)=\eta(\hat x)=\eta(x)$, $\xi(vx)=\xi(\hat x)$, and $\xi(\hat x')=\xi(\hat x)$, because
$P^N$ maps the unstable plaque of $vx$ into that of $\hat x$. All these points have $(\xi,\eta)$ in a
product $I\times J$ of small sets with disjoint closures, so they lie in one chart, and along an
orbit $u$ measures $f_\psi$-time.

(i) Since $\tau_\psi(P^ky)=u(P^{k+1}y)-u(P^ky)$ along the lifted orbit, the $n$-th partial sum of $\chi(z)$ is
$[u(z)-u(\pi z)]-[u(P^nz)-u(P^n\pi z)]$. The points $P^nz$ and $P^n\pi z$ lie on a stable fibre $\gamma_n$ in the
leaf $\{\eta=\eta(z)\}$, and they have the $\xi$-coordinates of $z$ and $\pi z$. By Lemma~\ref{lem:chart}(iv),
$u(P^nz)-u(P^n\pi z)=\int_{\gamma_n}\alpha_\psi+B(\xi(z),\eta)-B(\xi(\pi z),\eta)$. Hence the partial sum equals
$\mathcal G(z)-\mathcal G(\pi z)-\int_{\gamma_n}\alpha_\psi$. Since $\alpha_\psi$ is bounded and the length of $\gamma_n$ tends to $0$
exponentially by (M3), (i) follows. The second expression is Lemma~\ref{lem:chart}(iv) on the
stable fibre of $z$.

(ii) By (i), $\chi(\hat x)=\mathcal G(\hat x)-\mathcal G(x)$. Since $vx$ and $\hat x$ lie on one orbit,
$\tau_{\psi,N}(vx)=u(\hat x)-u(vx)$, and $B$ takes the same value at both points. Hence
\[
\tilde\tau_N(vx)=\tau_{\psi,N}(vx)-\chi(\hat x)=\mathcal G(x)-\mathcal G(vx).
\]
This formula defines $\tilde\tau_N\circ v$ on $\hat U_i$. If $p,p'$ lie in a leaf $\{\xi=c\}$, then
Lemma~\ref{lem:chart}(iv) gives $u(p')-u(p)=\int_{[p,p']}\alpha_\psi+A(c,\eta(p'))-A(c,\eta(p))$, so
$\mathcal G(p')-\mathcal G(p)=\int_{[p,p']}\alpha_\psi+\Upsilon(c,\eta(p'))-\Upsilon(c,\eta(p))$. Apply this to $[x,x']$ with $c=\xi(x)$ and to
$[vx,vx']$ with $c=\xi(\hat x)$, and subtract.

(iii) Subtract (ii) for $v_k$ from (ii) for $v_j$. The integrals over $[x,x']$ cancel, and the two
values of $\mathcal Q$ combine by Lemma~\ref{lem:chart}(iii).
\end{proof}

\begin{cor}\label{cor:roof-modulus}
There are $C$ and $d_0>0$ such that for every $N$, every inverse branch $v$ of $T^N$ on $U_i$ and all
$x,x'\in\hat U_i$ with $x'=e^{X^+(u)}(x)$ in the frame $\sigma_i(x)$ and $d=|u|\le d_0$:
\begin{enumerate}
\item[(a)] $|\tilde\tau_N(vx')-\tilde\tau_N(vx)|\le C\,d\,\big(1+\|\bY_{(d)}(\sigma_i(x))\|\big)$.
\item[(b)] Let $v_j,v_k$ be branches, and write $\hat x_j=f^{c'}\big(e^{X^-(w_0)}(\hat x_k)\big)$ in the frame $\sigma_i(\hat x_k)$.
Let $u'$ be defined by $\eta\big(e^{X^+(u')}(\sigma_i(\hat x_k))\big)=\eta(x')$, and let $\theta=\theta_{\sigma_i(\hat x_k),w_0}$ be the phase of
Section~\ref{sec:phase}. Then
\[
\varphi_{jk}(x')-\varphi_{jk}(x)=-\theta(u')+O(\kappa^Nd).
\]
The map $x'\mapsto u'$ is a $C^1$ diffeomorphism near $x$ with derivative and inverse derivative bounded
by $C$.
\end{enumerate}
\end{cor}

\begin{proof}
Since $\alpha_\psi$ is bounded, $|\int_{[x,x']}\alpha_\psi|\le Cd$, and $|\int_{[vx,vx']}\alpha_\psi|\le C\kappa^Nd$ by (M5).

(a) The point $\hat x$ lies on the stable fibre of $x$, so $\hat x=f^c\big(e^{X^-(w_1)}(x)\big)$ in the frame $\sigma_i(x)$, with
$|c|,|w_1|\le C\eps_0$. Let $w$ be the value at $u$ of the gap function of Section~\ref{sec:phase} based at
$\sigma_i(x)$ with $w_0=w_1$. Then $\xi\big(e^{X^-(w)}(e^{X^+(u)}(\sigma_i(x)))\big)=\xi(\hat x)$ and $\eta(e^{X^+(u)}(\sigma_i(x)))=\eta(x')$, so
$\mathcal Q(\xi(x),\xi(\hat x);\eta,\eta')=\Delta_\psi(\sigma_i(x);u,w)$. Lemma~\ref{lem:temporal}, with $|w|\le2|w_1|\le1$ and
$|w|\,\|\bB_{|w|}\|\le C|w|\log(e/|w|)\le C$, gives $|\Delta_\psi(\sigma_i(x);u,w)|\le Cd(1+\|\bY_{(d)}(\sigma_i(x))\|)$. Together
with Lemma~\ref{lem:bridge}(ii), this gives (a).

(b) The point $\hat x_k$ lies on the stable fibre of $x$, so $\eta(\hat x_k)=\eta(x)$, and the backward endpoint of
$e^{X^-(w_0)}(\sigma_i(\hat x_k))$ is $\xi(\hat x_j)$. So $\theta(u')=\mathcal Q\big(\xi(\hat x_k),\xi(\hat x_j);\eta(x),\eta(x')\big)$, and
Lemma~\ref{lem:bridge}(iii) gives (b). The map $x'\mapsto u'$ is the holonomy along weak stable leaves
between two strong unstable plaques at bounded distance, which is smooth with bounded
distortion.
\end{proof}

\begin{rem}[Choice of the section]\normalfont\label{rem:tau-positive}
We may and do assume the following three properties of the section.

\smallskip\noindent\emph{(1) Refinement.} Start from any Markov section as above, and replace each rectangle by
the sets of its points with a prescribed itinerary for the $\ell_1$ preceding and the $\ell_1$ following returns.
These sets are again rectangles. They cover the same set $R$ and have the same first return map $P$ and the
same return times, so $\min\tau_1$ and $\max\tau_1$ do not change. By (M3) their stable and unstable sides have
diameter at most $C\kappa^{\ell_1}$.

\smallskip\noindent\emph{(2) Re-centring.} A refined rectangle $R'$ lies in the old disc, but the strong unstable
plaque through a point $y'\in R'$ is not contained in it. Let $D'=\{e^{X^-(w)}(e^{X^+(u)}(y'))\}$ and slide $R'$ along the
flow onto $D'$; each point moves by a time $c(z)$ with $|c|\le C\kappa^{\ell_1}$. The slide preserves the weak
foliations, so the product structure and the Markov property persist, and the return time changes by the
coboundary $c\circ P-c$, which is at most $2C\kappa^{\ell_1}$. The slid rectangles are pairwise disjoint and the first return map is unchanged,
because two crossings of one orbit are at least $\min\tau_1$ apart while the slides are much smaller. In $D'$, the unstable fibre of $y'$ is the plaque
$W^u(y')\cap D'$, because the gap function with $w_0=0$ is $w\equiv0$, and the stable fibres are the slices
$\{u=\mathrm{const}\}$ (see (D1) below). So the new section has the form used in this section, with discs
$D_i=\{e^{X^-(w)}(e^{X^+(u)}(y_i)):|u|,|w|<\eps_0\}$, $U_i\subset W^u(y_i)$ and $S_i\subset W^s_D(y_i)$, where $\eps_0$ is
arbitrarily small, $\max\tau_1$ is at most the original value plus $1$, and $\min\tau_1$ is at least half of the
original value. In particular the smallness conditions of Section~\ref{sec:spec} can be met without enlarging
$\max\tau_1$.

\smallskip\noindent\emph{(3) Positivity.} By Lemma~\ref{lem:bridge}(i),
$|\chi|\le C\|\alpha_\psi\|_{C^0}\max_i\mathrm{diam}(S_i)\le C\kappa^{\ell_1}$. So we may take $\ell_1$ so large that
$\|\chi\|_\infty\le\frac14\min\tau_\psi$. Then $\tilde\tau\ge\frac12\min\tau_\psi>0$, and the suspension with
roof $\tilde\tau$ is well defined. We assume all this from now on.
\end{rem}

\subsection{The cone and the Lasota--Yorke inequality}

Consider the twisted transfer operators on functions on $U$,
\[
\mathcal L_bh(x)=\sum_{Ty=x}e^{\phi(y)+ib\tilde\tau(y)}h(y),\qquad
\mathcal L_b^Nh(x)=\sum_ve^{\phi_N(vx)+ib\tilde\tau_N(vx)}h(vx),
\]
where the last sum is over the inverse branches $v$ of $T^N$ on the $U_i$ containing $x$. In Section~\ref{sec:mixing} the
operators $\mathcal L_s$ carry $e^{-ib\tilde\tau}$ instead; since every statement below depends only on $|b|$ (and on $|b\,\theta|$), replacing
$b$ by $-b$ changes nothing. Put
\[
\omega(\rho)=\rho\log\frac e\rho\qquad(0<\rho\le1),
\]
and for $x,x'$ in a common $U_i$ with $d(x,x')\le1/\Lambda_b(x)$ write $\rho(x,x')=\Lambda_b(x)\,d(x,x')$.

\begin{definition}\label{def:cone}
For $\mathsf E>0$, let $\mathcal C_{\mathsf E}$ be the set of pairs $(h,H)$ of functions on $U$ with $H>0$ and $|h|\le H$ such
that, for all $x,x'$ in a common $U_i$ with $d(x,x')\le1/\Lambda_b(x)$,
\[
|h(x)-h(x')|\le \mathsf E\,H(x)\,\omega(\rho(x,x')),\qquad |\log H(x)-\log H(x')|\le \mathsf E\,\omega(\rho(x,x')).
\]
\end{definition}

\smallskip\noindent\textbf{Two facts on $\Lambda_b$.} For $x\in U_i$ and $|b|\ge b_0$:
\begin{enumerate}
\item[(F1)] If $z$ lies on the stable fibre of $x$, then $\|\bY_b(\sigma_i(z))-\bY_b(\sigma_i(x))\|\le C_0$. Indeed
$\sigma_i(z)=f^c(e^{X^-(w)}(\sigma_i(x)))$ with $|c|,|w|$ bounded, and $f^t(e^{X^-(w)}(y))=e^{X^-(e^{-2t}w)}(f^t(y))$.
\item[(F2)] If $x'\in U_i$ with $d(x,x')\le1/|b|$, then $\|\bY_b(\sigma_i(x'))-\bY_b(\sigma_i(x))\|\le C_0$, by fact (b) in the
proof of Lemma~\ref{lem:phase}, since $\sigma_i(x')=e^{X^+(u)}(\sigma_i(x))$. Hence $\Lambda_b(x')\le C_\Lambda\Lambda_b(x)$.
\end{enumerate}
By (F2), $\rho(x',x)\le C_\Lambda\rho(x,x')$, and $\omega(C_\Lambda\rho)\le C_\Lambda(1+\log C_\Lambda)\,\omega(\rho)$ for $C_\Lambda\rho\le1$. So the
two inequalities of Definition~\ref{def:cone} also hold with $x$ and $x'$ exchanged, after replacing
$\mathsf E$ by a constant multiple of $\mathsf E$.

\smallskip\noindent\emph{Comparison with \cite{BW20}.} Proposition~\ref{prop:LY} is the analogue of the Lasota--Yorke estimates
\cite[Lemmas~3.4 and~3.5]{BW20}, where $\mathcal L_z^n$ contracts the H\"older seminorm and the sup-norm term carries a factor
$1+|b|^\alpha$. Here the H\"older seminorm is replaced by the cone $\mathcal C_{\mathsf E}$, whose modulus $\omega$ is measured at the
scale $1/\Lambda_b(x)$ instead of a fixed scale.

\begin{prop}\label{prop:LY}
There are $\mathsf E_0$ and $b_0$ and, for every $\mathsf E\ge \mathsf E_0$, an integer $N_0(\mathsf E)$, such that for $\mathsf E\ge \mathsf E_0$,
$N\ge N_0(\mathsf E)$ and $|b|\ge b_0$:
\[
(h,H)\in\mathcal C_{\mathsf E}\ \Longrightarrow\ (\mathcal L_b^Nh,\ \mathcal L_0^NH)\in\mathcal C_{\mathsf E} .
\]
\end{prop}

\begin{proof}
Throughout, $x,x'\in U_i$, $d=d(x,x')$, $\rho=\Lambda_b(x)d\le1$, and $v$ is a branch of length $N$. We use three
inputs.

\emph{(K1) Phase modulus.} $|b|\,|\tilde\tau_N(vx')-\tilde\tau_N(vx)|\le C_1\,\omega(\rho)$. Indeed, $d\le1/|b|$, so by fact
(a) in the proof of Lemma~\ref{lem:phase}, $\|\bY_{(d)}(\sigma_i(x))\|\le\|\bY_b(\sigma_i(x))\|+C\log\frac{1+\|\bY_b(\sigma_i(x))\|}\rho$.
Combining this with Corollary~\ref{cor:roof-modulus}(a), $|b|d=\rho/(1+\|\bY_b\|)$ and $\log(1+y)\le1+y$
gives the claim.

\emph{(K2) Slow variation of $\Lambda_b$ along branches.} $1+\|\bY_b(\sigma(vx))\|\le(1+C_2N)(1+\|\bY_b(\sigma_i(x))\|)$, for
any frame $\sigma(vx)$ over $vx$. The forward orbit of $vx$ reaches $\hat x$ after time $\tau_{1,N}(vx)\le N\max\tau_1$.
Splitting the Birkhoff integral of $vx$ at that time, and using that the Birkhoff integral of $\hat x$ of length $T_b$ differs
from the shorter one by at most $C_{\bF}N\max\tau_1$, we get
$\|\bY_b(\sigma(vx))\|\le\|\bY_b(\sigma_i(\hat x))\|+2C_{\bF}N\max\tau_1$, since norms of Birkhoff integrals do not depend on the frame.
By (F1), $\|\bY_b(\sigma_i(\hat x))\|\le\|\bY_b(\sigma_i(x))\|+C_0$.

\emph{(K3) Lipschitz weights.} $|\phi_N(vx')-\phi_N(vx)|\le C_3d$, by the formula for $\phi_N\circ v$, the
smoothness of $\varrho$ and (M5). Note also that $d=\rho/\Lambda_b(x)\le\omega(\rho)$.

Put $\kappa_N=C(1+C_2N)\kappa^N$ and $\vartheta_N=\kappa_N(1+\log(1/\kappa_N))$, so that $\vartheta_N\to0$. By (K2) and (M5),
$\rho_v:=\Lambda_b(vx)\,d(vx,vx')\le\kappa_N\rho$. Since $\log(e/\rho)\ge1$,
\[
\omega(\kappa\rho)=\kappa\rho\Big(\log\frac e\rho+\log\frac1\kappa\Big)\le\kappa\Big(1+\log\frac1\kappa\Big)\omega(\rho)\qquad(\kappa\le1),
\]
so $\omega(\rho_v)\le\vartheta_N\,\omega(\rho)$. Take $N$ so large that $\kappa_N\le1$ and $\mathsf E\vartheta_N\le\log2$. The points $vx,vx'$
lie in a common $U_j$, so the cone condition applies to them and gives $H(vx')\le2H(vx)$.

\emph{The $h$-part.} Write $w_v(y)=e^{\phi_N(vy)}$. Each term of $\mathcal L_b^Nh(x')-\mathcal L_b^Nh(x)$ is bounded by
\[
|w_v(x')-w_v(x)|\,|h(vx')|+w_v(x)\,|b|\,|\tilde\tau_N(vx')-\tilde\tau_N(vx)|\,|h(vx')|+w_v(x)\,|h(vx')-h(vx)| .
\]
By (K3), (K1) and the cone condition at $vx$, this is at most
$w_v(x)H(vx)\big(4C_3+2C_1+\mathsf E\vartheta_N\big)\omega(\rho)$. Summing over $v$,
\[
|\mathcal L_b^Nh(x')-\mathcal L_b^Nh(x)|\le\big(4C_3+2C_1+\mathsf E\vartheta_N\big)\,\omega(\rho)\,\mathcal L_0^NH(x).
\]

\emph{The $H$-part.} The quotient of $\mathcal L_0^NH(x')$ by $\mathcal L_0^NH(x)$ is a weighted average of the
quotients of $w_v(x')H(vx')$ by $w_v(x)H(vx)$. Hence
\[
|\log\mathcal L_0^NH(x')-\log\mathcal L_0^NH(x)|\le C_3d+\mathsf E\vartheta_N\omega(\rho)\le(C_3+\mathsf E\vartheta_N)\,\omega(\rho).
\]

Clearly $|\mathcal L_b^Nh|\le\mathcal L_0^NH$. Choose $\mathsf E_0=2(4C_3+2C_1)$ and, for $\mathsf E\ge \mathsf E_0$, choose $N_0(\mathsf E)$ such that
$\mathsf E\vartheta_N\le\min\{\log2,\mathsf E/2\}$ for $N\ge N_0(\mathsf E)$. Then both cone inequalities hold with constant $\mathsf E$.
\end{proof}

\begin{rem}\normalfont\label{rem:LY-next}
A $C^1$ function satisfies $|h(x)-h(x')|\le(\mathrm{Lip}(h)/|b|)\,\omega(\rho)$, since $d\le\rho/|b|\le\omega(\rho)/|b|$. So
$(h,1)\in\mathcal C_{\mathsf E}$ whenever $\|h\|_\infty\le1$ and $\mathrm{Lip}(h)\le \mathsf E|b|$.
\end{rem}

\section{Specification with directions}\label{sec:spec}

This section proves the specification lemma (Lemma~\ref{lem:spec}). We construct two families of branch pairs along the Markov section, one whose suffix shadows a segment with a long backward Birkhoff integral of large norm and one with a short Birkhoff integral, whose stable gaps point in all prescribed directions. This is where Theorem~\ref{thm:rigidity} and the operator norm of a Birkhoff integral are used to produce the non-integrability.

Throughout this section $\psi$ is not of contact type. Let $C_{\bF}=\sup\|\bF\|$. Recall that
$\bB_r(xk)=k^{-1}\bB_r(x)k$ for $k\in M_0$, so $\|\bB_r\|$ is a function on $SM$.

\begin{definition}\label{def:family}
Let $\eps\in(0,1]$, $\delta\in(0,\frac14]$, $\delta_1\in(0,\frac14]$ and $0<\bar r_1\le\bar r_2\le\frac12$. An
\emph{$(\eps,\delta,\delta_1)$-family of length $N$ with gaps $(\bar r_1,\bar r_2)$} consists of the following
data. For each $i$, each $p\in\{1,2\}$ and each $k\in\{1,\dots,K_1\}$, there are two inverse branches
$v_{i,p,k}$ and $v'_{i,p,k}$ of $T^N$ on $U_i$. All of them are distinct, so their images have disjoint
interiors. For every $x\in U_i$ put $\hat x_{p,k}=P^Nv_{i,p,k}x$ and $\hat x'_{p,k}=P^Nv'_{i,p,k}x$. The data
must satisfy:
\begin{enumerate}
\item[(a)] \emph{Gaps.} Both points lie on the stable fibre of $x$, and
$\hat x_{p,k}=f^{c}\big(e^{X^-(w_{p,k}(x))}(\hat x'_{p,k})\big)$ in the frame $\sigma_i(\hat x'_{p,k})$. Here $|c|$ is
bounded and $|w_{p,k}(x)|\in[\bar r_p/C_*,C_*\bar r_p]$.
\item[(b)] \emph{Directions.} For each $p$, the unit vectors $\hat w_{p,k}(x)$, $k=1,\dots,K_1$, form a
$\delta_1$-net of the unit sphere of $\RR^{n-1}$.
\item[(c)] \emph{Common backward Birkhoff integrals.} With $\bB_{p,k}(x)=\bB_{|w_{p,k}(x)|}(\sigma_i(\hat x'_{p,k}))$,
we have $\|\bB_{p,k}-\bB_{p,k'}\|\le C_*$ for all $k,k'$.
\item[(d)] \emph{Separation.} With $\beta=\max_{p,k}\|\bB_{p,k}(x)\|$, for all $k$ and $k'$,
\[
\|\bB_{1,k}(x)-\bB_{2,k'}(x)\|\ge9\delta\beta+C_A\log\frac{e(2+\beta)}\eps,
\]
where $C_A\ge1$ is a constant, fixed in \eqref{eq:CA} below.
\end{enumerate}
The constants $C_*$ and $K_1$ must be independent of $\eps$, $\bar r_1$ and $N$.
\end{definition}

\smallskip\noindent\textbf{Coordinates on the discs.} Write $D_c=\{\Xi_c(u,w):|u|,|w|<\eps_0\}$, where
$\Xi_c(u,w)=e^{X^-(w)}(e^{X^+(u)}(\tilde y_c))$ and $\tilde y_c$ is the frame over $y_c$ used in Section~\ref{sec:markov}.
\begin{enumerate}
\item[(D1)] \emph{Stable fibres are coordinate slices.} The forward endpoint of $\Xi_c(u,w)$ equals that of
$e^{X^+(u)}(\tilde y_c)$, and $u\mapsto(e^{X^+(u)}(\tilde y_c))^+$ is injective. So the stable fibres of $D_c$ are the slices
$\{u=\text{const}\}$, $U_c\subset\{w=0\}$ and $S_c\subset\{u=0\}$. The adapted frame is
$\sigma_c(\Xi_c(u,w))=e^{X^-(w)}(e^{X^+(u)}(\tilde y_c))$. In particular, if $z,z_*$ lie on a common stable fibre, then
\[
\sigma_c(z_*)=e^{X^-(w(z_*)-w(z))}\big(\sigma_c(z)\big),
\]
with no flow component.
\item[(D2)] \emph{Unstable holonomy.} The unstable fibre through $\Xi_c(0,s)$ meets the slice $\{u=u_1\}$ at
$\Xi_c(u_1,H(u_1,s))$. Here $H(u_1,\cdot)$ is the gap function of Section~\ref{sec:phase} based at $\tilde y_c$, that is,
$e^{X^-(H)}(e^{X^+(u_1)}(\tilde y_c))$ and $e^{X^-(s)}(\tilde y_c)$ have the same backward endpoint. By \eqref{eq:gap} and the bound on
$\partial_w\tilde w$ there,
\[
\|\partial_sH(u_1,s)-I\|\le C_H|u_1||s|\le C_H\eps_0^2 .
\]
\end{enumerate}
We assume that the section was chosen with $\eps_0$ so small that $C_H\eps_0^2\le\frac1{64}$ and
$e^{2\max\tau_1}C_LC\eps_0\le r_0$, for the constants $r_0$ and $C_L$ of (D3) below. This is possible by Remark~\ref{rem:tau-positive}(2), which makes $\eps_0$ arbitrarily small without
enlarging $\max\tau_1$ by more than $1$.
\begin{enumerate}
\item[(D3)] \emph{Flow projection of unstable plaques.} There are $r_0,C_L$ such that for $z\in D_i$ and $|A|\le r_0$,
the orbit of $e^{X^+(A)}(z)$ (in a frame over $z$) crosses $D_i$ at a unique point $z(A)$ near $z$, which lies on the
unstable fibre of $z$, and $|A|\le C_L\,d(z,z(A))$. This holds because $A\mapsto z(A)$ is a diffeomorphism onto a
neighbourhood of $z$ in its unstable fibre, with derivative bounded below uniformly in $z$, by
compactness and transversality of $D_i$ to the flow.
\end{enumerate}

\begin{lemma}[Symbolic facts]\label{lem:symbolic}
\begin{enumerate}
\item[(S1)] Let $v,v'$ be inverse branches of $T^N$ on $U_i$ whose words end with the same $k$ symbols,
and let $x\in U_i$. Then for $0\le j\le k$ the points $z_j=P^{N-j}(vx)$ and $z'_j=P^{N-j}(v'x)$ lie on a
common stable fibre.
\item[(S2$'$)] For every $\delta_1\in[\frac18,1]$ there are $\ell_0$, $g>0$ and, for each rectangle $R_c$ and each unit
vector $\hat\rho$, words $e\ne e'$ of length $\ell_0$ that can be followed by $c$, with the following property.
Let $\Sigma_e\subset R_c$ be the set of points whose last $\ell_0$ preimages under $P$ follow $e$. If $z$ and $z_*$ lie
on a common stable fibre of $R_c$, with $z_*\in\Sigma_e$ and $z\in\Sigma_{e'}$, then $z_*=e^{X^-(\rho)}(z)$ in the frame
$\sigma_c(z)$, with $g\le|\rho|\le C$ and $|\rho/|\rho|-\hat\rho|\le\delta_1/2$.
\item[(S3$'$)] \emph{Shadowing.} There is $C_{\rm sh}$ such that the following holds. Let $y,y'\in R_c$
have the same itinerary $c=c_0,c_1,\dots,c_m$ under $P$, and let $\tilde y,\tilde y'\in FM$ be frames over
them. Then there is $k\in M_0$ such that, with $\tau=\tau_{1,m}(y)$,
\[
\Big\|\int_0^{\tau}\bF(f^t\tilde y')\,dt-k^{-1}\Big(\int_0^{\tau}\bF(f^t\tilde y)\,dt\Big)k\Big\|\le C_{\rm sh}.
\]
In particular the norms of the two Birkhoff integrals differ by at most $C_{\rm sh}$. If moreover $y'$ lies on
the stable fibre of $y$ and $\tilde y'=f^s(e^{X^-(w)}(\tilde y))$, the same holds with $k=e$.
\end{enumerate}
The constants depend only on the section and on $\|\psi\|_{C^3}$.
\end{lemma}

\begin{proof}
(S1) By the description of $v(U_i)$, $\pi(z_0)=T^Nvx=x=\pi(z'_0)$, so $z_0$ and $z'_0$ lie on the stable fibre
of $x$. Suppose $z_j$ and $z'_j$ lie on a common stable fibre $W$ for some $j<k$. Since the words agree at
the place $N-j-1$, the points $z_{j+1}$ and $z'_{j+1}$ lie in a common rectangle $R_a$. Put
$w=[z_{j+1},z'_{j+1}]$, using (M1). By (M2), $Pw\in P(W^s_R(z_{j+1}))\subset W$. Also $Pw$ lies in
$P(W^u_R(z'_{j+1}))$, a connected piece of a weak unstable leaf in the disc containing $Pz'_{j+1}=z'_j$,
whose intersection with the rectangle of $z'_j$ is the unstable fibre of $z'_j$, by (M1). Thus $Pw$ and
$z'_j$ lie on $W$ and on the unstable fibre of $z'_j$, so $Pw=z'_j$ by (M1). Since $P$ is injective,
$w=z'_{j+1}$, that is, $z'_{j+1}$ lies on the stable fibre of $z_{j+1}$.

(S2$'$) \emph{Strips.} Let $w\in\Sigma_e$ and $w'\in W^u_R(w)$. By (M2) applied to $P^{-1}w$,
$W^u_R(w)\subset P(W^u_R(P^{-1}w))$, so $P^{-1}w'\in W^u_R(P^{-1}w)$, which lies in the same rectangle as $P^{-1}w$. By
induction, the last $\ell$ preimages of $w'$ follow $e$. Hence $\Sigma_e$ is a union of full unstable fibres, of
the form $[U_c,S_e]$ with $S_e\subset S_c$. Moreover $S_e$ is contained in the image under $P^\ell$ of one stable
fibre, so its diameter is at most $C_1\kappa^\ell$ by (M3). Distinct words give sets $S_e$ with disjoint
interiors, and every point of $S_c$ has a past, so the sets $S_e$, for words of length $\ell$, cover $S_c$. By
(D2), the trace of $\Sigma_e$ on the slice $\{u=u_1\}$ is $H(u_1,S_e)$.

\emph{Choice of words.} In the coordinate $w$ of $\{u=0\}$, $S_c$ is the closure of its interior, so it contains
a ball $B(y_0,4g)$. Given $\hat\rho$, put $y_1=y_0+2g\hat\rho$. Choose $\ell_0$ with $C_1\kappa^{\ell_0}\le\delta_1g/16$, and words $e,e'$
of length $\ell_0$ with $y_1\in S_e$ and $y_0\in S_{e'}$. Then $e\ne e'$, because $|y_1-y_0|=2g$ exceeds the diameter of
$S_e$.

\emph{The gap.} Let $z_*\in\Sigma_e$ and $z\in\Sigma_{e'}$ lie on the slice $\{u=u_1\}$. Then $z_*=\Xi_c(u_1,H(u_1,s_*))$ and
$z=\Xi_c(u_1,H(u_1,s))$ with $s_*\in S_e$ and $s\in S_{e'}$. By (D1), $z_*=e^{X^-(\rho)}(z)$ in the frame $\sigma_c(z)$, where
$\rho=H(u_1,s_*)-H(u_1,s)$. By the mean value inequality and (D2),
$|\rho-(s_*-s)|\le C_H\eps_0^2|s_*-s|\le|s_*-s|/64$. Also $|(s_*-s)-2g\hat\rho|\le2C_1\kappa^{\ell_0}\le\delta_1g/8$, so
$|s_*-s|\le3g$ and
\[
|\rho-2g\hat\rho|\le\tfrac{\delta_1g}8+\tfrac{3g}{64}\le\tfrac{\delta_1g}2,
\]
using $\delta_1\ge\frac18$. Hence $|\rho|\ge g$, $|\rho|\le C$, and
$|\rho/|\rho|-\hat\rho|\le2|\rho-2g\hat\rho|/(2g)\le\delta_1/2$. The constants $\ell_0$ and $g$ do not depend on $\hat\rho$ or $c$.

(S3$'$) We work in a lift to $G$ of a neighbourhood of the orbit segments involved. Write $t_k(z)$ for the
$f$-time of the $k$-th return of $z$ to $R$.

\emph{Step 1: the bracket point.} Put $w=[y,y']\in R_c$. By induction on $k$, using (M1) and (M2) as in the proof
of (S1), $P^kw=[P^ky,P^ky']$ for $0\le k\le m$. So $P^kw$ lies on the stable fibre of $P^ky$ and on the unstable
fibre of $P^ky'$, all in $R_{c_k}$. By (M3), $d(P^kw,P^ky)\le C\kappa^k$, so by (M4)
$|t_m(w)-t_m(y)|\le\sum_{k<m}|\tau_1(P^kw)-\tau_1(P^ky)|\le C\sum_k\kappa^k\le C$.

\emph{Step 2: $y$ and $w$.} By (D1), $\sigma_c(w)=e^{X^-(r)}(\sigma_c(y))$ with $|r|\le2\eps_0$. Since
$f^t(e^{X^-(r)}(\sigma_c(y)))=e^{X^-(e^{-2t}r)}(f^t(\sigma_c(y)))$ and $\bF$ is Lipschitz for the left-invariant metric,
$\|\bF(f^t\sigma_c(w))-\bF(f^t\sigma_c(y))\|\le Ce^{-2t}$, so the two Birkhoff integrals over $[0,\tau]$ differ by at most $C$.

\emph{Step 3: $w$ and $y'$.} These points lie on a common unstable fibre, so they have the same backward
endpoint, and $\sigma_c(y')=f^s(e^{X^+(a)}(\sigma_c(w)))\,k_1$ with $k_1\in M_0$ and $|a|,|s|\le C\eps_0$, by the description of $W^{cu}$ as the union of the manifolds $W^u(f^s(\cdot)k)$. By \eqref{eq:conj},
\[
f^t\big(\sigma_c(y')\big)=f^s(e^{X^+(e^{2t}a)}(f^t\big(\sigma_c(w)\big)))\,k_1 .
\]
Put $A_k=e^{2t_k(w)}a$. We claim $|A_k|\le r_1:=C_LC\eps_0$ for $0\le k\le m$, where $C\eps_0$ bounds the diameters of the
unstable fibres. For $k=0$ this holds. If it holds for $k$, then $|A_{k+1}|\le e^{2\max\tau_1}r_1\le r_0$. The orbit of
$y'$ is the orbit of $e^{X^+(A_{k+1})}(P^{k+1}w)$, and it crosses $D_{c_{k+1}}$ near $P^{k+1}w$ at $P^{k+1}y'$, which lies on the
unstable fibre of $P^{k+1}w$ at distance at most $C\eps_0$. By (D3), $|A_{k+1}|\le C_LC\eps_0=r_1$. This proves the claim,
and with $k=m$ it gives $|a|\le r_1e^{-2t_m(w)}$.

For $0\le t\le t_m(w)$ put $p=f^t(\sigma_c(w))$ and $q=f^s(p)\,k_1$. By \eqref{eq:conj},
\[f^t\sigma_c(y')=f^s\big(e^{X^+(e^{2t}a)}(p)\big)k_1=e^{X^+(v')}(q),\qquad v'=e^{2s}k_1^{-1}e^{2t}a,\]
while $f^{t+s}(\sigma_c(w))\,k_1=q$. Since $\sigma\mapsto e^{X^+(\sigma v')}(q)$ has velocity of bounded norm,
\[
d\big(f^t\sigma_c(y'),\,f^{t+s}(\sigma_c(w))\,k_1\big)=d\big(e^{X^+(v')}(q),q\big)\le C|v'|\le Ce^{2t}|a|\le Cr_1e^{-2(t_m(w)-t)} .
\]
Since $\bF(xk_1)=k_1^{-1}\bF(x)k_1$, integrating over $[0,t_m(w)]$ and accounting for the shift $|s|\le C\eps_0$ gives
\[
\Big\|\int_0^{t_m(w)}\bF(f^t\sigma_c(y'))\,dt-k_1^{-1}\Big(\int_0^{t_m(w)}\bF(f^t\sigma_c(w))\,dt\Big)k_1\Big\|\le C .
\]

\emph{Step 4: conclusion.} By Step~1, replacing $t_m(w)$ by $\tau=t_m(y)$ changes both Birkhoff integrals by at most
$C_{\bF}|t_m(w)-\tau|\le C$. Combining with Step~2,
\[
\int_0^\tau\bF(f^t\sigma_c(y'))\,dt=k_1^{-1}\Big(\int_0^\tau\bF(f^t\sigma_c(y))\,dt\Big)k_1+O(1).
\]
For arbitrary frames $\tilde y=\sigma_c(y)k_y$
and $\tilde y'=\sigma_c(y')k_{y'}$, the claim holds with $k=k_y^{-1}k_1k_{y'}$. Finally, if $\tilde y'=f^s(e^{X^-(w)}(\tilde y))$, then
$f^t(\tilde y')=e^{X^-(e^{-2(t+s)}w)}(f^{t+s}(\tilde y))$, and the computation of Step~2, with the time shift $s$, gives the
claim with $k=e$.
\end{proof}

\begin{lemma}[Specification]\label{lem:spec}
Assume that $\psi$ is not of contact type, and put $\delta=\frac1{18}$. For every $\delta_1\in[\frac18,\frac14]$
there are $L_0$ and $K_1=K_1(\delta_1)$, and for every $\eps\in(0,1]$ there are $L=L(\eps)\ge L_0$ and
$N_1(\eps,\delta_1)$, such that for every $N\ge N_1$ there is an $(\eps,\delta,\delta_1)$-family of length
$N$ with gaps $(e^{-2L},e^{-2L_0})$.
\end{lemma}

\begin{proof}
\emph{A long Birkhoff integral.} Let $C_{\mathrm b}$ be a bound for $\|\bB_{p,k}\|$ over the bounded family $p=2$
constructed below; it does not depend on $\eps$. Choose $K_*(\eps)$ such that every $B\ge K_*(\eps)$
satisfies
\begin{equation}\label{eq:Kstar}
\tfrac12B\ge\tfrac32\big(C_*+C_{\mathrm b}\big)+C_A\log\frac{e(2+B+C_*+C_{\mathrm b})}\eps .
\end{equation}
By Corollary~\ref{cor:unbounded}, there are $\tilde x\in FM$ and $L'>0$ with
$\|\int_0^{L'}\bF(f^t\tilde x)\,dt\|\ge K_*(\eps)+C_{\mathrm w}$, where $C_{\mathrm w}$ is fixed below. Moving $\tilde x$ along its
orbit by a bounded time, and then slightly, changes this norm by $O(1)$. So we may assume that
$\tilde x$ lies over the interior of a rectangle and has a unique itinerary. Let $s_\star$ be its word of
the first $m$ returns, with $\tau_{1,m}\in[L'-C,L']$.

\emph{Words.} By irreducibility of the transition matrix there are connecting words of bounded
length, in the appropriate residue classes modulo $d_0$. Put $s_{i,1}=g_1s_\star g_{2,i}$, and let $s_{i,2}=g'_{2,i}$ be a fixed word of bounded length whose
return time is at least $L_0$, where $L_0$ is a constant chosen so that the gaps below are at most
$\frac12$. Let $c_p$ be the first symbol of $s_{i,p}$. Let $\{\hat\rho_k\}_{k\le K_1}$ be a $\delta_1/2$-net, and
let $(e_{c_p,k},e'_{c_p,k})$ be the words given by (S2$'$) for $\hat\rho_k$. Choose $4K_1$ distinct words
$\pi_{p,k},\pi'_{p,k}$ of a fixed length. Let $v_{i,p,k}$ and $v'_{i,p,k}$ have the words
\[
\pi_{p,k}\,f_{p,k}\,e_{c_p,k}\,s_{i,p}\qquad\text{and}\qquad\pi'_{p,k}\,f'_{p,k}\,e'_{c_p,k}\,s_{i,p},
\]
with connecting fillers $f_{p,k},f'_{p,k}$ of the right length. These words are pairwise distinct
because their first blocks are.

\emph{Gaps and directions.} Fix $p$, and let $z_k$ and $z_{*k}$ be the points at the start of the
common suffix, so that $\hat x'_{p,k}=f^{t_k}(z_k)$ with $t_k$ the suffix time. By (S1), all the
$z_k$ and $z_{*k}$ lie on one stable fibre of $R_{c_p}$. By (S2$'$), $z_{*k}=e^{X^-(\rho_k)}(z_k)$ in
$\sigma_{c_p}(z_k)$, with $|\rho_k/|\rho_k|-\hat\rho_k|\le\delta_1/2$. By \eqref{eq:conj},
\[
\hat x_{p,k}=f^{O(1)}\big(e^{X^-(e^{-2t_k}\rho_k)}(\hat x'_{p,k})\big)
\]
in the frame $f^{t_k}(\sigma_{c_p}(z_k))$. Write $f^{t_k}(\sigma_{c_p}(z_k))=\sigma_i(\hat x'_{p,k})\,m_k$ with $m_k\in M_0$.
The frames $\sigma_{c_p}(z_k)$ are related by right multiplication by elements of the subgroup $H\subset G$ with Lie algebra $E^s\oplus E^c$ (that is, by moves along $W^s$ and the flow), and so are the frames $\sigma_i(\hat x'_{p,k})$, because the
$\hat x'_{p,k}$ all lie on the stable fibre of $x$. Since $M_0$ normalizes $H$ and $H\cap M_0=\{e\}$,
we get $m_k=m_1$ for all $k$. Hence $\hat w_{p,k}(x)=m_1\hat\rho_k+O(\delta_1/2)$, a $\delta_1$-net. Also
$|w_{p,k}|=e^{-2t_k}|\rho_k|$. With $L$ the suffix time of $s_{i,1}$ (up to $O(1)$, uniformly in $i$ and
$x$), this gives $|w_{1,k}|\asymp e^{-2L}$ and $|w_{2,k}|\asymp e^{-2L_0}$. This proves (a) and (b).

\emph{Birkhoff integrals.} The backward Birkhoff integral of $\hat x'_{p,k}$ has length
$T(|w_{p,k}|)=t_k-\frac12\log|\rho_k|=t_k+O(1)$, which is the suffix time up to $O(1)$. For fixed $p$ the
points $\hat x'_{p,k}$ come from the points $z_k$, which lie on one stable fibre of $R_{c_p}$ and follow
the same suffix. Apply the last part of (S3$'$) to the frames $\sigma_{c_p}(z_k)$, which are related by
$H$. Since $\sigma_i(\hat x'_{p,k})=f^{t_k}(\sigma_{c_p}(z_k))\,m_1^{-1}$ with the same $m_1$ for all $k$, the
matrices $\bB_{p,k}$ agree up to $O(1)$. This gives (c). For $p=2$ the integration times are bounded, so $\|\bB_{2,k}\|\le C_{\mathrm b}$. For
$p=1$, the part of the backward orbit of $\hat x'_{1,k}$ that follows $s_\star$ has the same itinerary as
$\tilde x$. The connecting words and the end effects contribute $O(1)$. By (S3$'$),
\[
\|\bB_{1,k}(x)\|\ge\Big\|\int_0^{\tau_{1,m}}\bF(f^t\tilde x)\,dt\Big\|-C_{\mathrm w}\ge K_*(\eps),
\]
for a suitable $C_{\mathrm w}$ independent of $\eps$. This is where we fix $C_{\mathrm w}$.

\emph{Separation.} Put $B=\|\bB_{1,1}(x)\|\ge K_*(\eps)$. By (c), $\beta\le\max(B+C_*,C_{\mathrm b})\le B+C_*+C_{\mathrm b}$ and
$\|\bB_{1,k}-\bB_{2,k'}\|\ge B-C_*-C_{\mathrm b}$. With $\delta=\frac1{18}$, $9\delta\beta\le\frac12(B+C_*+C_{\mathrm b})$. So (d)
follows from \eqref{eq:Kstar}.
\end{proof}

\begin{rem}\normalfont
(1) The proof only uses norms of backward Birkhoff integrals, which are functions on $SM$. The frame
holonomy along the long suffix, the rotation $m_1$ above, is not controlled and is not needed. The
directions are recovered afterwards. For each $x$, Lemma~\ref{lem:phase}(A) produces the singular
vectors $(\hat u,\hat w)$ of $\bY_b+\bB_p$, and (b) supplies a branch pair whose gap direction is within
$\delta_1$ of $\hat w$.

(2) No closed geodesics, no Liv\v{s}ic theorem and no combinatorics on words are used.
\end{rem}

\section{Cancellation and the $L^2$ estimate}\label{sec:L2}

This section proves the $L^2$ contraction. We cut the section into cells of size $\eps/\Lambda_b$, choose in each cell a pair of branches and a direction from Section~\ref{sec:spec}, obtain cancellation from the phase of Section~\ref{sec:phase}, and assemble Dolgopyat's operators, which give Corollary~\ref{cor:dolgopyat}, the case $\sigma=0$ of Proposition~\ref{prop:localfreq}.

Throughout this section $\psi$ is positive and not of contact type, and we fix the section, the
adapted frames and the constants $C_0$, $C_\Lambda$ of (F1) and (F2), and $C_1,C_2,C_3$ of the proof of
Proposition~\ref{prop:LY}. Note that (F2) also holds, with a larger constant, for
$d(x,x')\le C/|b|$, because fact (b) in the proof of Lemma~\ref{lem:phase} is linear in $|b||u|$.

\subsection{Cells}

A \emph{cylinder} of length $k\ge0$ is a set $v(U_j)$, where $v$ is an inverse branch of $T^k$ on $U_j$; for
$k=0$ it is $U_j$ itself. If $k\ge1$ and $v$ has the word $a_0\cdots a_{k-1}$, its \emph{parent} is the
cylinder of length $k-1$ with the word $a_0\cdots a_{k-2}$, defined on $U_{a_{k-1}}$; it contains $v(U_j)$.

\begin{lemma}\label{lem:cylinders}
There are $c_*\in(0,1]$ and $C$ such that:
\begin{enumerate}
\item[(i)] every cylinder $Q$ contains a ball $B_Q=B(x_Q,c_*\,\mathrm{diam}\,Q)$, and if $Q$ has a parent
$Q'$, then $\mathrm{diam}\,Q\ge c_*\,\mathrm{diam}\,Q'$;
\item[(ii)] for every ball $B'\subset Q$ of radius $r$, $\mu(B')\ge c_*\,(r/\mathrm{diam}\,Q)^{n-1}\mu(Q)$.
\end{enumerate}
\end{lemma}

\begin{proof}
(i) The first claim is the last statement of Lemma~\ref{lem:M5}. For the second, write $v=v'\circ w$,
where $v'$ is the branch of the parent and $w$ is an inverse branch of $T$ from $U_j$ into $U_{a_{k-1}}$.
There are finitely many $w$, and each $w(U_j)$ contains a ball of a fixed radius $r_1$. By (M5),
$\|Dv'\|$ varies on $\hat U_{a_{k-1}}$ by a bounded factor, and $Dv'$ is quasi-conformal, so $v'$ maps a
ball of radius $r_1$ onto a set containing a ball of radius $c\,r_1\|Dv'\|$, while
$\mathrm{diam}\,Q'\le C\|Dv'\|$. Hence $\mathrm{diam}\,Q\ge c\,\mathrm{diam}\,Q'$.

(ii) The density of $\mu$ with respect to the Riemannian volume of the leaves is bounded above and below on
$U$. The volume of $B'$ is $c_nr^{n-1}$, and the volume of $Q$ is at most $C(\mathrm{diam}\,Q)^{n-1}$.
\end{proof}

Fix $\eps\in(0,1]$ and $|b|\ge b_0$, where $b_0$ is large. A cylinder $Q\subset U$ is \emph{admissible} if
$\mathrm{diam}\,Q\le\eps/\Lambda_b(y)$ for every $y\in Q$. A child of an admissible cylinder is admissible, and
the cylinders of length $0$ are not admissible once $b_0$ is large. A \emph{cell} is an admissible
cylinder whose parent is not admissible.

\smallskip\noindent\emph{Comparison with \cite{BW20}.} The cells play the role of the partition of unity at scale $|b|^{-1/2}$
in \cite[Lemma~3.10]{BW20}. Two things differ: the scale $\eps/\Lambda_b(y)$ depends on the point, and the cells are cylinders
of $T$, which cover $U$ up to a $\mu$-null set, so no partition of unity is used.

\begin{lemma}\label{lem:cells}
The cells cover $U$ up to a $\mu$-null set, and distinct cells have disjoint interiors. There is
$c_Q\in(0,c_*]$, independent of $b$ and $\eps$, such that every cell $Q$ satisfies
\[
c_Q\,\frac\eps{\Lambda_b(y)}\le\mathrm{diam}\,Q\le\frac\eps{\Lambda_b(y)},\qquad \Lambda_b(y)\le C_\Lambda\Lambda_b(y')\qquad(y,y'\in Q).
\]
\end{lemma}

\begin{proof}
Let $x\in U$ be a point whose forward $T$-orbit avoids the boundaries of the $U_j$; these points have
full measure. The cylinders containing $x$ form a decreasing sequence, of diameters tending to $0$ by
(M5). Since $\Lambda_b\le C|b|\log|b|$, the first admissible one is a cell. Two cells are either nested
or have disjoint interiors. If $Q_1\subsetneq Q_2$ are cells, then the parent of $Q_1$ is contained in
$Q_2$, hence admissible, a contradiction.

The upper bound is admissibility, and the comparison of $\Lambda_b$ follows from (F2), since
$\mathrm{diam}\,Q\le1/|b|$. Let $Q'$ be the parent of $Q$. It is not admissible, so
$\mathrm{diam}\,Q'>\eps/\Lambda_b(y')$ for some $y'\in Q'$. By Lemma~\ref{lem:cylinders}(i),
$\mathrm{diam}\,Q'\le c_*^{-1}\mathrm{diam}\,Q\le c_*^{-1}/|b|$, so (F2) gives $\Lambda_b(y')\le C\Lambda_b(y)$ for $y\in Q$. Hence
$\mathrm{diam}\,Q\ge c_*\,\mathrm{diam}\,Q'\ge c_*\eps/(C\Lambda_b(y))$.
\end{proof}

\subsection{Choice of constants}

Let $C_N$ be the constant $C$ of \eqref{eq:N-hyp} in Lemma~\ref{lem:phase}(N), and let $C_*,K_1$ be the
constants of Definition~\ref{def:family} for the families of Lemma~\ref{lem:spec}, with $\delta=\frac1{18}$ and
$\delta_1=\frac14$; they do not depend on $C_A$. Put
\begin{equation}\label{eq:CA}
C_A'=C_N\big(1+\log(1+2C_0)\big)+3C_0+C_*,\qquad C_A=C_{(A)}(2C_0,C_A'),
\end{equation}
with $C_{(A)}$ as in Lemma~\ref{lem:phase}(A), and use $C_A$ in Definition~\ref{def:family}(d); Lemma~\ref{lem:spec}
holds for any choice of $C_A$. Let
$C_u\ge1$ be a bound for the derivative and the inverse derivative of the maps $x'\mapsto u'$ of
Corollary~\ref{cor:roof-modulus}(b). The constants are chosen in the following order.
\begin{enumerate}
\item Fix $\mathsf E\ge\max\{\mathsf E_0',8(C_3+1)\}$, where $\mathsf E_0'\ge \mathsf E_0$ is the constant of Proposition~\ref{prop:LY}, enlarged
as in Proposition~\ref{prop:complex-s}(i) so that it also covers complex $s$.
\item Fix $\eps\in(0,1]$ with $C_3\eps\le\frac1{10}$, and put $\eps_1=c_Q^2\eps/(4C_u(1+C_0))$.
\item Take $L_0$ and $L=L(\eps_1)$ from Lemma~\ref{lem:spec}, and put
$\theta_0=\delta\eps_1e^{-2L}/(4C_*)$. The smaller gap $e^{-2L}$ determines $\theta_0$.
\item Choose $\rho_*\in(0,\eps_1]$ such that the bound of Lemma~\ref{lem:phase}(R) is at most $\theta_0/8$
whenever $\rho\le2\rho_*$ and $|w_0|\le C_*e^{-2L_0}$.
\item Choose $N\ge\max\{N_0(\mathsf E),N_1(\eps_1,\delta_1)\}$ such that $\mathsf E\vartheta_N\le\theta_0/256$ and
$C\kappa^N\le\theta_0/32$, where $C$ is the constant of Corollary~\ref{cor:roof-modulus}(b), and fix the
family of Lemma~\ref{lem:spec} of length $N$.
\item Put $\eta_0=\frac18\big(1-\cos(\theta_0/32)\big)$ and $\eta=\min\{\eta_0,\frac14,\mathsf E\rho_*/(4C_T)\}$, where $C_T$ is the
constant in \eqref{eq:tent-lip} below.
\item Take $b_0$ so large that Proposition~\ref{prop:LY}, Lemma~\ref{lem:phase} with gaps
$|w_0|\ge e^{-2L}/C_*$, and the lemmas above apply.
\end{enumerate}

\subsection{The pair and the direction in a cell}

Let $|b|\ge b_0$ and let $Q\subset U_i$ be a cell, with ball $B_Q=B(x_Q,c_*\mathrm{diam}\,Q)$. Write
$\hat x'_{p,k}=\hat x'_{p,k}(x_Q)$, $\bB_{p,k}=\bB_{p,k}(x_Q)$, and so on.
\begin{enumerate}
\item Put $v_p=\sigma_i(\hat x'_{p,1})$ for $p=1,2$. These points lie on the stable fibre of $x_Q$, so
$\|\bY_b(v_1)-\bY_b(v_2)\|\le2C_0$ by (F1). By Definition~\ref{def:family}(d), Lemma~\ref{lem:phase}(A) with
$\eps_1$, $2C_0$ in place of $C_0$, $C'=C_A'$ and $\bB_p=\bB_{p,1}$ gives $p=p(Q)$ such that
$\mathbf M=\bY_b(v_p)+\bB_{p,1}$ satisfies $\|\mathbf M\|\ge2\delta(1+\|\bY_b(v_p)\|)+2C_A'L_{\eps_1}(v_p)$. The hypothesis of
(A) holds because the right side of Definition~\ref{def:family}(d) is increasing in $\beta$. Let $(\hat u,\hat w)$ be
top singular vectors of $\mathbf M$.
\item By Definition~\ref{def:family}(b), choose $k=k(Q)$ with $|\hat w_{p,k}-\hat w|\le\delta_1=\frac14$.
\item Put $v=\sigma_i(\hat x'_{p,k})$ and $w_0=w_{p,k}$, so that $\bB_{|w_0|}(v)=\bB_{p,k}$.
\end{enumerate}

\begin{lemma}\label{lem:N-holds}
The hypothesis \eqref{eq:N-hyp} of Lemma~\ref{lem:phase}(N) holds at the base $v$, with gap $w_0$, direction
$\hat u$, and $\eps_1$ in place of $\eps$.
\end{lemma}

\begin{proof}
By Lemma~\ref{lem:phase}(A), $|\langle\mathbf M\hat u,\hat w_0\rangle|\ge\delta(1+\|\bY_b(v_p)\|)+C_A'L_{\eps_1}(v_p)$, since
$|\hat w_0-\hat w|\le\frac14$. By
(F1), $\|\bY_b(v)-\bY_b(v_p)\|\le2C_0$, and by Definition~\ref{def:family}(c), $\|\bB_{p,k}-\bB_{p,1}\|\le C_*$. Hence
\[
\big|\big\langle(\bY_b(v)+\bB_{p,k})\hat u,\hat w_0\big\rangle\big|\ge\delta(1+\|\bY_b(v)\|)+C_A'L_{\eps_1}(v_p)-2\delta C_0-2C_0-C_* .
\]
Moreover $L_{\eps_1}(v)\le L_{\eps_1}(v_p)+\log(1+2C_0)\le\big(1+\log(1+2C_0)\big)L_{\eps_1}(v_p)$, since $L_{\eps_1}\ge1$. By
\eqref{eq:CA}, $C_A'L_{\eps_1}(v_p)-3C_0-C_*\ge C_NL_{\eps_1}(v)$.
\end{proof}

Let $y\mapsto u(y)$ be the map of Corollary~\ref{cor:roof-modulus}(b) for the base $v$, defined by
$\eta(e^{X^+(u(y))}(v))=\eta(y)$, so that $u(x_Q)=0$. Put $t_1=\eps_1/\Lambda_b(v)$. By (F1), $\Lambda_b(v)\ge\Lambda_b(x_Q)/(1+C_0)$, so by
Lemma~\ref{lem:cells} and the choice of $\eps_1$,
\begin{equation}\label{eq:inside}
u^{-1}\big(B(0,2t_1)\big)\subset B\big(x_Q,2C_ut_1\big)\subset B\big(x_Q,\tfrac12c_*c_Q\eps/\Lambda_b(x_Q)\big)\subset\tfrac12B_Q .
\end{equation}

\smallskip\noindent\emph{Comparison with \cite{BW20}.} Lemma~\ref{lem:cancel} corresponds to the pair estimates
\cite[Lemmas~3.13--3.15]{BW20}, which treat the pairs of branches with transversal image cones by an oscillatory integral and
the other pairs by the transversality estimate \cite[Proposition~3.6]{BW20}. Here the pair of branches and the direction in a cell are produced by the
specification lemma (Lemma~\ref{lem:spec}), and Theorem~\ref{thm:rigidity} plays the role of the non-cohomology assumption on the roof.

\begin{lemma}[Cancellation]\label{lem:cancel}
Let $(h,H)\in\mathcal C_{\mathsf E}$. For a cell $Q$ as above put $v_1=v_{i,p,k}$, $v_2=v'_{i,p,k}$, $w_j(y)=e^{\phi_N(v_jy)}$ and
\[
Z_j(y)=w_j(y)\,e^{ib\tilde\tau_N(v_jy)}\,h(v_jy)\qquad(j=1,2).
\]
Then there are $j=j(Q)\in\{1,2\}$ and a ball $I_Q\subset\frac12B_Q$ of radius $r_Q\ge\rho_*/(2C_u\Lambda_b(v))$ such that,
with $j'$ the other index,
\[
|Z_1(y)+Z_2(y)|\le(1-\eta_0)\,w_j(y)H(v_jy)+w_{j'}(y)H(v_{j'}y)\qquad(y\in I_Q).
\]
\end{lemma}

\begin{proof}
For $y,y'\in Q$ we have $\Lambda_b(y)d(y,y')\le\eps$. By (K2) and (M5),
$\Lambda_b(v_jy)\,d(v_jy,v_jy')\le\kappa_N\eps\le1$, so the cone condition applies at $v_jy$ and gives
\begin{equation}\label{eq:cancel-drift}
|h(v_jy)-h(v_jy')|\le \mathsf E\vartheta_N\,H(v_jy),\qquad H(v_jy')\le e^{\mathsf E\vartheta_N}H(v_jy).
\end{equation}

\emph{Case 1: $|h(v_jy_0)|<\frac12H(v_jy_0)$ for some $j$ and some $y_0\in Q$.} By \eqref{eq:cancel-drift},
$|h(v_jy)|\le(\frac12+\mathsf E\vartheta_N)e^{\mathsf E\vartheta_N}H(v_jy)\le\frac34H(v_jy)$ on $Q$. Take $I_Q=\frac12B_Q$ (which has the
required radius by \eqref{eq:inside}) and use $|Z_1+Z_2|\le|Z_j|+|Z_{j'}|$.

\emph{Case 2: $|h(v_jy)|\ge\frac12H(v_jy)$ for both $j$ and all $y\in Q$.} Put $a_j=h(v_j\cdot)/|h(v_j\cdot)|$. By
\eqref{eq:cancel-drift}, $|a_j(y)-a_j(y')|\le2|h(v_jy)-h(v_jy')|/|h(v_jy)|\le4\mathsf E\vartheta_N$ for $y,y'\in Q$, so the angle between
$a_j(y)$ and $a_j(y')$ is at most $2\pi \mathsf E\vartheta_N\le\theta_0/32$. We choose the branch of $\arg a_j$ on $Q$ within $\theta_0/32$ of
its value at $x_Q$; this needs no connectedness of $Q$. The relative phase of
$Z_1$ and $Z_2$ is
\[
\Theta(y)=b\,\varphi(y)+\arg a_1(y)-\arg a_2(y),\qquad\varphi=\tilde\tau_N\circ v_1-\tilde\tau_N\circ v_2 .
\]
By Definition~\ref{def:family}(a), $\hat x_{p,k}=f^c(e^{X^-(w_0)}(\hat x'_{p,k}))$ in the frame $v=\sigma_i(\hat x'_{p,k})$. So
Corollary~\ref{cor:roof-modulus}(b), with the branches $v_1,v_2$ in the roles of $v_j,v_k$, gives for $y\in Q$
\[
b\big(\varphi(y)-\varphi(x_Q)\big)=-b\,\theta(u(y))+O(|b|\kappa^Nd(y,x_Q)),
\]
where $\theta=\theta_{v,w_0}$. Since $d(y,x_Q)\le\eps/|b|$, the error is at most $\theta_0/32$.

Let $g(t)=b\,\theta(t\hat u)$ for $0\le t\le t_1$. Then $g(0)=0$, and by Lemmas~\ref{lem:N-holds} and~\ref{lem:phase}(N),
$|g(t_1)|\ge\delta\eps_1|w_0|\ge\delta\eps_1e^{-2L}/C_*=4\theta_0$. By continuity there is $t^*\in(0,t_1)$ with $|g(t^*)|=\theta_0$.
By Lemma~\ref{lem:phase}(R) and the choice of $\rho_*$, $b\,\theta$ varies by at most $\theta_0/8$ between points
$u,u'\in B(0,2t_1)$ with $\Lambda_b(v)|u-u'|\le2\rho_*$. Hence $t^*>2\rho_*/\Lambda_b(v)$. Put $r=\rho_*/(2\Lambda_b(v))$ and
consider the disjoint balls $B(0,r)$ and $B(t^*\hat u,r)$, contained in $B(0,2t_1)$ since $\rho_*\le\eps_1$. Their
preimages under $u$ lie in $\frac12B_Q$ by \eqref{eq:inside} (note $2t_1\le1/\Lambda_b(v)$, as $\eps_1\le\frac12$), and contain balls $I_1$ and $I_2$ of radius
$r/C_u$, centred at $x_Q$ and at $y^*=u^{-1}(t^*\hat u)$.

On $I_1$, the three contributions above give $|\Theta(y)-\Theta(x_Q)|\le\frac{\theta_0}8+\frac{\theta_0}{32}+\frac{\theta_0}{16}\le\frac{\theta_0}4$, and
similarly $|\Theta(y)-\Theta(y^*)|\le\frac{\theta_0}8+\frac{\theta_0}{16}+\frac{\theta_0}{16}=\frac{\theta_0}4$ on $I_2$. Also
$\big||\Theta(y^*)-\Theta(x_Q)|-\theta_0\big|\le\frac{\theta_0}{32}+\frac{\theta_0}{16}<\frac{\theta_0}8$. Hence for $y\in I_1$ and $y'\in I_2$,
\[
\tfrac38\theta_0\le|\Theta(y)-\Theta(y')|\le2\theta_0\le2 .
\]
If both $I_1$ and $I_2$ contained a point where $\mathrm{dist}(\Theta,2\pi\ZZ)<\theta_0/32$, the difference of $\Theta$ at these
points would be within $\theta_0/16$ of $2\pi\ZZ$, which is impossible by the last display. So on one of the
two balls, which we call $I_Q$, the angle between $Z_1$ and $Z_2$ is at least $\theta_0/32$.

Let $j$ be the index for which $w_j(x_Q)H(v_jx_Q)$ is smaller. By (K3), $w_j$ varies on $Q$ by a factor at
most $e^{C_3\eps}$, and by \eqref{eq:cancel-drift} $H(v_j\cdot)$ by a factor at most $e^{\mathsf E\vartheta_N}$. Hence
$w_jH(v_j\cdot)\le2w_{j'}H(v_{j'}\cdot)$ on $Q$, and since $|Z_l|\ge\frac12w_lH(v_l\cdot)$,
$\min(|Z_1|,|Z_2|)\ge\frac14w_jH(v_j\cdot)$ on $Q$. For complex numbers $a,b$ with angle $\phi$,
$|a+b|^2=(|a|+|b|)^2-2|a||b|(1-\cos\phi)$, hence $|a+b|\le|a|+|b|-\frac12(1-\cos\phi)\min(|a|,|b|)$. On $I_Q$ this
gives $|Z_1+Z_2|\le|Z_1|+|Z_2|-\frac18(1-\cos\frac{\theta_0}{32})w_jH(v_j\cdot)$, and the claim follows from
$|Z_l|\le w_lH(v_l\cdot)$.
\end{proof}

\subsection{Dolgopyat operators}

For each cell $Q$, let $I_Q$, $j(Q)$ be as in Lemma~\ref{lem:cancel}, with centre $z_Q$ and radius $r_Q$, and
let $\vartheta_Q(y)=\min\{1,\max\{0,2-2d(y,z_Q)/r_Q\}\}$. It equals $1$ on $\frac12I_Q$, vanishes outside $I_Q$, and is
$2/r_Q$-Lipschitz. Write $v_Q$ for the branch $v_{j(Q)}$ of the cell. Define $\beta$ on $U$ by
\[
\beta(v_Qy)=1-\eta\,\vartheta_Q(y)\quad(y\in Q),\qquad \beta=1\ \text{elsewhere},
\]
and $\mathcal MH=\mathcal L_0^N(\beta H)$. This is well defined: $\vartheta_Q$ is supported in the interior of $Q$, the
branch $v_Q$ is injective, distinct cells have disjoint interiors, and the branches of the family have
images with disjoint interiors. By Lemma~\ref{lem:cells}, \eqref{eq:inside} and (F2), for
$x,x'\in U_i$ with $d(x,x')\le1/\Lambda_b(x)$ and every branch $v$,
\begin{equation}\label{eq:tent-lip}
|\log\beta(vx)-\log\beta(vx')|\le\frac43\,\eta\sup|\nabla\textstyle\sum_Q\vartheta_Q|\;d(x,x')\le\frac{C_T\,\eta}{\rho_*}\,\Lambda_b(x)\,d(x,x') ,
\end{equation}
since the tents met along a geodesic from $x$ to $x'$ belong to cells $Q$ with $\Lambda_b(x_Q)\le C\Lambda_b(x)$.

\smallskip\noindent\emph{Comparison with \cite{BW20}.} Proposition~\ref{prop:dolgopyat} and Corollary~\ref{cor:dolgopyat} correspond to the oscillatory estimate
\cite[Proposition~3.9]{BW20}, which bounds $\|\mathcal L_z^nf\|_{L^1}$ for $n$ of order $\log|b|$. The form differs: we prove an $L^2$ contraction for
Dolgopyat operators \cite{Dol98}, and sup-norm decay follows in Section~\ref{sec:mixing}.

\begin{prop}[Dolgopyat operators]\label{prop:dolgopyat}
For $|b|\ge b_0$ and $(h,H)\in\mathcal C_{\mathsf E}$:
\begin{enumerate}
\item[(1)] $|\mathcal L_b^Nh|\le\mathcal MH$;
\item[(2)] $(\mathcal L_b^Nh,\mathcal MH)\in\mathcal C_{\mathsf E}$;
\item[(3)] $\int(\mathcal MH)^2\,d\mu\le(1-\eta_*)\int H^2\,d\mu$, where $\eta_*=c\,\eta\,c_N(\rho_*/\eps)^{n-1}>0$ and $c_N>0$ is the
minimum of the weights $w_v$ over the branches of the family.
\end{enumerate}
The constants do not depend on $b$, $h$ or $H$.
\end{prop}

\begin{proof}
(1) Let $x$ lie in a cell $Q$. A branch $v$ with $\beta(vx)<1$ must be $v_{Q'}$ for a cell $Q'$ with
$\vartheta_{Q'}(x)>0$, so $Q'=Q$. Split $\mathcal L_b^Nh(x)$ into $Z_1(x)+Z_2(x)$ for the pair of $Q$ and the
remaining terms, each of which is at most $w_v(x)H(vx)=w_v(x)\beta(vx)H(vx)$. If $\vartheta_Q(x)>0$, then $x\in I_Q$,
and Lemma~\ref{lem:cancel} with $\eta\le\eta_0$ gives $|Z_1+Z_2|\le w_j\beta(v_jx)H(v_jx)+w_{j'}H(v_{j'}x)$. If
$\vartheta_Q(x)=0$ this is trivial.

(2) The proof of Proposition~\ref{prop:LY} gives
\[
|\mathcal L_b^Nh(x)-\mathcal L_b^Nh(x')|\le\big(\tfrac{\mathsf E}2+\mathsf E\vartheta_N\big)\,\omega(\rho)\,\mathcal L_0^NH(x).
\]
Since $\mathcal MH\ge(1-\eta)\mathcal L_0^NH\ge\frac34\mathcal L_0^NH$ and $\mathsf E\vartheta_N\le\frac{\mathsf E}4$, this is at most $\mathsf E\,\omega(\rho)\,\mathcal MH(x)$.
For the logarithmic part, the ratio $\mathcal MH(x')/\mathcal MH(x)$ is a weighted average of the ratios
\[
\frac{w_v(x')\,\beta(vx')\,H(vx')}{w_v(x)\,\beta(vx)\,H(vx)} .
\]
By (K3),
\eqref{eq:tent-lip}, the choice of $\eta$ and $\rho\le\omega(\rho)$, the logarithm of each ratio is at most
$C_3d+\frac{\mathsf E}4\omega(\rho)+\mathsf E\vartheta_N\omega(\rho)\le \mathsf E\,\omega(\rho)$, because $C_3\le\frac{\mathsf E}8$ and $d\le\omega(\rho)$.

(3) By the Cauchy--Schwarz inequality and $\sum_vw_v=\mathcal L_01=1$,
\[
(\mathcal MH(x))^2\le\Big(\sum_vw_v(x)\beta(vx)^2\Big)H^\sharp(x),\qquad H^\sharp=\mathcal L_0^N(H^2).
\]
Since $1-(1-a)^2\ge a$ for $a\in[0,1]$,
the first factor is at most $1-c_N\eta\vartheta_Q(x)$ for $x\in Q$. So $(\mathcal MH)^2\le(1-c_N\eta)H^\sharp$ on $\frac12I_Q$ and
$(\mathcal MH)^2\le H^\sharp$ elsewhere. The argument of \eqref{eq:cancel-drift}, applied to every branch $v$, and (K3) show that $H^\sharp$
varies on each cell by a factor at most $e^{C_3\eps+2\mathsf E\vartheta_N}\le2$. By Lemma~\ref{lem:cylinders}(ii) and Lemma~\ref{lem:cells}, since $r_Q\ge c\rho_*/\Lambda_b(x_Q)$ and
$\mathrm{diam}\,Q\le\eps/\Lambda_b(x_Q)$,
\[
\int_{\frac12I_Q}H^\sharp\,d\mu\ge\tfrac12\,\frac{\mu(\frac12I_Q)}{\mu(Q)}\int_QH^\sharp\,d\mu\ge c\Big(\frac{\rho_*}\eps\Big)^{n-1}\int_QH^\sharp\,d\mu .
\]
Summing over the cells and using $\int H^\sharp\,d\mu=\int H^2\,d\mu$, which follows from $\mathcal L_0^*\mu=\mu$, gives (3).
\end{proof}

\begin{cor}[Dolgopyat estimate]\label{cor:dolgopyat}
There are $N$, $\eta_*>0$ and $b_0$ such that for $|b|\ge b_0$, every $h$ with $\|h\|_\infty\le1$ and
$\mathrm{Lip}(h)\le \mathsf E|b|$ on each $U_i$, and every $n\ge0$,
\[
\int_U|\mathcal L_b^{nN}h|^2\,d\mu\le(1-\eta_*)^n .
\]
\end{cor}

\begin{proof}
By Remark~\ref{rem:LY-next}, $(h,1)\in\mathcal C_{\mathsf E}$. Put $h_0=h$, $H_0=1$, and given $(h_n,H_n)\in\mathcal C_{\mathsf E}$ let
$\mathcal M_n$ be the operator of Proposition~\ref{prop:dolgopyat} for $(h_n,H_n)$ (the cells and $\beta$ depend
on $h_n,H_n$ only through the choices of $j(Q)$ and $I_Q$). Set $h_{n+1}=\mathcal L_b^Nh_n$ and
$H_{n+1}=\mathcal M_nH_n$. Then $(h_{n+1},H_{n+1})\in\mathcal C_{\mathsf E}$, $|h_{n+1}|\le H_{n+1}$ and
$\int H_{n+1}^2\,d\mu\le(1-\eta_*)\int H_n^2\,d\mu$. Since $\mu$ is a probability measure,
$\int|h_n|^2\,d\mu\le\int H_n^2\,d\mu\le(1-\eta_*)^n$.
\end{proof}

\section{Proof of Theorem~\ref{thm:main}}\label{sec:mixing}

This section deduces Theorem~\ref{thm:main} from Proposition~\ref{prop:localfreq}. We extend the estimate to complex $s$, pass to sup-norm decay, treat small frequencies using the non-arithmetic length spectrum, invert the Laplace transform of the correlation function for the suspension semiflow, and return to $f^t_\psi$. The scheme is that of \cite[Section~3 and the end of Section~2.1]{BW20}: estimates for the twisted transfer operators $\mathcal L_s$ of the
expanding semiflow, a Laplace transform argument for the semiflow, and a return to the flow through the disintegration of the measure along
stable fibres, with the stable leaf term exponentially small (Step~1 of the proof of Theorem~\ref{thm:main} below, compare \cite[(5)--(7)]{BW20}).
The stable bundle is not $C^1$ here, so the dependence of the conditional measures on the base point is taken from the symbolic model
(Step~2 below, and Lemma~\ref{lem:cond-meas} for equilibrium states).

If $\psi$ is of contact type, then $f^t_\psi$ is a smooth contact Anosov flow preserving $\nu$
(Proposition~\ref{cor:contact}), and the claim is Liverani's theorem \cite{Liv04}. From now on $\psi$ is not of
contact type. For $s=\sigma+ib$ put $\mathcal L_sh=\mathcal L_0(e^{-s\tilde\tau}h)$, so that
$\mathcal L_s^nh(x)=\sum_ve^{\phi_n(vx)-s\tilde\tau_n(vx)}h(vx)$. Functions on $U$ are understood piecewise on the
$U_i$. If $d\le1/\Lambda_b(x)$ and $\rho=\Lambda_b(x)d$, then
\begin{equation}\label{eq:loglip-omega}
d\log\frac ed=\frac{\omega(\rho)}{\Lambda_b(x)}+\frac{\rho\log\Lambda_b(x)}{\Lambda_b(x)}\le2\,\omega(\rho),
\end{equation}
because $\Lambda_b\ge1$, $\log\Lambda_b\le\Lambda_b$ and $\rho\le\omega(\rho)$. So log-Lipschitz functions are controlled by
the cone of Definition~\ref{def:cone}.

\smallskip\noindent\textbf{Function spaces.} Let $\Sigma^+$ be the one-sided subshift of finite type of $\mathsf A$. The map
sending $(a_k)_{k\ge0}\in\Sigma^+$ to the point $x\in U_{a_0}$ with $T^kx\in U_{a_k}$ for all $k$ is continuous and onto, and it is
injective off the $\mu$-null set of points whose forward orbit meets the boundaries of the $U_j$. Through
it we regard functions on $U$ as functions on $\Sigma^+$, $T$ as the shift and $\mathcal L_s$ as the transfer operator
of the shift for the potential $\phi-s\tilde\tau$. This does not affect the cone $\mathcal C_{\mathsf E}$, which only compares
points of a common $U_i$. For $x\ne x'$ in $\Sigma^+$ let $k(x,x')$ be the length of their common prefix, fix
$\theta_1=\kappa^{1/2}$, and let $\mathcal F$ be the space of functions with
$\|h\|_{\mathcal F}=\|h\|_\infty+\sup_{x\ne x'}|h(x)-h(x')|\,\theta_1^{-k(x,x')}<\infty$. By (M5), points of a common cylinder of
length $k$ are at distance at most $C\kappa^k$, so a function that is $\frac12$-H\"older on each $U_i$ belongs to $\mathcal F$,
with $\|h\|_{\mathcal F}\le C\|h\|_{C^{1/2}}$. The potential $\phi$ is $C^1$ on cylinders of length one, so $\phi\in\mathcal F$. By
the Ruelle--Perron--Frobenius theorem \cite[Thm.~2.2]{PP90}, applied to $T^{d_0}$ on each of the $d_0$ cyclic classes of $\mathsf A$
\cite[Ch.~1]{PP90}, $\mathcal L_0$ is quasi-compact on $\mathcal F$. Its peripheral
spectrum consists of the $d_0$-th roots of unity, which are simple eigenvalues with eigenfunctions
constant on the cyclic classes of $\mathsf A$. Hence there are $C$ and $\theta_2<1$ with
\begin{equation}\label{eq:RPF}
\|\mathcal L_0^kg\|_\infty\le C\int g\,d\mu+C\theta_2^k\|g\|_{\mathcal F}\qquad(g\ge0,\ k\ge0).
\end{equation}
By Corollary~\ref{cor:roof-modulus}(a) and $\|\bY_{(d)}\|\le C_{\bF}T(d)$, we have
$|\tilde\tau_n(vx)-\tilde\tau_n(vx')|\le Cd\log(e/d)\le Cd^{1/2}$ for $x,x'\in U_i$ with $d=d(x,x')\le d_0$, and, by chaining, a
bounded difference for all $x,x'\in U_i$. Hence the functions $\tilde\tau_n\circ v$ have $\mathcal F$-seminorm bounded
uniformly in $n$ and in the branch $v$.

\begin{prop}\label{prop:complex-s}
There is $\sigma_0>0$ such that, for $|\sigma|\le\sigma_0$ and $|b|\ge b_0$, Lemma~\ref{lem:cancel},
Proposition~\ref{prop:dolgopyat} and Corollary~\ref{cor:dolgopyat} hold for $\mathcal L_s$ in place of $\mathcal L_b$, with the
operators $\mathcal M_\sigma H=\mathcal L_0^N(e^{-\sigma\tilde\tau_N}\beta H)$ and with $\eta_*/2$ in place of $\eta_*$. Moreover the
functions $H_n$ of the proof of Corollary~\ref{cor:dolgopyat} satisfy $\|H_n\|_\infty\le e^{|\sigma|N\|\tilde\tau\|_\infty n}$.
\end{prop}

\begin{proof}
Replace the weights $w_v(x)=e^{\phi_N(vx)}$ by $w_v^\sigma(x)=e^{\phi_N(vx)-\sigma\tilde\tau_N(vx)}$ in
Sections~\ref{sec:markov} and~\ref{sec:L2}. They enter in three places.

\emph{(i) Cone invariance.} Here $\mathsf E$ is the constant fixed in Section~\ref{sec:L2}, which was chosen to
allow for this step. (K3) is replaced by
$|\log w_v^\sigma(x')-\log w_v^\sigma(x)|\le C_3d+|\sigma||\tilde\tau_N(vx')-\tilde\tau_N(vx)|\le C_3d+\sigma_0C_1\omega(\rho)/|b|\le(C_3+1)\omega(\rho)$, by
(K1) and $d\le\omega(\rho)$, once $|b|\ge\sigma_0C_1$. So Proposition~\ref{prop:LY} holds for the pair
$(\mathcal L_s^Nh,\sum_vw_v^\sigma H(v\cdot))$ after enlarging $\mathsf E_0$ by an absolute amount.

\emph{(ii) Cancellation.} On a cell, $\log w_j^\sigma$ varies by at most $C_3\eps+\sigma_0C_1\omega(\eps)/|b|\le\frac1{10}+\frac1{100}$
for $|b|\ge b_0$, by (K3) and (K1). So the comparison $w_j^\sigma H(v_j\cdot)\le2w_{j'}^\sigma H(v_{j'}\cdot)$ on $Q$ holds as before,
the weights are positive, and the phases are unchanged. The rest of the proof of Lemma~\ref{lem:cancel}
is unchanged.

\emph{(iii) $L^2$ contraction.} In the Cauchy--Schwarz step,
$\sum_vw_ve^{-2\sigma\tilde\tau_N(vx)}\beta(vx)^2\le e^{2|\sigma|N\|\tilde\tau\|_\infty}(1-c_N\eta\vartheta_Q(x))$. Hence
$\int(\mathcal M_\sigma H)^2d\mu\le e^{2\sigma_0N\|\tilde\tau\|_\infty}(1-\eta_*)\int H^2d\mu\le(1-\eta_*/2)\int H^2d\mu$ once $\sigma_0$ is small. The
bound on $\|H_n\|_\infty$ follows from $\mathcal L_01=1$ and $\beta\le1$.
\end{proof}

Corollary~\ref{cor:dolgopyat} and Proposition~\ref{prop:complex-s} prove Proposition~\ref{prop:localfreq}, with $\eta_*/2$ in
place of $\eta_*$: the hypothesis on $h$ there says exactly that $(h,1)\in\mathcal C_{\mathsf E}$.

\begin{prop}[Sup-norm decay]\label{prop:sup-decay}
There are $C$ and $\lambda<1$ such that, for $|\sigma|\le\sigma_0$, $|b|\ge b_0$, $m\ge0$ and every $h$ with
$(h,1)\in\mathcal C_{\mathsf E}$,
\[
\|\mathcal L_s^mh\|_\infty\le C\,|b|^{1/2}\lambda^m .
\]
\end{prop}

\begin{proof}
Put $h_n=\mathcal L_s^{nN}h$, let $H_n$ be as in Proposition~\ref{prop:complex-s}, and put $g=|h_n|^2\le H_n^2$. By the
Cauchy--Schwarz inequality and $\mathcal L_01=1$, $|\mathcal L_s^kh_n|\le e^{|\sigma|k\|\tilde\tau\|_\infty}(\mathcal L_0^kg)^{1/2}$.

We estimate $\|g\|_{\mathcal F}$. If $x,x'$ lie in a common cylinder of length $k$ and $d=d(x,x')\le1/\Lambda_b(x)$,
the cone gives $|g(x)-g(x')|\le2\mathsf E\|H_n\|_\infty^2\omega(\Lambda_b(x)d)\le C\|H_n\|_\infty^2(\Lambda_b(x)d)^{1/2}$. Otherwise
$|g(x)-g(x')|\le2\|H_n\|_\infty^2\le2\|H_n\|_\infty^2(\Lambda_b(x)d)^{1/2}$. Since $d\le C\kappa^k$ by (M5) and
$\Lambda_b\le C|b|\log|b|\le C|b|^2$, we get $\|g\|_{\mathcal F}\le C|b|\,\|H_n\|_\infty^2\le C|b|\,e^{2|\sigma|N\|\tilde\tau\|_\infty n}$.

Take $k=nN$. By \eqref{eq:RPF} and Proposition~\ref{prop:complex-s},
\[
\begin{aligned}
\|\mathcal L_s^{2nN}h\|_\infty^2&\le Ce^{2\sigma_0\|\tilde\tau\|_\infty nN}\Big[\int|h_n|^2d\mu+\theta_2^{nN}\|g\|_{\mathcal F}\Big]\\
&\le Ce^{4\sigma_0\|\tilde\tau\|_\infty nN}\Big[(1-\eta_*/2)^n+|b|\,\theta_2^{nN}\Big]\le C|b|\,\lambda_1^n,
\end{aligned}
\]
with $\lambda_1<1$ once $\sigma_0$ is small. For $m=2nN+r$ with $0\le r<2N$, apply the remaining $r$ steps at the
end, at the cost of a factor $e^{2N\sigma_0\|\tilde\tau\|_\infty}$.
\end{proof}

\begin{lemma}[Non-arithmetic length spectrum]\label{lem:weakmix}
For every $b\ne0$ there is a closed orbit $\gamma$ of $f^t_\psi$ whose period $L_\psi(\gamma)$ satisfies $e^{ibL_\psi(\gamma)}\ne1$.
\end{lemma}

\begin{proof}
Suppose that $e^{ibL_\psi(\gamma)}=1$ for every closed orbit $\gamma$.

\emph{Step 1: a circle-valued time function.} The flow $f^t_\psi$ is a transitive Anosov flow, since it has
the same orbits as the geodesic flow. Fix $T_0>0$ so small that every point has a flow box of time
length $2T_0$; then $d(f^r_\psi y,y)<\eps$ with $0\le r<T_0$ implies $r\le C\eps$. We use the Anosov closing lemma
\cite{Ano67}: there are $\eps_0,C$ such that if $d(f^s_\psi x,x)<\eps\le\eps_0$ with $s\ge T_0$, then there is a closed orbit of
period $L$ with $|L-s|\le C\eps$. Let $x_0$ have a dense orbit, and put $\Psi(f^t_\psi x_0)=e^{ibt}$. Suppose
$d(f^t_\psi x_0,f^s_\psi x_0)<\eps\le\eps_0$ with $s\ge t$. If $s-t\ge T_0$, the closing lemma gives a closed orbit of period $L$
with $|L-(s-t)|\le C\eps$, so $|e^{ibs}-e^{ibt}|=|e^{ib(s-t)}-e^{ibL}|\le C|b|\eps$. If $s-t<T_0$, then $|s-t|\le C\eps$ by the choice of
$T_0$. So $\Psi$ is uniformly continuous on a dense orbit, and it
extends to a continuous map $\Psi:SM\to S^1$ with $\Psi(f^t_\psi x)=e^{ibt}\Psi(x)$ for all $x$ and $t$.

\emph{Step 2: homology.} $\Psi$ defines a class $[\Psi]\in H^1(SM;\ZZ)$. Along a closed orbit $\gamma$, $\Psi$ winds
$bL_\psi(\gamma)/2\pi$ times, so $\langle[\Psi],[\gamma]\rangle=bL_\psi(\gamma)/2\pi$, which is non-zero and has the sign of $b$. Now
$\pi^*:H^1(M;\RR)\to H^1(SM;\RR)$ is an isomorphism (for $n\ge3$ because the fibre $S^{n-1}$ is simply connected,
for $n=2$ because the Euler class of $SM\to M$ is non-zero). Hence $[\Psi]=\pi^*[\zeta]$ in real cohomology, for
a closed $1$-form $\zeta$ on $M$, and $\langle[\Psi],[\gamma]\rangle=\int_{\pi(\gamma)}\zeta$. The geodesic $\pi(\gamma)$ traversed
backwards is the projection of another closed orbit $\gamma'$ of $f^t$, hence of $f^t_\psi$, with
$\langle[\Psi],[\gamma']\rangle=-\int_{\pi(\gamma)}\zeta$. So the two pairings have opposite signs, a contradiction.
\end{proof}

\begin{rem}\normalfont
The same argument shows that $f^t_\psi$ is topologically mixing, and hence weak mixing for $\nu$, but
we only use Lemma~\ref{lem:weakmix} through Proposition~\ref{prop:small-freq}.
\end{rem}

\begin{prop}[Small frequencies]\label{prop:small-freq}
There is $\delta_0>0$ such that for every $b_0>\delta_0$ there is $\sigma_1\in(0,\delta_0/2]$ for which the following hold.
\begin{enumerate}
\item[(i)] For $|\sigma|\le\sigma_1$ and $\delta_0/2\le|b|\le b_0$, $1$ is not in the spectrum of $\mathcal L_s$ on $\mathcal F$, and
$\|(I-\mathcal L_s)^{-1}\|_{\mathcal F}$ is bounded uniformly in $s$. For $\mathrm{Re}\,s>0$ the spectral radius of $\mathcal L_s$ on $\mathcal F$
is less than $1$.
\item[(ii)] On $\{|s|\le\delta_0\}$, $(I-\mathcal L_s)^{-1}=\Pi_s/(1-\lambda(s))+\mathcal R_s$, where $\lambda(s)$ is the eigenvalue of $\mathcal L_s$
near $1$, with $\lambda(s)=1-s\int\tilde\tau\,d\mu+O(s^2)$, $\Pi_s$ is the corresponding spectral projector, and $\mathcal R_s$
is analytic.
\end{enumerate}
\end{prop}

\begin{proof}
Since the functions $\tilde\tau_n\circ v$ have $\mathcal F$-seminorm bounded uniformly in $n$ and $v$, the standard
Lasota--Yorke inequality $\|\mathcal L_s^nh\|_{\mathcal F}\le C(1+|s|)e^{n|\sigma|\|\tilde\tau\|_\infty}\big(\theta_1^n\|h\|_{\mathcal F}+\|h\|_\infty\big)$ holds. So $\mathcal L_s$ is
quasi-compact on $\mathcal F$, with essential spectral radius at most $\theta_1e^{|\sigma|\|\tilde\tau\|_\infty}<1$ for $|\sigma|$ small, and
at most $\theta_1$ for $\sigma\ge0$. For $\sigma>0$, $|\mathcal L_s^nh|\le e^{-\sigma n\min\tilde\tau}\mathcal L_0^n|h|$, so every eigenvalue
has modulus at most $e^{-\sigma\min\tilde\tau}<1$.

Let $b\ne0$, and suppose $\mathcal L_{ib}h=h$ with $0\ne h\in\mathcal F$. Then $|h|\le\mathcal L_0|h|$, and integrating against $\mu$
gives $|h|=\mathcal L_0|h|$. Since $\mathsf A$ is irreducible, $1$ is a simple eigenvalue of $\mathcal L_0$, so $|h|$ is a
positive constant, say $1$. Since $\sum_{Ty=x}e^{\phi(y)}=1$, the equality
$h(x)=\sum_{Ty=x}e^{\phi(y)}e^{-ib\tilde\tau(y)}h(y)$ between unimodular numbers forces $e^{-ib\tilde\tau(y)}h(y)=h(Ty)$ for all $y$.
Iterating, $e^{-ib\tilde\tau_m(y)}h(y)=h(T^my)$ for all $m$, so $e^{ib\tilde\tau_m(y)}=1$ whenever $T^my=y$ in $\Sigma^+$. Now let $\gamma$ be a closed orbit of $f^t_\psi$. The coding map from the suspension of the two-sided shift of
$\mathsf A$ onto $SM$ commutes with the flows, and it is onto and bounded-to-one \cite{Bow73}, say with bound $M$.
The preimage $S$ of a point of $\gamma$ is finite, $|S|\le M$, and the time-$L_\psi(\gamma)$ map of the symbolic flow
is a homeomorphism that maps $S$ into itself, hence permutes it. So some point of $S$ is periodic for the symbolic
flow, with period $jL_\psi(\gamma)$ for some $1\le j\le M$. This gives $z\in\gamma\cap R$ and $m\ge1$ with $P^mz=z$,
where $m$ is the number of returns to $R$ during $j$ periods of $\gamma$, and $y\in\Sigma^+$ with $T^my=y$, coding $\pi z$.
Since $\tilde\tau$ is constant on stable fibres,
\[
\tilde\tau_m(y)=\sum_{k<m}\tilde\tau(P^kz)=\tau_{\psi,m}(z)+\chi(z)-\chi(P^mz)=jL_\psi(\gamma).
\]
(If $\gamma$ does not meet the boundaries of the rectangles, then $j=1$. The general case is needed because
closed orbits may meet the boundaries.) Hence $e^{ibjL_\psi(\gamma)}=1$, and since $j$ divides $M!$,
$e^{i(M!\,b)L_\psi(\gamma)}=1$ for every closed orbit $\gamma$. This contradicts Lemma~\ref{lem:weakmix} for the
frequency $M!\,b\ne0$. So $1$ is not an eigenvalue of $\mathcal L_{ib}$, hence not in its spectrum. Since
$s\mapsto\mathcal L_s$ is analytic in operator norm, the set where $I-\mathcal L_s$ is invertible is open, and the
inverse is continuous there. A compactness argument on $\{\sigma=0,\ \delta_0/2\le|b|\le b_0\}$ gives (i).

Part (ii) is analytic perturbation theory for the simple eigenvalue $1$ of $\mathcal L_0$, with
$\lambda'(0)=-\int\tilde\tau\,d\mu\ne0$. The other peripheral eigenvalues of $\mathcal L_0$ stay away from $1$ for small $|s|$. We take $\delta_0$ so small that this applies on $\{|s|\le\delta_0\}$ and $1-\lambda(s)$ vanishes there only at $s=0$. Since $\sigma_1\le\delta_0/2$, the disc $\{|s|\le\delta_0\}$ contains the strip $\{|\sigma|\le\sigma_1,\ |b|\le\delta_0/2\}$, so (i) and (ii) together cover $\{|\sigma|\le\sigma_1,\ |b|\le b_0\}$.
\end{proof}

\smallskip\noindent\textbf{The semiflow.} Let $\bar X=\{(x,u):x\in U,\ 0\le u<\tilde\tau(x)\}$ with $(x,\tilde\tau(x))\sim(Tx,0)$, let $\bar f^t$
be the suspension semiflow, and $\bar\nu=\mu\otimes du/\int\tilde\tau\,d\mu$. For a function $A$ on $\bar X$ put
\[
\|A\|_\omega=\|A\|_\infty+\sup\frac{|A(x,u)-A(x',u)|}{d(x,x')\log(e/d(x,x'))},
\]
the supremum being over $x\ne x'$ in a common cylinder of length one and
$u<\min\{\tilde\tau(x),\tilde\tau(x')\}$. For $A,B$ with $\int A\,d\bar\nu=0$ let $\varrho(t)=\int A\cdot B\circ \bar f^t\,d\bar\nu$, and put
\[
\hat A_s(x)=\int_0^{\tilde\tau(x)}e^{su}A(x,u)\,du,\qquad \check B_s(y)=\int_0^{\tilde\tau(y)}e^{-sw}B(y,w)\,dw .
\]

\begin{lemma}\label{lem:laplace}
For $\mathrm{Re}\,s>0$,
\[
\Big(\int\tilde\tau\,d\mu\Big)\hat\varrho(s)=R_0(s)+\sum_{n\ge1}\int_U\mathcal L_s^n(\hat A_s)\,\check B_s\,d\mu ,
\]
where $R_0$ is entire and $|R_0(s)|\le C\|A\|_\infty\|B\|_\infty$ for $|\sigma|\le\sigma_0$. Moreover, for $|\sigma|\le\sigma_0$ and
$|b|\ge b_0$, the pair $(\mathcal L_s\hat A_s/(K\|A\|_\omega),1)$ belongs to $\mathcal C_{\mathsf E}$, where $K$ depends only on $\psi$, the
section and $\mathsf E$.
\end{lemma}

\begin{proof}
For $(x,u)\in\bar X$ and $t\ge0$ we have $\bar f^t(x,u)=(T^nx,w)$ with $w=u+t-\tilde\tau_n(x)\in[0,\tilde\tau(T^nx))$.
Substituting $t=w-u+\tilde\tau_n(x)$ in $\int_0^\infty e^{-st}\varrho(t)\,dt$ and using the $T$-invariance of $\mu$, the
term with index $n\ge1$ becomes $\int e^{-s\tilde\tau_n}\hat A_s\cdot\check B_s\circ T^n\,d\mu=\int\mathcal L_s^n(\hat A_s)\check B_s\,d\mu$. The term
$n=0$ is $R_0(s)=\iint_{u\le w}e^{-s(w-u)}A(x,u)B(x,w)\,du\,dw\,d\mu(x)$.

For the cone, let $v$ be an inverse branch of $T$ and let $x,x'\in U_i$ with $d\le1/\Lambda_b(x)$. Then $vx,vx'$
lie in a common cylinder of length one, and $d(vx,vx')\le d$. With $y=vx$, $y'=vx'$ and
$\tau_-=\min\{\tilde\tau(y),\tilde\tau(y')\}$,
\[
\hat A_s(y)-\hat A_s(y')=\int_{\tau_-}^{\tilde\tau(y)}e^{su}A(y,u)\,du-\int_{\tau_-}^{\tilde\tau(y')}e^{su}A(y',u)\,du+\int_0^{\tau_-}e^{su}\big(A(y,u)-A(y',u)\big)\,du .
\]
By (K1) with $N=1$, $|\tilde\tau(y)-\tilde\tau(y')|\le C_1\omega(\rho)/|b|$, which bounds the first two terms, and by
\eqref{eq:loglip-omega} and the monotonicity of $d\log(e/d)$ the last term is at most
$C\|A\|_\omega\omega(\rho)$. The phase $e^{-ib\tilde\tau(v\cdot)}$ and the weight $e^{\phi-\sigma\tilde\tau}$ are handled as in the proof of
Proposition~\ref{prop:LY}. Finally $|\mathcal L_s\hat A_s|\le K_0\|A\|_\infty$, and the oscillation of $\mathcal L_s\hat A_s$ is at most $C\|A\|_\omega\,\omega(\rho)$, with $K_0,C$ depending only on $\psi$ and the section. So $K=\max\{K_0,C/\mathsf E\}$ works.
\end{proof}

\begin{prop}\label{prop:laplace-ext}
Let $\int A\,d\bar\nu=0$. Then $\hat\varrho$ extends analytically to $\{\mathrm{Re}\,s>-\sigma_2\}$, where $\sigma_2=\min\{\sigma_0,\sigma_1\}$,
and there $|\hat\varrho(s)|\le C(1+|b|)^{1/2}\|A\|_\omega\|B\|_\infty$.
\end{prop}

\begin{proof}
For $|b|\ge b_0$, Lemma~\ref{lem:laplace} and Proposition~\ref{prop:sup-decay}, applied to $h=\mathcal L_s\hat A_s/(K\|A\|_\omega)$,
give $|\int\mathcal L_s^n(\hat A_s)\check B_s\,d\mu|\le C|b|^{1/2}\lambda^{n-1}\|A\|_\omega\|B\|_\infty$. These bounds are summable, and each term
is analytic in $s$. For $\delta_0/2\le|b|\le b_0$ and $\mathrm{Re}\,s>0$, $\sum_{n\ge1}\mathcal L_s^n\hat A_s=(I-\mathcal L_s)^{-1}\mathcal L_s\hat A_s$, which by
Proposition~\ref{prop:small-freq}(i) extends analytically to $|\sigma|\le\sigma_1$, with norm at most
$C\|\mathcal L_s\hat A_s\|_{\mathcal F}\le C\|A\|_\omega$; the last bound follows as in Lemma~\ref{lem:laplace}, using
$d\log(e/d)\le Cd^{1/2}$. For $|s|\le\delta_0$ use Proposition~\ref{prop:small-freq}(ii). The singular part is
$\int\Pi_s(\mathcal L_s\hat A_s)\check B_s\,d\mu/(1-\lambda(s))$. Its numerator vanishes at $s=0$, because
$\Pi_0(\mathcal L_0\hat A_0)=\int\hat A_0\,d\mu=(\int\tilde\tau\,d\mu)\int A\,d\bar\nu=0$, and the denominator has a simple zero there. So
the singularity is removable.
\end{proof}

\begin{lemma}[Inversion]\label{lem:inversion}
Let $\varrho\in C^2([0,\infty))$ with $\varrho,\varrho',\varrho''$ bounded. Suppose that $\hat\varrho$ and $\widehat{\varrho''}$ extend
analytically to $\{\mathrm{Re}\,s>-\sigma_2\}$, with bounds $C_\varrho(1+|b|)^{1/2}$ there. Then for every
$\sigma'<\min\{\sigma_2,1\}$ there is $C'$, depending only on $\sigma'$, $\sigma_2$ and $C_\varrho+\|\varrho\|_{C^2}$, such that
$|\varrho(t)|\le C'e^{-\sigma't}$.
\end{lemma}

\begin{proof}
Put $g(t)=\varrho(t)-\varrho(0)e^{-t}-\varrho'(0)te^{-t}$. Integrating by parts twice,
$\hat\varrho(s)=\varrho(0)/s+\varrho'(0)/s^2+\widehat{\varrho''}(s)/s^2$, hence
\[
\hat g(s)=\frac{\widehat{\varrho''}(s)}{s^2}+\frac{\varrho(0)}{s(s+1)}+\frac{\varrho'(0)(2s+1)}{s^2(s+1)^2}.
\]
On the other hand $\hat g=\hat\varrho-\varrho(0)/(s+1)-\varrho'(0)/(s+1)^2$ is analytic in $\{\mathrm{Re}\,s>-\min(\sigma_2,1)\}$. The
first expression gives $|\hat g(s)|=O(|b|^{-3/2})$ uniformly in closed substrips. So the inversion formula
$g(t)=\frac1{2\pi}\int\hat g(\sigma+ib)e^{(\sigma+ib)t}\,db$, valid for $\sigma>0$, converges absolutely, and the line of
integration can be moved to $\sigma=-\sigma'$.
\end{proof}

On $\bar X$, $\varrho''$ is the correlation of $A$ with $\partial_u^2B$. So Proposition~\ref{prop:laplace-ext} and
Lemma~\ref{lem:inversion} give, for $\int A\,d\bar\nu=0$ and $B$ of class $C^2$ along the semiflow,
\begin{equation}\label{eq:semiflow-mixing}
\Big|\int A\cdot B\circ \bar f^t\,d\bar\nu\Big|\le C\,e^{-\sigma't}\,\|A\|_\omega\big(\|B\|_\infty+\|\partial_uB\|_\infty+\|\partial_u^2B\|_\infty\big).
\end{equation}

\begin{proof}[Proof of Theorem~\ref{thm:main} for $\psi$ not of contact type]
Replacing $A$ by $A-\int A\,d\nu$, we may assume $\int A\,d\nu=0$.

\emph{Coding.} Let $X$ be the suspension over $R$ with roof $\tilde\tau$. The map
$\Phi(z,u)=f_\psi^{u-\chi(z)}(z)$ is well defined on $X$, because $\Phi(z,\tilde\tau(z))=f_\psi^{\tau_\psi(z)-\chi(Pz)}(z)=\Phi(Pz,0)$, and
it intertwines the suspension flow with $f^t_\psi$. It is a measurable isomorphism onto
$(SM,f^t_\psi,\nu)$, bijective off a null set \cite{Bow73}, and it maps the measure
$\mu_R\otimes du/\int\tilde\tau\,d\mu$ to $\nu$, since $\mu_R$ is the transverse measure of $\nu$. Collapsing stable
fibres maps $X$ onto $\bar X$ and intertwines the flows. By Lemma~\ref{lem:bridge}(i) and
Lemma~\ref{lem:chart}(iv), the point $f_\psi^{-\chi(z)}(z)$ lies on the strong stable leaf of $f^t_\psi$ through $\pi z$. So
$\Phi(z,u)$ and $\Phi(\pi z,u)$ lie on a common strong stable leaf, at distance at most $C$.

\emph{Step 1: replacing $B$.} Fix $t$, put $k=t/4$, let $\bar\pi$ denote the projection along stable fibres at
fixed height, and put $B_1=(B\circ f^k_\psi\circ\Phi)\circ\bar\pi$. By the last remark,
$\|B\circ f^k_\psi\circ\Phi-B_1\|_\infty\le C\|B\|_{C^1}e^{-ck}$. Moreover $\partial_u^jB_1=(Z_\psi^jB)\circ f^k_\psi\circ\Phi\circ\bar\pi$, so these
derivatives are bounded by $C\|B\|_{C^2}$ for $j\le2$, independently of $k$.

\emph{Step 2: replacing $A$.} Since $B_1\circ \bar f^{t-k}$ is constant on stable fibres,
\[
\int A\cdot B\circ f^t_\psi\,d\nu=\int E_sA\cdot B_1\circ \bar f^{t-k}\,d\bar\nu+O\big(\|A\|_\infty\|B\|_{C^1}e^{-ck}\big),
\]
where $E_sA(x,u)=\int(A\circ\Phi)([x,y],u)\,d\mu_x(y)$ is the conditional expectation along stable fibres, with
the conditional measures $\mu_x$ of Section~\ref{sec:markov}, which depend smoothly on $x$. Let $x,x'$ lie in a
common cylinder of length one, $d=d(x,x')$, and $y\in S_i$. The points $[x,y]$ and $[x',y]$ lie on a common
unstable plaque at distance at most $Cd$. By Lemma~\ref{lem:bridge}(i),
$\chi([x',y])-\chi([x,y])=\big(\mathcal G([x',y])-\mathcal G([x,y])\big)-\big(\mathcal G(x')-\mathcal G(x)\big)$, and the formula for differences of $\mathcal G$
along weak unstable leaves in the proof of Lemma~\ref{lem:bridge}(ii) expresses this as
$\pm\mathcal Q\big(\xi(x),\xi([x,y]);\eta(x),\eta(x')\big)+O(d)$. As in the proof of
Corollary~\ref{cor:roof-modulus}(a), this is at most $Cd(1+\|\bY_{(d)}\|)\le Cd\log(e/d)$. Since the flow is applied for
a bounded time, $|(A\circ\Phi)([x,y],u)-(A\circ\Phi)([x',y],u)|\le C\|A\|_{C^1}d\log(e/d)$. Together with the smooth
dependence of $\mu_x$ on $x$, this gives $\|E_sA\|_\omega\le C\|A\|_{C^1}$. Also $\int E_sA\,d\bar\nu=\int A\,d\nu=0$.

\emph{Step 3.} Apply \eqref{eq:semiflow-mixing} to $E_sA$ and $B_1$ at time $t-k=\frac34t$. With Steps 1 and 2,
\[
\Big|\int A\cdot B\circ f^t_\psi\,d\nu\Big|\le C\big(e^{-ct/4}+e^{-3\sigma't/4}\big)\|A\|_{C^1}\|B\|_{C^2}.\qedhere
\]
\end{proof}

\section{Equilibrium states of H\"older potentials}\label{sec:gibbs}

\smallskip
This section extends Theorem~\ref{thm:main} for $\psi$ not of contact type from the measure $\nu$ to every equilibrium state of a H\"older potential. The point is that Sections~\ref{sec:temporal}--\ref{sec:spec} and the cancellation argument of Section~\ref{sec:L2} are about the geometry of the flow and use the measure only through the properties $\mathcal L_0 1=1$ and $\mathcal L_0^*\mu=\mu$ of the transfer operator, a distortion bound for the weights $e^{\phi_N\circ v}$, and a lower bound for the measure of balls (Lemma~\ref{lem:gibbs-dist}).

Theorem~\ref{thm:gibbs} is stated in the introduction. For $\varpi$ constant, $\mu_\varpi$ is the measure of maximal entropy of $f^t_\psi$, which in general is different from $\nu$. If $\psi$ is of contact type, the theorem follows from \cite[Thm.~1.1]{Sto25}, applied to $f^t_\psi$, which is a $C^\infty$ Anosov flow preserving the smooth contact form $\alpha_\psi$ (Proposition~\ref{cor:contact}). From now on $\psi$ is not of contact type, and the rest of this section proves the theorem in that case.

\subsection{The symbolic model}

Let $G_0(z)=\int_0^{\tau_\psi(z)}\varpi(f^t_\psi z)\,dt-p_\varpi\,\tau_\psi(z)$ for $z\in R$, where $p_\varpi$ is the pressure of $\varpi$ for $f^t_\psi$. By Bowen--Ruelle \cite{BR75,PP90}, $\mu_\varpi$ is the image of $\mu_R\otimes du/\!\int\tau_\psi\,d\mu_R$ under $\Phi$ (Section~\ref{sec:mixing}), where $\mu_R$ is the push-forward to $R$ of the symbolic equilibrium state of $G_0$ (which has pressure $0$ for the return map $P:R\to R$). By Lemma~\ref{lem:null-bdry} below, the orbits meeting the boundaries of the rectangles form a $\mu_\varpi$-null set, so the coding is injective almost everywhere and $\mu_R$ is the equilibrium state of $G_0$ for $P:R\to R$. Since $\varpi$ is $\beta_\varpi$-H\"older and $\tau_\psi$ is Lipschitz, $|G_0(P^nz)-G_0(P^n\pi z)|\le C\kappa^{\beta_\varpi n}$ by (M3), so
\[
\chi_\varpi=\sum_{n\ge0}\big[G_0(P^n\pi z)-G_0(P^nz)\big],\qquad \tilde G=G_0+\chi_\varpi-\chi_\varpi\circ P,
\]
converge, and $\tilde G$ is constant on stable fibres, hence a function on $U$, H\"older for $d$ with some exponent $\beta'\in(0,\beta_\varpi]$. By the Ruelle--Perron--Frobenius theorem there is a positive H\"older $h_0$ such that $\phi=\tilde G+\log h_0-\log h_0\circ T$ satisfies $\mathcal L_01=1$ (the eigenvalue of $\mathcal L_{\tilde G}$ is $e^0=1$, because the pressure of $G_0$ is $0$), where $\mathcal L_0h(x)=\sum_{Ty=x}e^{\phi(y)}h(y)$, and $\mu=\pi_*\mu_R$ satisfies $\mathcal L_0^*\mu=\mu$. The operators are $\mathcal L_s=\mathcal L_0(e^{-s\tilde\tau}\,\cdot\,)$ as before. For $\nu$ we had $\phi=\log\varrho-\log\varrho\circ T-\log|\det DT|$, which is a special case.

\begin{lemma}\label{lem:null-bdry}
Let $\mu_\varpi$ be the equilibrium state of a H\"older potential $\varpi$ for $f^t_\psi$. The set of points whose $f^t_\psi$-orbit meets $\partial R$ is $\mu_\varpi$-null.
\end{lemma}

\begin{proof}
We use three facts. (a) $\mu_\varpi$ is ergodic and gives positive mass to every nonempty open subset of $SM$ \cite[Thm.~3.3]{BR75}, which applies to the basic set $SM$ of the transitive Anosov flow $f^t_\psi$ and to the H\"older potential $\varpi$. (b) The boundary splits as $\partial R=\partial^sR\cup\partial^uR$, where $\partial^sR=\bigcup_i[U_i,\partial S_i]$ and $\partial^uR=\bigcup_i[\partial U_i,S_i]$ are closed, and by the Markov property (M2) $P(\partial^sR)\subset\partial^sR$ and $P^{-1}(\partial^uR)\subset\partial^uR$ where defined \cite{Bow73,Rat73b}. Property (b) is the boundary form of the Markov property (M2); it is standard and is used implicitly when the flow measure is induced on a Markov section in \cite{Dol98,TZ23}. (c) Each of $\partial^sR,\partial^uR$ has empty interior in $R$, because $U_i,S_i$ are closures of their interiors and by (M1).

By (a), the $P$-invariant measure $\mu_R$ is ergodic (a suspension is ergodic iff its base is), and it gives positive mass to every nonempty relatively open subset of $R$: such a set contains the nonempty open set $\mathrm{int}_{D_i}R_i\cap W$, over which the flow box is open in $SM$. Let $N=\partial^sR$. Then $N\subset P^{-1}N$ and $\mu_R(P^{-1}N)=\mu_R(N)$, so $N$ is $P$-invariant mod $\mu_R$-null sets and $\mu_R(N)\in\{0,1\}$. If $\mu_R(N)=1$, then $N$ is closed of full measure, so $N=R$ by the positivity just proved, contradicting (c). Hence $\mu_R(\partial^sR)=0$, and the same argument with $P^{-1}$ gives $\mu_R(\partial^uR)=0$. The set of points whose orbit meets $\partial R$ is the image under $\Phi$ of $\bigcup_{k\in\mathbb Z}P^{-k}(\partial R)\times\mathbb R$, which is null for $\mu_R\otimes du$ because $P$ preserves $\mu_R$.
\end{proof}

\begin{lemma}\label{lem:gibbs-dist}
There are $C_3$ and $c>0$ such that for all $N\ge1$, all branches $v$ of $T^N$ on $U_i$ and all $x,x'\in\hat U_i$, with $d=d(x,x')$:
\begin{enumerate}
\item[(K3$'$)] $|\phi_N(vx')-\phi_N(vx)|\le C_3\,d^{\beta'}$;
\item[(G)] $\mu(v(E))=\int_Ee^{\phi_N(vx)}\,d\mu(x)$ for Borel $E\subset U_i$, and for every ball $B'\subset Q$ of radius $r$ in a cylinder $Q$ there is $q>0$, depending only on $\phi$ and the section, with
$\mu(B')\ge c\,(r/\mathrm{diam}\,Q)^{q}\,\mu(Q)$.
\end{enumerate}
\end{lemma}

\begin{proof}
(K3$'$) $\phi_N(vx)=\sum_{k<N}\phi(T^kvx)$, and $T^k\circ v$ is an inverse branch of $T^{N-k}$, so $d(T^kvx,T^kvx')\le C\kappa^{N-k}d$ by (M5). Since $T^kvx,T^kvx'$ have the same itinerary of length $N-k\ge1$, they lie in a common cylinder of length one, on which $\phi$ is $\beta'$-H\"older, so $|\phi(T^kvx)-\phi(T^kvx')|\le C\kappa^{\beta'(N-k)}d^{\beta'}$, and the sum over $k$ is bounded.

(G) The first identity is $\mathcal L_0^{N*}\mu=\mu$. For the second, let $x_0\in B'$ have a typical itinerary and let $Q''$ be the first cylinder containing $x_0$ with $\mathrm{diam}\,Q''\le c_1r$, where $c_1$ is such that $Q''\subset B'$. By Lemma~\ref{lem:cylinders}(i), $\mathrm{diam}\,Q''\ge c_*c_1r$, and $Q''$ is a cylinder of $Q$ of relative length $m$ with $C\kappa^m\mathrm{diam}\,Q\ge c_*c_1r$, so $m\le\log(C'\mathrm{diam}\,Q/r)/\log(1/\kappa)$. By the first identity and (K3$'$), $\mu(Q'')\ge c\,e^{-m\|\phi\|_\infty}\mu(Q)$, which gives the claim with $q=\|\phi\|_\infty/\log(1/\kappa)$ (we use $\mu(U_j)>0$ for the finitely many $j$, by full support, and the distortion bound (K3$'$) with $d\le1$ to compare $\mu(Q'')$ with $\mu(Q)$).
\end{proof}

\subsection{The cone with a H\"older modulus}

Fix $\alpha\in(0,\beta']$, to be restricted in the last step. Replace the modulus $\omega(\rho)=\rho\log(e/\rho)$ by $\omega_\alpha(\rho)=\rho^\alpha$ in Definition~\ref{def:cone}, and write $\mathcal C_{\mathsf E}^\alpha$ for the resulting cone. Since $\rho\log(e/\rho)\le C_\alpha\rho^\alpha$, $\rho\le\rho^\alpha$ and $\Lambda_b\ge1$, three inputs of the proof of Proposition~\ref{prop:LY} become
\begin{itemize}
\item[(K1$'$)] $|b||\tilde\tau_N(vx')-\tilde\tau_N(vx)|\le C_1\omega(\rho)\le C_1'\rho^\alpha$;
\item[(K3$''$)] $|\phi_N(vx')-\phi_N(vx)|\le C_3d^{\beta'}\le C_3\rho^\alpha$, because $d=\rho/\Lambda_b(x)\le\rho\le1$;
\item[(K4)] $\omega_\alpha(\rho_v)\le\kappa_N^\alpha\,\omega_\alpha(\rho)$, with $\rho_v\le\kappa_N\rho$ as before.
\end{itemize}
Put $\vartheta_N=\kappa_N^\alpha$ in the proof of Proposition~\ref{prop:LY}. The proof is otherwise verbatim, with $4C_3d$ replaced by $4C_3\rho^\alpha$ and $2C_1$ by $2C_1'$.

\begin{prop}\label{prop:gibbs-cone}
Proposition~\ref{prop:LY}, Lemma~\ref{lem:cancel}, Proposition~\ref{prop:dolgopyat}, Corollary~\ref{cor:dolgopyat}, Proposition~\ref{prop:complex-s} and Proposition~\ref{prop:localfreq} hold for $\phi$ as above and the cone $\mathcal C_{\mathsf E}^\alpha$, with $\mathsf E_0=2(4C_3+2C_1')$ and with the modifications listed in the proof.
\end{prop}

\begin{proof}
\emph{Cone invariance} is the paragraph above, and the symmetry (F2) holds with $C_\Lambda$ replaced by $C_\Lambda^\alpha$.

\emph{Cancellation.} In Step~2 of Section~\ref{sec:L2} take $\eps$ with $C_3\eps^{\alpha}\le\frac1{10}$. On a cell, $d\le\eps/\Lambda_b\le\eps$, so $w_j$ varies by a factor at most $e^{C_3\eps^{\beta'}}\le e^{C_3\eps^\alpha}$. In \eqref{eq:cancel-drift}, the cone at $v_jy$ gives $|h(v_jy)-h(v_jy')|\le\mathsf E(\kappa_N\eps)^\alpha H(v_jy)\le\mathsf E\vartheta_NH(v_jy)$ and $H(v_jy')\le e^{\mathsf E\vartheta_N}H(v_jy)$, as before. Nothing else in Lemma~\ref{lem:cancel} involves the weights: the phase is $b\tilde\tau_N$ and Lemma~\ref{lem:phase}, Corollary~\ref{cor:roof-modulus} are about $\tilde\tau$ only.

\emph{Dolgopyat operators.} The tent estimate \eqref{eq:tent-lip} gives $|\log\beta(vx)-\log\beta(vx')|\le(C_T\eta/\rho_*)\rho\le(C_T\eta/\rho_*)\rho^\alpha$, so the choice $\eta\le\mathsf E\rho_*/(4C_T)$ still gives the bound $\frac{\mathsf E}4\omega_\alpha(\rho)$ in part~(2), and $C_3d^{\beta'}\le\frac{\mathsf E}8\rho^\alpha$. In part~(3) the factor $(\rho_*/\eps)^{n-1}$ becomes $(\rho_*/\eps)^q$ by Lemma~\ref{lem:gibbs-dist}(G), and $\int H^\sharp d\mu=\int H^2d\mu$ uses only $\mathcal L_0^*\mu=\mu$.

\emph{The classical class.} A function with $\|h\|_\infty\le1$ and $\mathrm{Lip}(h)\le\mathsf E|b|$ satisfies $|h(x)-h(x')|\le\mathsf E|b|d\le\mathsf E\rho\le\mathsf E\rho^\alpha$, so Remark~\ref{rem:LY-next} holds, and more generally every $h$ with $|h(x)-h(x')|\le\mathsf E(\Lambda_bd)^\alpha$ is admitted.

\emph{Complex $s$.} In Proposition~\ref{prop:complex-s}(i), $|\sigma||\tilde\tau_N(vx')-\tilde\tau_N(vx)|\le\sigma_0C_1'\rho^\alpha/|b|$, and (ii), (iii) are unchanged.
\end{proof}

\subsection{Sup-norm decay, conditional measures and the end of the proof}

Let $\mathcal F$ be defined as in Section~\ref{sec:mixing} with $\theta_1=\kappa^{\alpha/2}$. It contains $\phi$ (H\"older of exponent $\beta'\ge\alpha$) and the functions $\tilde\tau_n\circ v$. In the proof of Proposition~\ref{prop:sup-decay} (where $g=|h_n|^2$ is an auxiliary function, not the potential), $|g(x)-g(x')|\le C\|H_n\|_\infty^2\rho^\alpha$ with $\rho\le\Lambda_bC\kappa^k\le C|b|^2\kappa^k$, so $\|g\|_{\mathcal F}\le C|b|^{2\alpha}\|H_n\|_\infty^2$ and
\[
\|\mathcal L_s^mh\|_\infty\le C|b|^{\alpha}\lambda^m\qquad\big((h,1)\in\mathcal C^\alpha_{\mathsf E}\big).
\]
The exponent $\alpha<1$ is all that Lemma~\ref{lem:inversion} needs: if $\hat\varrho$ and $\widehat{\varrho''}$ are bounded by $C(1+|b|)^{\gamma}$ with $\gamma<1$, then $\hat g=O(|b|^{\gamma-2})$ is integrable, and the proof is unchanged. In Lemma~\ref{lem:laplace}, if $A$ is $\alpha$-H\"older on cylinders of length one, $|\hat A_s(y)-\hat A_s(y')|\le C\|A\|_\alpha\rho^\alpha$ directly, and \eqref{eq:loglip-omega} is not needed. Define $\|A\|_\alpha$ accordingly.

\begin{lemma}\label{lem:cond-meas}
For $x\in\hat U_i$ let $\mu_x$ be the conditional measure of $\mu_R$ on the stable fibre $\{[x,y]\}$ of the rectangle $R_i$, regarded as a measure on the stable coordinate $y$. We use the version of $x\mapsto\mu_x$ given by \eqref{eq:cond-formula} below. There is $C_4$ such that for $x,x'\in\hat U_i$ with $d=d(x,x')$, $\mu_{x'}=J\,\mu_x$ with $|\log J|\le C_3d^{\beta'}$ $\mu_x$-a.e. In particular, for bounded Borel $f$,
\[
\Big|\int f\,d\mu_x-\int f\,d\mu_{x'}\Big|\le C_4\,\|f\|_\infty\,d^{\beta'}.
\]
\end{lemma}

\begin{proof}
The stable coordinate $y$ of a point of $R_i$ is determined by its past itinerary $w=(\dots,i_{-2},i_{-1})$, and it does not depend on $x$. For $m\ge1$ let $w_m=(i_{-m},\dots,i_{-1})$ and let $v_{w_m}$ be the branch of $T^m$ with $v_{w_m}(\hat U_i)\subset U_{i_{-m}}$ that corresponds to $w_m$; it is defined on all of $\hat U_i$, so the same $w_m$ is admissible for $x$ and $x'$. Since $\mathcal L_01=1$ and $\mathcal L_0^*\mu=\mu$, the measure $\mu_R$ is the natural extension of $\mu$, and the transition probabilities of the past given the future are
\begin{equation}\label{eq:cond-formula}
\mu_x\big(\{y:\ y\text{ has past }w_m\}\big)=e^{\phi_m(v_{w_m}x)}.
\end{equation}
(Indeed, by Lemma~\ref{lem:gibbs-dist}(G), $\mu(v_{w_m}E)=\int_Ee^{\phi_m(v_{w_m}x)}d\mu(x)$ for $E\subset\hat U_i$, which is the disintegration over $x$ of the $\mu_R$-measure of the set of points in $E\times\{\text{past }w_m\}$.) By (K3$'$), for every $m$ and every $w_m$,
\[
e^{-C_3d^{\beta'}}\le\frac{\mu_{x'}(w_m)}{\mu_x(w_m)}\le e^{C_3d^{\beta'}}.
\]
The sets of points with a given $w_m$ generate the Borel $\sigma$-algebra of the stable coordinate as $m\to\infty$, since their diameters are at most $C\kappa^m$ by (M3). So the ratios converge $\mu_x$-a.e.\ to a density $J$ (martingale convergence), and $|\log J|\le C_3d^{\beta'}$. Then $\int|1-J|\,d\mu_x\le e^{C_3d^{\beta'}}-1\le2C_3d^{\beta'}$ when $C_3d^{\beta'}\le1$, which is the second claim with $C_4=2C_3$, after changing $C_4$ for the finitely many remaining cases.
\end{proof}

\begin{proof}[Proof of Theorem~\ref{thm:gibbs}]
Follow the proof of Theorem~\ref{thm:main} for $\psi$ not of contact type, replacing $\nu$ by $\mu_\varpi$, $\mu_R$ by the equilibrium state above, $\mathcal C_{\mathsf E}$ by $\mathcal C^\alpha_{\mathsf E}$, and the norm $\|A\|_\omega$ by $\|A\|_\alpha$. In Step~2, the difference $E_sA(x,u)-E_sA(x',u)$ is bounded by
\[
\int\big|A\circ\Phi([x,y],u)-A\circ\Phi([x',y],u)\big|\,d\mu_x(y)+\Big|\int A\circ\Phi([x',y],u)\,d(\mu_x-\mu_{x'})(y)\Big|,
\]
whose first term is at most $C\|A\|_{C^{\beta_0}}(d\log(e/d))^{\beta_0}\le C\|A\|_{C^{\beta_0}}d^{\beta_0/2}$ by the argument of Step~2 of that proof, and whose second term is at most $C_4\|A\|_\infty d^{\beta'}$ by Lemma~\ref{lem:cond-meas}. Fix now $\alpha=\min\{\beta',\beta_0/2\}$; the value of $\alpha$ is used only in this section, and $\beta_0$ is the H\"older exponent of the observables in Theorem~\ref{thm:gibbs}. The difference is then at most $C\|A\|_{C^{\beta_0}}d^{\alpha}$. So \eqref{eq:semiflow-mixing} holds with $\|A\|_\alpha$ in place of $\|A\|_\omega$, for $A,B$ with the stated regularity. Finally, for $A,B\in C^{\beta_0}$ replace $B$ by $B_\delta(x)=\int B(f^s_\psi x)\mathsf k_\delta(s)\,ds$ with a smooth mollifier $\mathsf k_\delta$. Then $\|B-B_\delta\|_\infty\le C\delta^{\beta_0}\|B\|_{C^{\beta_0}}$ and $\|\partial_u^jB_\delta\|_\infty\le C\delta^{\beta_0-j}\|B\|_{C^{\beta_0}}$ for $j\le2$, and the same applies to $A$ in the $C^1$ norm. Taking $\delta=e^{-ct}$ gives the exponential bound for H\"older observables, with a smaller rate $\gamma$.
\end{proof}

\begin{rem}\normalfont
The specification lemma (Lemma~\ref{lem:spec}) and Section~\ref{sec:spec} do not involve the measure or the potential: they construct branches from the geometry of the section and the temporal distance. The measure enters the construction of the cells only through two facts. The first is that cells cover $U$ up to a $\mu$-null set (Lemma~\ref{lem:cells}), which holds because $\mu(\partial U_j)=0$ by Lemma~\ref{lem:null-bdry} (indeed $\pi^{-1}(\partial U_j)\subset\partial^uR$). The second is Lemma~\ref{lem:cylinders}(ii), which is replaced by Lemma~\ref{lem:gibbs-dist}(G) with exponent $q$ in place of $n-1$. The weights enter only through $c_N=\min_vw_v$ over the branches of the family of length $N$, and $w_v=e^{\phi_N(vx)}\ge e^{-N\|\phi\|_\infty}>0$, since $\phi=\tilde G+\log h_0-\log h_0\circ T$ is bounded. As $N$ is fixed before $b$ is chosen (Corollary~\ref{cor:dolgopyat}), $c_N$ is a constant.
\end{rem}

\section*{Acknowledgements}
We would like to thank Amie Wilkinson, Carlangelo Liverani and Oliver Butterley for discussions about this result. We also thank Luchezar Stoyanov for useful comments on a previous draft of this paper, and Masato Tsujii and Zhiyuan Zhang, whose work inspired many of the ideas used here. We are also grateful to an anonymous referee for pointing out a mistake in an earlier version of this paper.

\end{document}